\documentclass[a4paper,11pt]{article}
\usepackage{amsmath,amssymb,amsthm,stmaryrd} 
\usepackage{mathtools}
\usepackage{faktor}
\usepackage{float} 
\usepackage{color}
\usepackage{graphics}
\usepackage[small,nohug,heads=littlevee]{diagrams}
\usepackage{enumerate}
\usepackage[dvipsnames]{xcolor}
\usepackage{tikz}
\usetikzlibrary{arrows.meta}
\usetikzlibrary{math}
\pgfdeclarelayer{bg}    
\pgfsetlayers{bg,main} 

\usepackage[normalem]{ulem}   
\numberwithin{equation}{section}

\theoremstyle{definition}

\theoremstyle{theorem}
\newtheorem{theorem}{Theorem}[section]
\newtheorem{lemma}[theorem]{Lemma}

\newtheorem{proposition}[theorem]{Proposition}
\newtheorem{corollary}[theorem]{Corollary}
\theoremstyle{definition}
\newtheorem{definition}[theorem]{Definition}
\newtheorem{remark}[theorem]{Remark}

\renewcommand{\a}{\alpha}
\newcommand{\ova}{{\overline\alpha}}
\renewcommand{\b}{\beta}
\newcommand{\g}{\gamma}
\renewcommand{\d}{\delta}
\newcommand{\z}{\zeta}
\newcommand{\x}{\xi}
\newcommand{\Ups}{\Lambda}

\newcommand{\sub}{\subset}

\newcommand{\R}{\mathbb{R}}
\newcommand{\N}{\mathbb{N}}
\newcommand{\C}{\mathbb{C}}
\newcommand{\Z}{\mathbb{Z}}

\newcommand{\X}{\mathbb{X}}

\newcommand{\MM}{\mathcal{M}}
\newcommand{\NN}{\mathcal{N}}

\newcommand{\FF}{\mathcal{F}}
\newcommand{\PP}{\mathcal{P}}
\newcommand{\TP}{\mathcal{TP}}

\newcommand{\we}{\wedge}

\newcommand{\ot}{\otimes}
\newcommand{\WE}{{\textstyle\bigwedge}}

\newcommand{\hel} {
\hskip2.5pt{\vrule height7pt width.5pt depth0pt}
\hskip-.2pt\vbox{\hrule height.5pt width7pt depth0pt}
\, }
\newcommand{\restr}{\hel}
\newcommand{\pf} {\,\!_\#\,} 

\renewcommand{\j}{\jmath}
\renewcommand{\i}{\imath}

\newcommand{\sm}{\backslash}
\newcommand{\pt}{\partial}
\newcommand{\vhi}{\varphi}
\newcommand{\eps}{\varepsilon}

\newcommand{\oo}{\infty}
\newcommand{\loc}{_{\textnormal{loc}}}

\newcommand{\ov}{\overline}
\newcommand{\wh}{\widehat}
\newcommand{\wt}{\widetilde}
\newcommand{\h}{\mathcal{H}}
\newcommand{\dw}{\downarrow}
\newcommand{\up}{\uparrow}
\newcommand{\cd}{\cdot}
\newcommand{\longto}{\longrightarrow}
\newcommand{\Id}{\text{Id}}
\renewcommand{\t}{\times}

\newcommand\un{\bs{1}}
\newcommand{\Om}{\Omega}
\newcommand{\om}{\omega}
\newcommand{\ds}{\displaystyle}

\newcommand{\lb}{\llbracket}
\newcommand{\rb}{\rrbracket}
\newcommand{\be}{\begin{equation}}
\newcommand{\ee}{\end{equation}}

\newcommand{\XXint}[3]{{\setbox0=\hbox{$#1{#2#3}{\int}$}
      \vcenter{\hbox{$#2#3$}}\kern-.5\we0}}

\newcommand{\void}{\varnothing}

\newcommand{\st}{\stackrel}
\newcommand{\lt}{\left}
\newcommand{\rt}{\right}
\newcommand{\bs}{\boldsymbol}
\newcommand{\M}{\mathbb{M}}
\newcommand{\F}{\mathbb{F}}
\newcommand{\Mt}{\mathbb{M}^{\text{\tiny$\times$}}}
\newcommand{\Ft}{\mathbb{F}^{\text{\tiny$\times$}}}
\newcommand{\Nt}{\mathbb{N}^{\text{\tiny$\times$}}}
\newcommand{\Fw}{\mathbb{F}^{\text{\tiny$\we$}}}
\newcommand{\Mw}{\mathbb{M}^{\text{\tiny$\we$}}}
\newcommand{\Nw}{\mathbb{N}^{\text{\tiny$\we$}}}
\newcommand{\chiw}{\chi^{\text{\tiny$\we$}}}
\newcommand{\psiw}{\psi^{\text{\tiny$\we$}}}
\newcommand{\et}{\,\!_*\,}

\newcommand{\E}{\mathcal{E}}
\newcommand{\GG}{\mathcal{G}}
\newcommand{\HH}{\mathcal{H}}

\DeclareMathOperator{\frt}{frt}
\DeclareMathOperator{\Lip}{Lip}

\DeclareMathOperator{\supp}{supp}

\DeclareMathOperator{\Span}{span}

\DeclareMathOperator{\Sl}{Sl}

\DeclareMathOperator{\bndry}{fr}

\DeclareMathOperator{\Ima}{Im}

\newcommand{\comfig}[1]{#1}

\title{Tensor rectifiable $G$-flat chains}
\author{M. Goldman\footnote{ CMAP, CNRS, \'Ecole polytechnique, Institut Polytechnique de Paris, 91120 Palaiseau,
France, email: michael.goldman@cnrs.fr} \and B. Merlet\footnote{Univ. Lille, CNRS, UMR 8524, Inria - Laboratoire Paul Painlevé, F-59000 Lille, email: benoit.merlet@univ-lille.fr}}

\begin{document}
\maketitle

\begin{abstract}
A rigidity result for normal rectifiable $k$-chains in $\R^n$ with coefficients in an Abelian normed group is established. Given some decompositions $k=k_1+k_2$, $n=n_1+n_2$ and some rectifiable $k$-chain $A$ in $\R^n$, we consider the properties:\\ 
(1)  The tangent planes to $\mu_A$ split as $T_x\mu_A=L^1(x)\times L^2(x)$ for some $k_1$-plane $L^1(x)\sub\R^{n_1}$ and some $k_2$-plane $L^2(x)\sub\R^{n_2}$.\\
(2)  $A=A\restr\Sigma^1\t\Sigma^2$ for some sets $\Sigma^1\sub\R^{n_1}$,  $\Sigma^2\sub\R^{n_2}$ such that $\Sigma^1$ is $k_1$-rectifiable and $\Sigma^2$ is $k_2$-rectifiable (we say that $A$ is $(k_1,k_2)$-rectifiable).\\
The main result is that for \emph{normal} chains, (1) implies (2), the converse is immediate. In the proof we introduce the new groups of \emph{tensor flat chains} (or $(k_1,k_2)$-chains) in $\R^{n_1}\t\R^{n_2}$ which generalize Fleming's $G$-flat chains. The other main tool is White's rectifiable slices theorem. 
We show that  on the one hand  any normal rectifiable chain satisfying~(1) identifies with a normal rectifiable $(k_1,k_2)$-chain and that on the other hand any normal rectifiable $(k_1,k_2)$-chain is $(k_1,k_2)$-rectifiable.
 
\end{abstract}


\section{Introduction}

 Let  $n,n_1,n_2,k,k_1,k_2$ be nonnegative integers such that 
 \[
0\le k_1\le n_1,\quad0\le k_2\le n_2,\quad  k_1+k_2=k,\quad n_1+n_2=n.
 \] 
 We consider a $k$-rectifiable finite positive measures $\mu$ in $\R^n$. Writing $\mu=\rho\h^k\restr\Sigma$ where $\rho:\R^n\to\R_+$ is Borel measurable and $\Sigma\sub\R^n$ is a $k$-rectifiable set, $\Sigma$ admits an approximate tangent $k$-plane $\mu$-almost everywhere. Denoting by  $T_x \mu$ this approximate tangent $k$-plane we have
\[
\forall\,\vhi\in C_c(\R^n)\qquad\dfrac1{r^k}\int\vhi\lt(\dfrac1r(y-x)\rt)\,d\mu(y)\ \st{r\dw\oo}\longto\ \rho(x)\int_{T_x\mu}\vhi\,d\h^k.
\] 
 
 \begin{definition}\label{def:murectif}
We say that $\mu$ is $(k_1,k_2)$-rectifiable if $\mu$ is $k$-rectifiable and $\mu=\mu\restr\Sigma^1\t\Sigma^2$ for some $k_l$-rectifiable sets $\Sigma^l\sub\R^{n_l},\ l=1,2$.
Equivalently, $\mu(\R^n\sm\Sigma^1\t\Sigma^2)=0$.
\end{definition}
If $\mu$ is $(k_1,k_2)$-rectifiable then at $\mu$-almost every point $x\in\R^n$ the approximate tangent plane $T_x\mu$ decomposes, with the notation of the definition, as  
 \be\label{obvious0}
T_x\mu=T_x(\Sigma^1\t\Sigma^2)=T_{x^1}\Sigma^1\t T_{x^2}\Sigma^2,
\ee
where we write $x=(x^1,x^2)$ with $x^l\in\R^{n_l}$ for $l=1,2$.
\begin{definition}\label{def:musplit}
We say that  $T\mu$  is $(k_1,k_2)$-split if for $\mu$-almost every $x$ the tangent plane $T_x\mu$ writes as $T_x\mu=L^1(x)\t L^2(x)$  for some $k_l$-planes $L^l(x)\sub\R^{n_l},\ l=1,2$.
\end{definition}
Notice that the two definitions depend on $(n_1,n_2)$ and not only on $(k_1,k_2)$. We should write for instance ``$T\mu$ is $(k_1,k_2)$-split with respect to the decomposition $(n_1,n_2)$'' but to lighten notation we avoid explicit references to $n_1$, $n_2$.\medskip

With this vocabulary,~\eqref{obvious0} rewrites as follows. For any $k$-rectifiable measure $\mu$ on $\R^n$,
 \be\label{obvious}
\mu\text{ is }(k_1,k_2)\text{-rectifiable}
 \implies\ T\mu\text{ is }(k_1,k_2)\text{-split.}
\ee

The main result of the article is a partial converse of~\eqref{obvious}.
 
 \begin{theorem}\label{thm_main}
Let $G$ be a complete Abelian normed group and let $A$ be a normal and rectifiable $k$-chain with coefficients in $G$. Denoting $\mu_A$ the measure defined by $\mu_A(S)=\M(A\restr S)$ for $S$ Borel subset of $\R^n$, the following are equivalent.
\begin{enumerate}[(i)]
\item $T\mu_A$  is $(k_1,k_2)$-split.
\item $\mu_A$ is $(k_1,k_2)$-rectifiable.
\end{enumerate}
The latter is equivalent to $A=A\restr\Sigma^1\times\Sigma^2$ for some $k_l$-rectifiable sets $\Sigma^l\sub\R^{n_l}$, $l=1,2$. We also say that $A$ is $(k_1,k_2)$-rectifiable.
\end{theorem}
This result is a \emph{first order} rigidity property in the sense that a local information about the tangent planes to $\mu_A$ yields a global behavior of the measure $\mu_A$, compared for instance to Alexandrov's rigidity theorem~\cite{Alexandrov62} where an information about \emph{second order} derivatives (the compact manifold has constant mean curvature) yields a global information (it is a sphere).\medskip

The theorem concerns rectifiable chains and not general rectifiable measures. However, given a $k$-rectifiable finite positive Borel measure $\mu$ on $\R^n$, we have $\mu_A=\mu$ for any real flat chain $A=\x  \mu$ where $x\mapsto \x(x)$ is any Borel measurable choice of orientation of the tangent planes $T_x\mu$. The difficulty is that the theorem applies only to chains with finite boundary mass and except in the cases $\mu=0$ and $k=0$, there holds $\M(\pt A)=\oo$ for most (and sometimes all)  of the choices of  $\x$. On the bright side, we can pick $A$ with coefficients in any complete Abelian normed group, which might help for measures with branching structure, see e.g. \cite{MarMas12,MarMul,BonOrlOud16}. Anyway, the hypothesis on $\pt A$ (or at least some weaker condition) is necessary as shown by the counterexample~5 of Section~\ref{Sexples}. Consequently, without further assumption, the fact that for a $k$-rectifiable measure $\mu$, $T\mu$ is $(k_1,k_2)$-split does not imply in general that $\mu$ is $(k_1,k_2)$-rectifiable.\medskip

This article is motivated by former and current works by the authors. More precisely,  in~\cite{GM1} we studied some functional $\E(u)$ that penalize oblique oscillations of a function $u:(0,1)^{n_1}\t(0,1)^{n_2}\to\R$. The functional vanishes on functions $u(x^1,x^2)$ of the form $u(x)=u_1(x^1)$ or $u(x)=u_2(x^2)$.  The sequel~\cite{GM2} studies the $\R^{n_1\t n_2}$-valued distribution $\mu[u]$ obtained by extracting the block $D_{x^1}D_{x^2} u$ from the Hessian matrix $D^2u$. In particular $\mu[u]$ vanishes if and only if $u$ decomposes as $u_1(x^1)+u_2(x^2)$. Theorem \ref{thm_main} is used in \cite{GM2} to establish that when $u$ is bounded and such that $\E(u)<\oo$ then $\mu[u]$ is a $(n_1-1,n_2-1)$-rectifiable measure. As pointed out in \cite{GM2}, this example fits in the general framework treated in \cite[Theorem 1.5]{arroyo} which however would yield a much weaker result namely that the most singular part of $\mu[u]$ is $(n-2)-$rectifiable. \medskip

\subsection*{Outline of the proof of Theorem~\ref{thm_main}}
\label{Ss_outline}
 
First, the direction (ii)$\implies$(i) is the easy part and in fact holds for general rectifiable measures as stated in~\eqref{obvious}.\medskip

Let us give a quick idea of the proof of (i)$\implies$(ii).\\
\noindent
\textit{Step 1. The case $k_1=0$ (Proposition~\ref{prop_main_0k}).} 
 We prove that  if $A$ is a normal rectifiable $k$-chain in $\R^n$ such that $T\mu_A$ is $(0,k)$-split then, there exist sequences $y^1_j\in \R^{n_1}$ and $A^2_j$ normal and rectifiable $k-$chains in $\R^{n_2}$ such that 
\be
\label{intro_k1=0}
A=\sum \lb y^1_j \rb \times A^2_j. 
\ee

By~\cite[Theorem~1.2]{GM_decomp}, given a normal rectifiable $k$-chain $A$ in $\R^n$, there exists a partition $S_j$ of $\R^n$ in Borel sets such that setting $A_j:=A\restr S_j$, the $A_j$'s are normal rectifiable chains and 
\[
\N(A)=\sum\N(A_j).
\]
Moreover this decomposition is maximal in the sense that for each $j$ and for every Borel set $S\sub\R^n$ there holds 
\[
\N(A_j)=\N(A_j\restr S)+\N(A_j\restr(\R^n\sm S))\ \implies\  A_j\restr S=A_j\text{ or }A_j\restr S=0.
\]
 We say that $A_j$ is set-indecomposable. Since we argue separately on each $A_j$, we may assume that $A=A_j$ is set-indecomposable and prove that $A=\lb y\rb\times A^2$.\\ 
 Next, given $i\in\{1,\dots,n_1\}$, we consider the family of half-spaces $H_i(s)=\{x\in\R^n:x_i>s\}$. The fact that $T\mu_A$ is $(0,k)$-split implies that $T_x\mu_A$ and $\bndry H_i(s)=\{x:x_i=s\}$ are parallel so that, roughly speaking, the intersection of $A$ and  $\{x:x_i=s\}$ is not transverse ($A$ does not cross the hyperplane $\{x:x_i=s\}$). With this idea, we establish rigorously that the slice of $A$ by $\{x:x_i=s\}$ vanishes. Equivalently, the hyperplane $\{x:x_i=s\}$ splits $A$ in $A=A\restr H_i(s)+ A\restr(\R^n\sm H_i(s))$ without creating additional boundary, we have:
 \[
 \N(A)=\N(A\restr H_i(s))+\N\lt(A\restr(\R^n\sm H_i(s))\rt)\
 \]
Since $A$ is indecomposable this implies that $A\restr H_i(s)=0$ or $A\restr (\R^n\backslash H_i(s))=0$. Proceeding by dichotomy we obtain that $A=A\restr\{x:x_i=y_i\}$ for some $y_i\in\R$ and using all the directions $e_i$ for $i=1,\dots,n_1$ we deduce that $A$ is supported in $\{y^1\}\t\R^{n_2}$ for some $y^1\in \R^{n_1}$. This concludes the proof of \eqref{intro_k1=0}.
\medskip

\noindent
\textit{Step 2. The  general case (Theorem~\ref{thm_main_bis}).}
If $P\in \PP_k^G(\R^n)$ is a polyhedral chain and $T\mu_P$ is $(k_1,k_2)$-split it is not hard to see that it must be of the form
\[
P=\sum g_ip_i^1\t p_i^2,
\]
where $g_j\in G$ and for $l=1,2$, $p^l_i$ is an oriented polyhedral $k_l$-cell in $\R^{n_l}$. In particular $P$ is in the tensor product $\PP_{k_1,k_2}^\Z(\R^n):=\PP_{k_1}^G(\R^{n_1}) \ot \PP_{k_2}^G(\R^{n_2})$. Moreover, it can be alternatively interpreted as an element of $\FF_{k_1}(\R^{n_1}, \FF_{k_2}^G(\R^{n_2}))$ (the set of $k_1$-flat chains with values in the group $\FF_{k_2}^G(\R^{n_2})$ of $k_2$-chains with values in $G$) through the identification
\[
\i P:=\sum\lt[g_ip_i^2\rt]p_i^1.
\]
In this notation we consider $ g_ip_i^2$ as an element of $\FF_{k_2}^G(\R^{n_2})$.
 One of the main points of the proof of Theorem~\ref{thm_main_bis} is to show that every normal and rectifiable chain $A$ for which $T\mu_A$ is $(k_1,k_2)$-split identifies with an element $\i A$ of $\FF_{k_1}(\R^{n_1}, \FF_{k_2}^G(\R^{n_2}))$. The advantage is that we can now  apply the rectifiable slices theorem  \cite{White1999-2} of White to $\i A$. Indeed, at least formally, $0$-slices of $\i A$ correspond to $(0,k_2)-$slices of $A$ and the fact that $A$ is $(k_1,k_2)$-split implies that the $0$ slices of $\i A$ are $(0,k_2)$-split. Using  {\it Step 1} we can write these slices in the form \eqref{intro_k1=0}. In particular this means that the $0$-slices of $\i A$ should write as
 \[
  \sum A_j^2 \lb y_j^1\rb
 \]
which exactly means that they are $0-$rectifiable. By White's rectifiable slices theorem we then obtain that $\i A=(\i A)\restr\Sigma_1$ for some $k_1$-rectifiable set $\Sigma_1\subset \R^{n_1}$. This yields $A= A\restr\Sigma_1\times \R^{n_2}$. Eventually, exchanging the roles of $n_1$ and $n_2$ concludes the proof.\medskip

In order to make this sketch of proof rigorous, an essential step is to extend the operator $\i$. As usual for flat chains, this is done by continuity. Setting for $P\in \PP_{k_1,k_2}^G(\R^n)$
\begin{equation}\label{eq_isom_intro}
\Fw(P):=\F(\i P),
\end{equation}
we see that $\Fw$ is a norm on $\PP_{k_1,k_2}^G(\R^n)$. The operator $\i$ then naturally extends to the completion of this space for the $\Fw$ norm. It turns out (see Proposition \ref{prop_tensor}) that this space is exactly the same as the space obtained as completion by the projective norm (which coincides with $\Fw$ in this case) of the tensor product $\FF^\Z_{k_1}(\R^{n_1})\ot\FF^G_{k_2}(\R^{n_2})$. For this reason we call this  (new) group, the group of tensor flat chains (or $(k_1,k_2)$-chains). One of the contributions of the paper is to start developing  the theory of these tensor flat chains. The overall theory resemble the one for flat chain. For instance, $\Fw$ has a Whitney type representation (see Proposition \ref{prop_Wwalternative}). In order to clarify the analogies with  the classical theory of flat chain, rather than \eqref{eq_isom_intro}, we actually take  this definition as the  starting point of our construction, see \eqref{WwP}. An important role in the theory is played by the two partial boundary operators $\pt_1$ and $\pt_2$ which for a tensor flat chain $A'=A^1\ot A^2$ are given by
\[
\pt_1A':=(\pt A^1)\ot A^2,\qquad\qquad\pt_2A':=(-1)^{k_1} A^1\ot(\pt A^2).
\]
Because of this extra structure, in the core of the paper, we use to the notation $A^1\we A^2$ rather that $A^1\ot A^2$.

%

Let us point out that another source of difficulties comes from the fact that a  tensor chain does not always identify with a classical chain. There are two reasons for that. The first one is of \emph{geometrical nature}: if $n_1=n_2=1$ and $S$ is the segment with vertices $(0,0)$ and $(1,1)$, the two 1-currents $T_1=e_1\h^1\restr S$ and $T_2=e_2\h^1\restr S$ are $(1,0)$- and $(0,1)$-chains respectively but they are not flat chains (however, their sum is a flat chain). The other reason is of \emph{Functional Analysis} nature: the norm on the groups of tensor chains is weaker than the classical flat norm and the groups of tensor chains are ``larger''. Indeed, the usual flat norm involves only boundary operators of first order (see e.g. Proposition \ref{prop_alternateWG}) and is thus reminiscent of the space $W^{-1,1}(\Om)$.  The tensor flat norm $\Fw$ instead, involves not only $\pt_1$ and $\pt_2$ but also $\pt_1 \pt_2$, see \eqref{WwP}, and is thus closer in spirit to   $W^{-2,1}(\Om)$.\\
Conversely, in general a $k$-chain cannot be interpreted as a tensor chain. For instance if $n_1=n_2=1$, the ``transverse'' $1$-chain associated with the integration over the segment $S$ defined above is not a tensor chain (but writes as the sum of a $(1,0)$-chain and a $(0,1)$-chain).\\

This example shows that we need to be careful when considering the embeddings of subgroups of  $k-$chains  in the group of tensor chains (which is a necessary step in light of the sketch of proof above).   We do this through the operator
\begin{align*}
 \j :\FF^G_k(\R^n)&\longto  \lt(\FF^G_{k'_1,k-k'_1}(\R^n)\rt)_{0\le k'_1\le k},\\
A\quad&\longmapsto\quad (\j_{k'_1,k-k'_1}A)_{k'_1},
\end{align*}
which morally speaking decomposes a flat chain into its $(k_1',k-k_1')$-components.
\medskip

The following properties are essential in the proof of the main results.
\begin{enumerate}[(0)]
\item The operators $\i$ and $\j$ satisfy the expected commutation properties with slicing and restrictions (see Proposition \ref{prop_Sltfc} and Proposition \ref{prop_restr}).
\item[(1)] (Proposition~\ref{prop_pt_and_j}) We have the identities  $\j_{k'_1,k'_2}\pt A=\pt_1\j_{k'_1+1,k'_2}A+\pt_2\j_{k'_1,k'_2+1}A$.
\item[(2)]  (Theorem~\ref{thm_Ups}) Normal $(0,k_2)$-chains identify with the subgroup of normal $k_2$-chains such that 
\[
\forall\, (k'_1,k'_2)\ne (0,k_2)\quad \j_{k'_1,k'_2}A=0.
\]
\item[(3)] (Theorem~\ref{thm_j_121}) $\j$ is one-to-one. 
\item[(4)] (Corollary~\ref{coro_splitting_iff_jk'1k'2=0}) If $A$ is a normal rectifiable $k$-chain in $\R^n$, 
\[
T\mu_A\text{ is }(k_1,k_2)\text{-split}\qquad\iff\qquad\forall\, (k'_1,k'_2)\ne (k_1,k_2)\quad \j_{k'_1,k'_2}A=0.
\]

\end{enumerate}

\subsubsection*{Further remarks}

\noindent
(a) Building on the existing theory of flat chains, the axiomatic definition of tensor chains and the proof of their basic properties is mostly routine though lengthy. On the contrary, and maybe surprisingly, the isomorphism stated in~(2) and the fact~(3) that $\j$ is one-to-one are more delicate. Notice that in the complementary note~\cite{GM_comp} we generalize Theorem~\ref{thm_Ups} to all dimensions: the group of normal $(k_1,k_2)$-chain identifies with a subgroup of the group of normal $k$-chains. We point out that this is a delicate question since the counterpart for chains of finite mass is not true, see \cite{GM_comp}.\medskip

\noindent
(b) In the case $k_1=0$, as in Step~1,  a normal $(0,k)$-rectifiable chain writes as (see Section \ref{Ss_ctrex})
\[
A=\sum \lb y_j^1 \rb\t A_j^2\qquad\text{with }\quad\N(A)=\sum\N(A_j^2),
\]
for some sequence $y_j^1\in\R^{n_1}$ and some sequence $A_j^2$ of normal $k$-rectifiable $G$-flat chains in $\R^{n_2}$. One may wonder whether such property generalizes in higher dimension. Let us restrict ourselves to cycles for simplicity. The question is: given a $(k_1,k_2)$-rectifiable cycle $A$ in $\R^n$, is it possible to find sequences $A_j^l$ for $l=1,2$ with $A_j^l$ a $k_l$-cycle in $\R^{n_l}$ such that
\be\label{contrex_intro}
A=\sum A_j^1\t A_j^2,\qquad\qquad\M(A)=\sum \M(A_j^1)\M(A_j^2).
\ee

As already noticed, the answer is \emph{yes} if $k_1=0$ (or $k_2=0$). In the other limit case $k_1=n_1$  (or $k_2=n_2$) the question is trivial as the only $(n_1,k_2)$-rectifiable cycle is $0$ by the constancy theorem (see~\cite[Lemma~7.2]{Fleming66}) applied to $\i\j_{n_1,k_2}A$. On the contrary, in the intermediate cases $1\le k_1\le n_1-1$, $1\le k_2\le n_2-1$, the answer is \emph{no} in general for any non trivial group. This is established in the final subsection~\ref{Ss_ctrex} by exhibiting a counterexample. We do not know whether the answer is still \emph{no} if we relax  the identity constraint on masses and demand only
\[
\sum \M(A_j^1)\M(A_j^2)\le C \M(A),
\]
for some constant $C$ possibly depending on the $n_l$'s and $k_l$'s.

\subsubsection*{Organization of the paper}

In Section~\ref{Sexples}, we construct a few examples which show both why the question is delicate and to which extent our result is sharp.

In Section~\ref{SGfc}, we recall the theory of $G$-flat chains in $\R^n$ as introduced by Fleming and developed by White. We also give a characterization of the $(k_1,k_2)$-splitting property in terms of slices and state Theorem~\ref{thm_decomp} from~\cite{GM_decomp} about the set-decomposition of normal rectifiable chains.

The groups of tensor flat chains are defined in Section~\ref{Stfc}. We introduce various operations involving these objects (slices, restrictions, the morphisms $\i$ and $\j$) and study their basic properties.
 
In Section~\ref{SId}, we focus on identifications of chains with tensor chains \textit{via} the operator $\j$.  We prove Theorem~\ref{thm_j_121}, Corollary~\ref{coro_splitting_iff_jk'1k'2=0} and Theorem~\ref{thm_Ups}, (the points~(2)(3)\&(4) in the above sketch). 

The main results, namely Proposition~\ref{prop_main_0k}, Theorem~\ref{thm_main2} and Theorem~\ref{thm_main},  are established successively in the last section. In a final subsection, we exhibit some counterexamples for~\eqref{contrex_intro}.

\subsubsection*{Conventions and notation}
\label{S_CPC}
In the article $(G,+,|\cd|_G)$ is a complete Abelian normed group, that is $(G,+)$ is a commutative group, $|\cd|_G:G\to\R_+$ satisfies for $g,g'\in G$,
\[
|g+g'|_G\le |g|_G+|g'|_G,\qquad\qquad|g|_G=0\ \iff g=0,
\]
and denoting $d_G(g,g')=|g'-g|_G$, the metric space $(G,d_G)$ is complete.
\medskip

\noindent
The integers $n_1,n_2\ge0$ and $n\ge1$ are fixed and satisfy $n_1+n_2=n$. The integers $k_1,k_2,k$ are non negative and we always assume $0\le k_1\le n_1$, $0\le k_2\le n_2$ and $k=k_1+k_2$.\medskip

\noindent
The standard basis of $\R^n$ is $(e_1,\dots,e_n)$. \medskip

\noindent
We denote by $\h^k$ the $k$-dimensional Hausdorff measure in $\R^n$. If $S\subset \R^n$, we denote by $\frt S$ its topological boundary. \medskip

\noindent
Unless otherwise specified, the sequences are indexed on $j\ge1$. Most of the time we use the index $i$ for finite sequences. \medskip

\noindent
We denote by $I^n_k$ the set of subsets of $\{1,\dots,n\}$ with $k$ elements. We also denote $I^n:=\bigcup I^n_k$ the collection of subsets of $\{1,\dots,n\}$. We simply write $j$ for the singleton $\{j\}\in I^n_1$ and $\b\sm j$ for $\b\sm\{j\}$ and so on.\medskip

For $\b,\g\in I^n$,  
\begin{enumerate}[($*$)]
\item $\ov\b$ is the complement of $\b$ in $\{1,\dots,n\}$, In particular $\ov\void=\{1,\dots,n\}$ and $\ov j=\ov\void\sm j$;
\item $\X_\b:=\Span\{e_j:j\in\b\}$ and for $x\in\R^n$, $\X_\b(x)$ is the affine space $x+\X_\b$. These spaces are called coordinate spaces;
\item $|\b|$ is the number of elements of $\b$. Setting $m:=|\b|$, we denote by  $\b_1,\dots,\b_m$ the elements of $\b$ arranged in increasing order and we define the $m$-vector,
\begin{equation}\label{ebeta}
e_\b:=e_{\b_1}\we\dots\we e_{\b_m}\,\in{\WE}_m\X_\b.
\end{equation}
Here $\WE_m\X$ denotes the space of $m$-vectors in $\X$.
\end{enumerate}

Given $\b\in I^n$, we denote by $\PP^G_r(\X_\b)\sub\NN^G_r(\X_\b)\sub\MM^G_r(\X_\b)\sub\FF^G_r(\X_\b)$ the groups of polyhedral $r$-chains,  normal $r$-chains, finite mass $r$-chains and flat $r$-chains in $\X_\b$ with coefficients in $G$.  By convention, for $r<0$ and $r>|\b|$, these groups are identified with the trivial group $\{0\}$. Moreover, the boundary operator is denoted $\pt$, the flat norm of a chain $A$ is denoted $\F(A)$ and its mass $\M(A)$. We also set $\N(A):=\M(A)+\M(\pt A)$.\medskip

We have some specific notation for tensor flat chains.  For $0\le k\le n$, we denote
\[
D_k:=\{(k'_1,k'_2): 0\le k'_1\le n_1,\ 0\le k'_2\le n_2,\ k'_1+k'_2=k\}.
\]
We introduce 
\[
\a:=\{1,\dots,n_1\}\, \in I^n,
\]
so that $\ova=\{n_1+1,\dots,n\}$.\medskip

\noindent 
 For $\b\in I^n$ we denote $\b^1:=\b\cap\a$ and $\b^2:=\b\sm\a=\b\cap\ova$.\medskip

\noindent 
For $x\in\R^n$, we denote $x^1$ (resp. $x^2$) the orthogonal projection of $x$ on $\X_\a$ (resp. on $\X_\ova$), so that $x$ decomposes as $x^1+x^2$.
\medskip

\noindent 
The groups of tensor chains are defined in $\R^n$ with respect to the decomposition $\X_\a+\X_\ova$. The groups of tensor chains in $\X_\b$ are always defined with respect to the decomposition $\X_{\b^1}+\X_{\b^2}$. The paper introduces $\PP_{k_1,k_2}^G(\X_\b)\sub\NN^G_{k_1,k_2}(\X_\b)\sub\MM^G_{k_1,k_2}(\X_\b)\sub\FF^G_{k_1,k_2}(\X_\b)$ which are respectively the groups of polyhedral, finite mass, normal and flat $(k_1,k_2)$-chains in $\X_\b$ with coefficients in $G$.\medskip

\noindent 
We sometimes deal with functions in $L^1(\Om,G)$ or $L^1\loc(\Om,G)$ where $\Om$ is a measure space and $(G,+,|\cdot|_{G})$ is a complete Abelian normed group (in fact $\Om$ is always a finite dimensional space of the form $\X_\b(x)$). We use the integration in the sense of Bochner, that is, $f\in L^1(\Om,G)$  if there exist a sequence of measurable simple functions $f_j:\Om\to G$ and a sequence of integrable functions $g_j:\Om\to\R_+$ such that $|f_j-f|_G\le g_j$ almost everywhere in $\Om$ and $\int g_j\to0$. In this case $\int |f|_{G}\,:=\lim \int |f_j|_{G}$.

\section{Examples and counterexamples}
\label{Sexples}

\textit{Example~1}.  Let us start with the obvious. If $A$ is the $1$-chain with multiplicity 1 associated with the graph of a Lipschitz continuous function $f:(0,1)\to\R$, assuming that $T\mu_A$ is $(1,0)$-split in $\R\t\R$ is equivalent to $f'\equiv 0$, hence $f$ is constant. Using the theorem, we only obtain that $f$ takes at most countably many values on a set of full measure and we still have to use the Lipschitz continuity of $f$ to conclude that it is constant. 

Let us build a slightly more complex example, again with $k=n_1=n_2=1$. Let $S_j\sub\R^2$ be a sequence of disjoint oriented segments of length $2^{-j}$ and let us define $A$ as the sum of the $1$-chains associated with the integration along the $S_j$'s. If we assume that $T\mu_A$ is $(1,0)$-split, then for every $j$ the approximate tangent direction to $\mu_A$ is horizontal (except at the end points of the $S_j$'s where it is not defined). We deduce that each $S_j$ is an horizontal segment as above so that, up to a $\h^1$-negligible set, $\cup S_j\sub\R\t\Sigma^2$ with $\Sigma^2\sub\R$ finite or countable. This is exactly the conclusion of the theorem.\bigskip

\noindent
\textit{Example~2}. To understand why  the statement of the theorem is reasonable, let us discuss the case of smooth chains. Let $\MM$  be a $k$-manifold of class $C^1$ in $\R^n$ and assume that $\MM$ is \emph{connected} and \emph{without boundary}. We consider the property:
\be\label{localM1M2}
\forall\,x\in\MM\ \exists\,k'_l(x)\text{-planes }L^l(x)\sub\R^{n_l}\text{ for }l=1,2\text{ such that }T_x\MM=L^1(x)\t L^2(x). 
\ee
This condition is weaker than the $(k_1,k_2)$-splitting property in the sense that the values of $k_1$ and $k_2$ might depend on $x$. 

We claim that if $\MM$ satisfies~\eqref{localM1M2}, then $\MM$ is of the form $\MM^1\t\MM^2$ where for $l=1,2$, $\MM^l$ is a $k_l$-manifold in $\R^{n_l}$ for some $k_l$'s such that $k=k_1+k_2$. Let us prove this claim. Assuming without loss of generality that $0\in\MM$ we set $(k_1,k_2):=(k'_1(0),k'_2(0))$. Up to a change of orthonormal frames in $\X_\a$ and $\X_\ova$ we may assume that
\[
\begin{array}{ll}
V^1:=L^1(0)=\Span\{e_1,\dots,e_{k_1}\},& \quad V^2:=L^2(0)=\Span\{e_{n_1+1},\dots,e_{n_1+k_2}\},
\end{array}
\] 
Let us also denote,
\[
W^1:=\Span\{e_{k_1+1},\dots,e_{n_1}\}\qquad\text{and}\qquad W^2:=\Span\{e_{n_1+k_2+1},\dots,e_n\},
\]
so that $V^1+W^1= \X_\a$ and $V^2+W^2=\X_{\ov \a}$.
In some nonempty open ball centered at $0$, $\MM$ is the graph of a $C^1$ mapping
\[
f:V^1+V^2\longto W^1+W^2,
\] 
such that $f(0)=0$ and $Df(0)=0$. Writing $x=x^1+x^2\in V^1+V^2$ and 
\[
f(x)=f^1(x)+f^2(x)\ \in W^1+W^2,
\]
the linear operator
\[
\Id+Df(x): V^1+V^2\to V^1+W^1+V^2+W^2,
\]
writes as the block matrix,
\[
M(x):=\Id+Df(x) =
\lt(
\begin{array}{cc}
\Id_{k_1}&0\\ 
D_{x^1} f^1(x) &D_{x^2} f^1(x)\\
0& \Id_{k_2}\\
D_{x^1} f^2(x)& D_{x^2} f^2(x)
\end{array}
\rt).
\]
For $x$ in some neighborhood of $0$ in $\MM$, the tangent space to $\MM$ at point $x+f(x)$ is the space $L(x)$ spanned by the columns of $M(x)$. The condition~\eqref{localM1M2} is then equivalent to the following property: 
\be\label{dimL(x)}
\text{For }l=1,2\qquad\dim \lt[L(x)\cap (V^l+W^l)\rt]=k'_l(x).
\ee
We first remark that the left-hand side of these identities are upper semi-continuous and integer valued and since $k'_1(x)+k'_2(x)=k$ is constant, we must have $k_l'(x)=k_l$ in some neighborhood of $0$.\\
Next, for $l=1$, the condition~\eqref{dimL(x)} and the form $(0\ \Id_{k_2})$ of the third block of rows of $M(x)$ enforces that $L(x)\cap (V^1+W^1)$ is the space spanned by the first $k_1$ columns of $M(x)$. Furthermore, all these columns belong to $V^1+W^1$ if and only if $D_{x^1} f^2(x)=0$. Similarly, for $l=2$, we obtain $D_{x^2} f^1(x)=0$. Therefore $f(x)=(f^1(x^1),f^2(x^2))$ in some neighborhood of $0$ and locally $\MM$ is of the form $\MM^1\t\MM^2$. Using the fact that $\MM$ is connected, we deduce that $\MM=\MM^1\t\MM^2$ for some $k_l$-manifolds as described above.\\
As a conclusion, in the case of a connected $C^1$ manifold without boundary, the theorem is elementary. Moreover:
\begin{enumerate}[($*$)]
\item The $(k_1,k_2)$-splitting can be replaced by the weaker $(k'_1(x),k'_2(x))$-splitting assumption~\eqref{localM1M2}.
\item The conclusion can be strengthen by replacing the inclusion $\MM\sub\Sigma^1\t\Sigma^2$ by an equality.
\end{enumerate}
\medskip

\noindent
\textit{Example~3}. If $\MM$ is a manifold with boundary, this local analysis is still valid locally away from the boundary, but in the neighborhood of a boundary point the conditions $D_{x^1} f^2(x)=0$, $D_{x^2} f^1(x)=0$ do not imply that $f(x)$ is of the form $(f^1(x^1),f^2(x^2))$. Let us build a counterexample in $\R^3$ with $n_1=2$ and $n_2=k_1=k_2=1$.\\ 
Let $f\in C^\oo([-1/2,1],\R)$ such that $f\equiv 0$ on $[-1/2,0]$ and $f>0$ on $(0,1]$. For $j\ge0$ we set $f_j:= 2^{-j}f$ and we denote $\GG_j$ the graph of $f_j$. By construction, theses graphs are the same in $[-1,0]\t\R$ but do not meet in $(0,1]\t\R$. Let us define,
\[
\Sigma:=\bigcup_{j\ge0}\GG_j\t \lt[2^{-j-1},2^{-j}\rt]  \ \sub\ \R^3.
\]

\comfig{
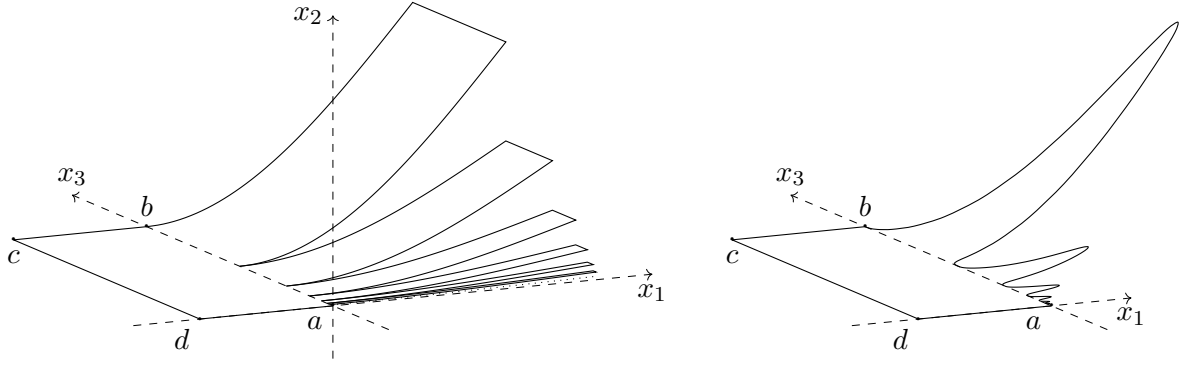
\begin{figure}[ht]   
\centering
\begin{tikzpicture}
\pgfmathsetmacro{\A}{-.7}
\pgfmathsetmacro{\B}{.3}
\pgfmathsetmacro{\C}{.1}
\pgfmathsetmacro{\L}{1}

\begin{scope}[scale=3.5]
\pgfmathsetmacro{\xa}{-.75*\L}
\pgfmathsetmacro{\xb}{1.2*\L}
\draw[very thin,dashed,->] (\xa,\C*\xa) -- (\xb,\C*\xb) node[below]{$x_1$};
\pgfmathsetmacro{\ya}{-.2*\L}
\pgfmathsetmacro{\yb}{1.1*\L}
\draw[very thin,dashed,->] (0,\ya) -- (0,\yb) node[left]{$x_2$};
\pgfmathsetmacro{\za}{-.3*\L}
\pgfmathsetmacro{\zb}{1.4*\L}
\draw[very thin,dashed,->] (\A*\za,\B*\za) -- (\A*\zb,\B*\zb) node[above]{$x_3$};

\draw[thin]  (\A*\L,\B*\L) node{$\cdot$} node[above]{$b$} --
                   (\A*\L-\L/2,\B*\L-\C*\L/2) node{$\cdot$} node[below]{$c$}  --
                   (-\L/2,-\C*\L/2)  node{$\cdot$} node[below left]{$d$} -- 
                   (0,0) node{$\cdot$} node[below left]{$a$};                   

\foreach \i in {0,...,5}
{
	\pgfmathsetmacro{\za}{1/2^\i}
	\pgfmathsetmacro{\zb}{1/2^(\i+1)}
	\draw [thin,smooth]
			 plot[domain={0:\L}] ({\A*\za + \x}, {\C*\x + \B*\za + 1.5*\za*sin(45*\x)^2}) --
			 plot[domain=(\L:0)] ({\A*\zb + \x}, {\C*\x +\B*\zb + 1.5*\za*sin(45*\x)^2});
}

\pgfmathsetmacro{\zb}{1.5/2^7}
\draw [thin,dotted] plot[domain={0:\L}] ({\A*\zb + \x}, {\C*\x +\B*\zb + 1.5*\zb*sin(45*\x)^2});
\end{scope}

\begin{scope}[scale=3.5, xshift=2.7cm]
\pgfmathsetmacro{\xa}{-.75*\L}
\pgfmathsetmacro{\xb}{.3*\L}
\draw[very thin,dashed,->] (\xa,\C*\xa) -- (\xb,\C*\xb) node[below]{$x_1$};
\pgfmathsetmacro{\za}{-.3*\L}
\pgfmathsetmacro{\zb}{1.4*\L}
\draw[very thin,dashed,->] (\A*\za,\B*\za) -- (\A*\zb,\B*\zb) node[above]{$x_3$};

\draw[thin]  (\A*\L,\B*\L) node{$\cdot$} node[above]{$b$} --
                   (\A*\L-\L/2,\B*\L-\C*\L/2) node{$\cdot$} node[below]{$c$}  --
                   (-\L/2,-\C*\L/2)  node{$\cdot$} node[below left]{$d$} -- 
                   (0,0) node{$\cdot$} node[below left]{$a$};    

\foreach \i in {0,...,3}
{
	\pgfmathsetmacro{\za}{1/2^\i}
	\pgfmathsetmacro{\zb}{1/2^(\i+1)}
	\pgfmathsetmacro{\om}{360*2^\i}
	\pgfmathsetmacro{\X}{(2/5)^\i}
	\draw [thin, smooth]	plot[variable=\z, domain={\zb:\za}] 
		({\A*\z - \X*sin(\om*\z)^3}, {\B*\z - \C* \X*sin(\om*\z)^3 + 1.5*\za*sin(45*\X*sin(\om*\z)^3)^2});
}  
\foreach \i in {4,...,5}
{
	\pgfmathsetmacro{\za}{1/2^\i}
	\pgfmathsetmacro{\zb}{1/2^(\i+1)}
	\pgfmathsetmacro{\om}{360*2^\i}
	\pgfmathsetmacro{\X}{(2/5)^\i}
	\draw [very thin, smooth]	plot[variable=\z, domain={\zb:\za}]
		({\A*\z - \X*sin(\om*\z)^3}, {\B*\z - \C* \X*sin(\om*\z)^3 + 1.5*\za*sin(45*\X*sin(\om*\z)^3)^2});
}  

\end{scope}
\end{tikzpicture}
\caption{\label{Fig:1} On the left, the set $\Sigma$, on the right, the $2$-manifold $\MM$ of Example~3.} 
\end{figure}
}

The set $\Sigma$ is of dimension $2$, contains the rectangle $[-1/2,0]\t\{0\}\t[0,1]$ and has Lipschitz regularity away from the cusp points $(0,0, 2^{-j})$ for $j\ge 1$. To obtain a Lipschitz manifold, we cut just below these points in the first coordinate. We set for $j\ge1$,
\[
C_j:=\lt\{(x_1,x_2,x_3) :x_2\in\R,\ 2^{-j}\le|x_3|<2^{1-j},\ x_1 <3^{-j}\sin^2 \lt(2^{-j}\pi x_3\rt)\rt\},
\]
and then $\MM:=\Sigma\cap \bigcup_j C_j$, see Figure~\ref{Fig:1}.

We have built a Lipschitz connected 2-manifold with boundary in $\R^3$. Moreover, $\MM$ is of class $C^1$ away from the four vertices $a$, $b$, $c$, $d$ of the rectangle $[-1/2,0]\t\{0\}\t[0,1]$. 
By construction $\MM\sub\Sigma^1\t\R$ where $\Sigma^1$ is the countable union of the graphs $\GG_j$. This time, we do not have $\MM=\MM^1\t\MM^2$, neither $\MM\sub \MM^1\t\MM^2$ with $\MM^1\sub\R^2$, $\MM^2\sub\R$ connected 1-manifolds.\\
Obviously the tangent planes of $\MM$ satisfy the $(1,1)$-splitting property with $n_1=2$ and $n_2=1$. Moreover, orienting $\MM$ arbitrarily, the $2$-chain $A$ with multiplicity 1 associated with the integration on $\MM$ has  finite boundary mass (the boundary of $\MM$ has finite length) so the theorem applies to $A$. We observe that the conclusion of the theorem is sharp in the sense that we do need a \emph{countable} union of Cartesian products of connected manifolds to cover $\MM$.\bigskip

\noindent
\textit{Example~4}. We now turn to a more complex example. Let $f_1,f_2\in C^\oo(\R,[0,1])$ with 
\[
f_1\equiv0 \text{ on }\R\sm(-2,2),\qquad f_1\equiv 1\text{ on }[-1,1],\qquad f_2\equiv0\text{ on }\R\sm(-1,1),\qquad f_2>0\text{ on }(-1,1).
\]
We define for $x\in\R^2$, $f(x_1,x_2):=f_1(x_1)f_2(x_2)$ and we denote $Q:=([-2,-1]\cup [1,2])\t [-1,1]$. This set is the union of two closed squares and $\pt_{x_1} f\equiv 0$ on $\R^2\sm Q$. Next, let $R\ge10$ and let $x^{(j)}$ be a dense sequence of points in a ball $B$ of $\R^2$ with radius $R$. We set, for $x\in\R^2$, 
\[
F(x):=\sum 8^{-j} f\lt(2^{j}(x-x^{(j)})\rt), 
\]
so that $F$ is of class $C^2$. Denoting 
\[
E:=\bigcup \lt[x^{(j)}+2^{-j} Q\rt],
\] 
there holds, $\h^2(E)\le 8/3$ and $\pt_{x_1}F=0$ in $\R^2\sm E$. However, remark that for every open set $\om\sub\R^2$ intersecting $B\sm E$, there exists a set of real numbers $y^2$ with positive measure such that $F$ is not constant on $\om\cap \{x:x^2=y^2\}$.
Let us define $\GG_F$ as the graph of $F_{|B}$ and
\[
\Sigma:=\lt\{(x,F(x)):x\in B\sm E\rt\} =\GG_f\cap\lt[(\R^2\sm E)\t\R\rt].
\] 
Orienting arbitrarily $\GG_f$, we define $A$ as the $2$-chain with multiplicity $1$ corresponding to the integration over $\Sigma$. This chain has finite positive mass and finite boundary mass and with $n_1=1$ and $n_2=2$, $T\mu_A$ is $(1,1)$-split. The theorem applies and we get that $\Sigma\sub \R\t\Sigma^2$ where, up to negligible set, $\Sigma^2\sub\R^2$ is a countable union of $C^1$-curves.

  The point of this example is that even though, like in Example~3,  $A=A\restr \GG_F$ with $\GG_F$ diffeomorphic to the 2-dimensional ball, we cannot cover $\R^3$ by countably many balls $B_j$ such that $A\restr B_j$ is the integration over $(\MM_j^1\t\MM_j^2)\cap B_j$ for some $C^1$ curves $\MM_j^1\sub\R$, $\MM_j^2\sub\R^2$. 
\bigskip

\noindent
\textit{Example~5}. To stress the necessity of the assumption $\M(\pt A)<\oo$ in Theorem~\ref{thm_main}, we build an example for which the conclusion of the theorem fails.  Here, $k=n_1=n_2=1$. We first introduce the following positive atomic measure on $\R$.
\[
\nu:=\dfrac12\sum_{j\ge0}\, \dfrac1{4^j} \,\sum_{i=1}^{2^j} \d_{i/2^j},
\]
We have $\nu(\R)=1$ and $\nu=\nu\restr D$ where $D$ is the set of dyadic numbers in $(0,1]$. More precisely:
\[
\nu=\dfrac23\sum_{x\in D}\dfrac1{4^{j(x)}} \d_x,
\]
where for $x\in D$, $j(x):=j$ is the nonnegative integer such that $x=i2^{-j}$ for some odd number $0<i<2^j$. 
Next,  for $x_1\in\R$, we set  $f(x_1):=\nu((-\oo,x_1))$. We have that $f'=\nu$ in the sense of distributions, $f$ is left-continuous, increasing on $[0,1]$ and differentiable on $\R\sm D$ with $f'(x_1)=0$ in the classical sense. We denote $\GG_f$ the graph of the restriction of $f$ to $(0,1)$. Remark that the closure of the projection of $\GG_f$ on the vertical axis is a classical Cantor set. Since $f'=0$ almost everywhere, the $1$-chain $A_1$ with support $\ov{\GG_f}$, orientation $e_1$ and multiplicity 1 is rectifiable, see Figure~\ref{Fig:2}.\\
 Now, on the one hand, since $f'\equiv0$ on $\R\sm D$, $T\mu_{A_1}$ is $(1,0)$-split (at every point $\ds\lt(x_1,f(x_1)\rt)$ with $x_1\in (0,1)\sm D$). 
 On the other hand, $f$ being increasing on $(0,1)$, for any countable set $\Sigma^2\sub\R$, the intersection  $\GG_f\cap \R\t\Sigma^2$ is at most countable and therefore $\h^1(\GG_f\cap(\R\t\Sigma^2))=0$. Since $\mu_{A_1}\ll\h^1\restr \GG_f$, we also have $\mu_{A_1}(\R\t\Sigma^2)=0\ne 1=\mu_{A_1}(\R^2)$. This implies that $\mu_{A_1}$ is not $(1,0)$-rectifiable (the conclusion of the theorem is wrong). Here the theorem does not apply because $A$ has infinite boundary mass. This example shows that some control on $\pt A$ is necessary.

\comfig{
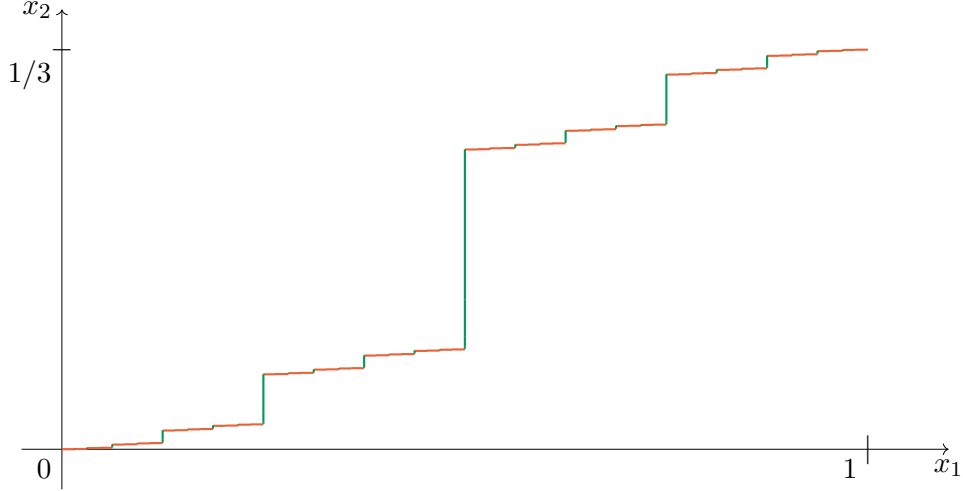
\begin{figure}[ht]   
\centering
\begin{tikzpicture}[scale=5.3]
\draw[very thin,->] (-.1,0) -- (2.2,0) node[below]{$x_1$};
\draw[very thin,->] (0,-.1) -- (0,1.1) node[left]{$x_2$};
\draw (0,0) node[below left]{$0$};
\draw (2,0) node{$\vert$} node[below left]{$1$};
\draw (0,1) node{$-$} node[below left]{$1/3$};

\foreach \i in {0,1}
{
\foreach \j in {0,1}
{
\foreach \k in {0,1}
{
\foreach \l in {0,1}
	{
	\foreach \m in {0,1}
		{	
		\foreach \n in {0,1}
			{
			\pgfmathsetmacro{\f}{\i/2+\j/8+\k/32+\l/128+\m/512+\n/2048}
			\pgfmathsetmacro{\x}{\i+\j/2+\k/4+\l/8+\m/16+\n/32}
			\draw[RedOrange, thick]  (\x, 1.5*\f) -- (\x+1/32, 1.5*\f);
			}
		}
	}
}
}	
}

\begin{scope}
\clip (0,0) rectangle (2,1); 
\foreach \i in {0,1}
{
\pgfmathsetmacro{\fm}{(2*\i+1)/4}
\pgfmathsetmacro{\fp}{(2*\i+2)/4}
\pgfmathsetmacro{\x}{(\i+1) }
\draw[ForestGreen, thick] (\x,1.5*\fm) -- (\x,1.5*\fp);
\foreach \j in {0,1}
{
\pgfmathsetmacro{\fm}{\i/2+(2*\j+1)/16}
\pgfmathsetmacro{\fp}{\i/2+(2*\j+2)/16}
\pgfmathsetmacro{\x}{\i+(\j+1)/2}
\draw[ForestGreen, thick] (\x,1.5*\fm) -- (\x,1.5*\fp);
\foreach \k in {0,1}
{
\pgfmathsetmacro{\fm}{\i/2+\j/8+(2*\k+1)/64}
\pgfmathsetmacro{\fp}{\i/2+ \j/8+(2*\k+2)/64}
\pgfmathsetmacro{\x}{\i+\j/2+(\k+1)/4}
\draw[ForestGreen, thick] (\x,1.5*\fm) -- (\x,1.5*\fp);
\foreach \l in {0,1}
	{
	\pgfmathsetmacro{\fm}{\i/2+\j/8+\k/32+(2*\l+1)/256}
	\pgfmathsetmacro{\fp}{\i/2+ \j/8+\k/32+(2*\l+2)/256}
	\pgfmathsetmacro{\x}{\i+\j/2+\k/4+(\l+1)/8}
	\draw[ForestGreen, thick] (\x,1.5*\fm) -- (\x,1.5*\fp);
	\foreach \m in {0,1}
		{
		\pgfmathsetmacro{\fm}{\i/2+\j/8+\k/32+\l/128+(2*\m+1)/1024}
		\pgfmathsetmacro{\fp}{\i/2+\j/8+\k/32+\l/128+(2*\m+2)/1024}
		\pgfmathsetmacro{\x}{\i+\j/2+\k/4+\l/8+(\m+1)/16}
		\draw[ForestGreen, thick] (\x,1.5*\fm) -- (\x,1.5*\fp);	
		\foreach \n in {0,1}
			{
			\pgfmathsetmacro{\fm}{\i/2+\j/8+\k/32+\l/128+\m/512+\n/2048}
			\pgfmathsetmacro{\x}{\i+\j/2+\k/4+\l/8+\m/16+\n/32}
			\draw[ForestGreen, thick] (\x,1.5*\fm) -- (\x,1.5*\fm+0.001);
			}
		}
	}
}
}	
}
\end{scope}
\end{tikzpicture}
\caption{\label{Fig:2} An approximate representation of the ``horizontal'' chain $A_1$ (in red) and of the ``vertical'' chain $A_2$ (in green) of Examples {5} and {6}.} 
\end{figure}
}

\noindent
\textit{Example~6}. Let us rule out a possible generalization of the theorem suggested by the weaker assumption~\eqref{localM1M2}. We consider the following variant of the splitting condition where the dimensions $k_1$, $k_2$ may vary.
\be\label{wtsc}
\text{for }\mu\textit{ a.e. }x,\ T_x\mu=L^1(x)\t L^2(x)\text{ for some }k'_l(x)\text{-planes }L^l(x)\sub\R^{n_l},\ l=1,2.
\ee
A natural question is whether, when $A$ is a rectifiable normal $k$-chain such that $\mu=\mu_A$ satisfies~\eqref{wtsc}, there holds, 
\be
\label{k'_1k'_2rectifiability}
\mu_A=\mu_A\restr \bigcup_{k'_1}\Sigma^1_{k'_1}\t\Sigma^2_{k-k'_1},
\ee
for some $k'_l$-rectifiable sets $\Sigma^l_{k'_l}\sub \R^{n_l}$ for $0\le k'_l\le k$ and $l=1,2$. The answer is~\emph{no} and we get a counterexample by building on the construction of Example~5. \\
Let $A_2$ be the sum of the $1$-chains with multiplicity $1$, orientation $e_2$ and supports $V_{x_1}:=\{x_1\}\t[f(x_1),f(x_1^+)]$ for $x_1\in D\sm\{1\}$, see again Figure~\ref{Fig:2}. The chain $A_2$ ``fills the holes'' of $A_1$ in the sense that $\pt (A_1+A_2)=\lb(1,1/3) \rb -\lb(0,0)\rb$. Besides $A_2$ is $(0,1)$-rectifiable so that $T\mu_{A_2}$ is $(0,1)$-split. Summing up, the rectifiable chain  $A=A_1+A_2$ has finite boundary mass and $\mu_A$ satisfies~\eqref{wtsc}.  However, for any countable sets $\Sigma^1,\Sigma^2\sub\R$ we have,
\[
\mu_A(\R^2\sm(\Sigma^1\t\R\cup \R\t\Sigma^2)\ge \mu_{A_1}(\R^2\sm(\Sigma^1\t\R\cup \R\t\Sigma^2))=\mu_{A_1}(\R^2\sm\Sigma^1\t\R)=1.
\]
Consequently,~\eqref{k'_1k'_2rectifiability} does not hold.

\begin{remark}
The proof of the main result presented in this paper takes a long detour by introducing tensor flat chains. As building a theory with a single application might be seen as overkill we could look for a more direct method. However, mimicking the smooth case arguments of Example~2 in the general case seems hopeless. As illustrated by Examples~{3}\,\&\,{4}, the presence of a  boundary with a possibly dense support is an obstacle. Even in the very favorable situation where $A$ is a rectifiable cycle with multiplicity 1, the assumptions do no not ensure that the restriction of $A$ to a $k$-surface of class $C^1$ has a reasonable boundary. 
\end{remark}

\section{$G$-flat chains}
\label{SGfc}
Here we recall the theory of $G$-flat chains in $\R^n$ as introduced by Fleming in~\cite{Fleming66}. We also recall the notion of slices of $G$-flat chains by $m$-planes and the rectifiable slices theorem of White~\cite{White1999-2}. As it is now standard and contrarily to~\cite{Fleming66}, we use the notation $\restr$ and not $\cap$ for restrictions, the notation $\cap$ is used for slicings. Two remarks are in order. First, to avoid the introduction of compact sets which are not relevant here, the flat chains we consider are not assumed to be compactly supported.  Second, in addition to the classical definition of the slice $A\cap \X_{\ov\g}(x)$ which is a flat chain in $\R^n$ supported in $ \X_{\ov\g}(x)$, we introduce the object $\Sl_\g^x A$ which is the flat chain in $ \X_{\ov\g}$ obtained by projection of the former on $\X_{\ov\g}=\X_{\ov\g}(0)$.\\ 
In anticipation of the generalization of Section~\ref{Stfc} we describe the basic theory in detail.
The expert reader can skip most of the section, however, besides the unusual slicing operators $\Sl_\g$, let us point out the following less standard or new results:
\begin{enumerate}[(1)] 
\item Proposition~\ref{prop_divformula} gives a formula for the boundary of the restriction of a chain on a half-space. 
\item Proposition~\ref{prop_split_and_Sl} expresses the $(k_1,k_2)$-splitting property (recall Definition \ref{def:musplit}) by the vanishing of some family of slices.
\item Theorem~\ref{thm_decomp} from~\cite{GM_decomp} asserts that normal rectifiable chains can be decomposed in set-indecomposable subchains.\smallskip
\end{enumerate}

Let $(G,+,|\cdot|_G)$ be a complete Abelian normed group, let $\b\in I^n$ and $0\le k\le |\b|$. 

We say that the pair $p=(q,\x)$ is a $k$-cell of $\X_\b$ if  $q\sub\X_\b$ is a closed convex $k$-polyhedron and $\x\in\WE_k\X_\b$ is a unit simple $k$-vector orienting the affine $k$-plane spanned by $q$. There are two possible choices of orientation and the $k$-cell $(q,-\x)$ is denoted $-p$. The multivector $\xi$ is the orientation of $p$ and is denoted $\x_p$. The polyhedron $q$, denoted $\supp p$, is the support of $p$,  However, for set operations such as unions and intersections, we usually write $p$  for $\supp p$.\\
As  usual the $0$-cell with support $x$ and orientation $1$ is denoted $\lb x\rb$ and for $x\ne y\in\X_\b$, the $1$-cell supported by $[x,y]$ and with orientation $(y-x)/|y-x|$ is denoted $\lb(x,y)\rb$. If $q$ is a $|\b|$-polyhedron, we denote $\lb\un_q\rb$ the $|\b|$-cell supported by $q$ and with orientation $e_{\b_1}\we e_{\b_2}\we\dots\we e_{\b_{|\b|}}$ (referred to as \emph{positive} orientation). \medskip
 
The group $\PP^G_k(\X_\b)$ of $G$-polyhedral $k$-chains in $\X_\b$ is the group of formal finite sums,
\be\label{Gpolyhedralchain}
P= \sum g_i p_i,\qquad g_i\in G,\quad p_i\ k \text{-cells in }\X_\b,
\ee
quotiented by the following relations : for $g,g'\in G$ and $p,p',p''$, $k$-cells,
\begin{enumerate}[(i)]
\item $g p+g(-p)=0$;
\item $g p+g' p = (g+g')p$;
\item $g p'+g p'' = g p$ if $\x_p=\x_{p'}=\x_{p''}$ and, as sets, $p=p'\cup p''$ with $p'$ and $p''$ intersecting on a common $(k-1)$-face.
\end{enumerate}

The representation~\eqref{Gpolyhedralchain} of $P\in\PP^G_k(\X_\b)$ is not unique. The mass of $P$ is defined as,
\be\label{defM(P)}
\M(P):=\min\lt\{\sum |g_i|_G\h^k(p_i)\rt\},
\ee
where the infimum runs over all the representations of $P$.\footnote{This infimum if reached at any representation such that $|g_i|_G|g_j|_G \h^k(p_i\cap p_j)=0$ whenever $i\ne j$.}\\
The boundary $\pt p$ of a $k$-cell $p$ is the sum of its $(k-1)$-faces $p'_j$ with orientation $\x'_j$ such that $\nu_j\we \x'_j=\x_p$ where $\nu_j$ is the exterior normal to $p$ on $p'_j$. Coming back to~\eqref{Gpolyhedralchain}, it is easy to see that the $(k-1)$-chain $\sum g_i \pt p_i$ does not depend on the particular representation of $P$ and we can safely define the boundary of $P$ as
\[
\pt P:=\sum g_i \pt p_i.
\]
For instance for $g\in G$ and $x,y\in\X_\b$,
\[
\pt(g\lb x\rb)=0\qquad{ and }\qquad\pt(g\lb(x,y)\rb)=g\lb y\rb-g\lb x\rb.
\]
We obtain a group morphism $\pt:\PP^G_k(\X_\b)\to\PP^G_{k-1}(\X_\b)$ satisfying $\pt^2=0$. Whitney's flat norm is then given for polyhedral chains by:
\[
\F(P):=\inf\lt\{\M(Q)+\M(R) :Q\in\PP^G_k(\X_\b),\  R\in \PP^G_{k+1}(\X_\b),\  P=Q+\pt R\rt\}.
\]
It turns out that $\F$ is a norm on $(\PP_k(\X_\b),+)$ (\cite[Theorem~2.2]{Fleming66}) moreover $\pt$ is continuous with respect to the $\F$-norm. Indeed writing $P=Q+\pt R$, we have $\pt P=\pt Q$ and $\F(\pt P)\le\M(Q)$ which yields $\F(\pt P)\le\F(P)$ by taking the infimum over the possible representations. The group $\FF^G_k(\X_\b)$ of $k$-chains in $\X_\b$ and with coefficients in $G$ is the completion of $(\PP^G_k(\X_\b),+,\F)$. Unless specified otherwise, we always intend convergence of flat chains in $\F$ norm. The boundary operator extends by continuity as a $1$-Lipschitz continuous morphism $\pt:\FF^G_k(\X_\b)\to\FF^G_{k-1}(\X_\b)$. In summary, the sequence $(\FF^G_k(\X_\b))_{k\in\Z}$ endowed with the norm $\F$ and the operator $\pt$ forms a chain complex of normed groups: 
\[
\begin{array}{ccccccc} 
\cdots\ \st\pt\longto\ &\FF^G_{k+1}(\X_\b) &\st\pt\longto\ & \FF^G_k(\X_\b) &\st\pt\longto&\FF^G_{k-1}(\X_\b) &\st\pt\longto\ \cdots \\
\end{array}
\]  
The mass of $A\in\FF^G_k(\X_\b)$ is then defined as,
\[
\M(A):=\inf\lt\{\liminf \M(P_j):\PP^G_k(\X_\b)\ni P_j\st\F\to A\rt\}.
\]
By~\cite[Theorem~2.3]{Fleming66}, the function $\M$ first introduced in~\eqref{defM(P)} is lower semi-continuous with respect to $\F$-convergence in the group of polyhedral chains, so the definition of $\M(P)$ is not ambiguous.   We have the following properties.
\begin{proposition}[{\cite[Theorem~3.1]{Fleming66} for~(ii), \cite[Theorem~5.6.]{Fleming66} for (iii)}]
\label{prop_alternateWG} 
Let $A\in\FF^G_k(\X_\b)$.
\begin{enumerate}[(i)]
\item $\F(\pt A)\le\F(A)\le \M(A)$ with possibly $\M(A)=\oo$.
\item $\F(A)=\inf \lt\{\M(B)+\M(C): B\in\FF^G_k(\X_\b),\, C\in\FF^G_{k+1}(\X_\b)\text{ such that }A=B+\pt C\rt\}$.
\end{enumerate}
\end{proposition}

\subsubsection*{The limit case $k=|\b|$}
We assume here that $k=|\b|$. Since $\FF^G_{k+1}(\X_\b)=\{0\}$, there holds $\F=\M$ on $\FF^G_k(\X_\b)$. Moreover, a $G$-polyhedral  $k$-chain of the form~\eqref{Gpolyhedralchain} identifies with the function $f_P:\X_\b\to G$ given by 
\[
f_P:=\sum \lt<\x_{p_i},e_\b\rt>g_i \un_{\supp p_i},
\]
(here the factors $ \lt<\x_{p_i},e_\b\rt>=\pm 1$ account for the orientation). There holds $\|f_P\|_{L^1}=\M(P)$ and as the $f_P$'s form a dense subgroup of $L^1(\X_\b,G)$, we have in fact the isometry, 
\[
(\FF^G_{|\b|}(\X_\b),\F)=(\FF^G_{|\b|}(\X_\b),\M)\sim(L^1(\X_\b,G),\|\cdot\|_{L^1}).
\]

 \subsubsection*{Subgroups of chains}
We denote by $\MM^G_k(\X_\b)$, the subgroup of finite mass $k$-chains in $\X_\b$ with coefficients in $G$. This group, endowed with $\M$ is complete but the family $(\MM^G_*(\X_\b),\pt)$ does not form a chain complex. Denoting $\N(A):=\M(A)+\M(\pt A)$, we define the group of normal flat $k$-chains, 
\[
\NN^G_k(\X_\b):=\lt\{A\in\FF^G_k(\X_\b):\N(A)<\oo\rt\}.
\]
Again, $(\NN^G_k(\X_\b),\N)$ is a complete normed group. Moreover, since $\N(\pt A)=\M(\pt A)\le\N(A)$, the family $(\NN^G_*(\X_\b),\pt)$ forms a chain complex of normed groups.

A chain such that $\pt A=0$ is called a cycle (the finite mass $k$-cycles form a closed subgroup of $\NN^G_k(\X_\b)$ and the sequence of subgroups of finite mass cycles also form a chain complex).

\subsubsection*{Supports}
We say that the flat chain $A\in \FF^G_k(\X_\b)$ is supported in the closed set $S\sub\X_\b$ if for any neighborhood $U$ of $S$, there exists a sequence $P_j\in \PP^G_k(\X_\b)$ which admits representations $P_j=\sum g^j_ip^j_i$ with $\supp p^j_i\sub U$, such that $P_j\to A$ in $\FF^G_k(\X_\b)$. If there exists a smallest closed subset $S$ supporting $A$, we write $\supp A=S$. Finite mass flat chains admit a support and moreover for $A\in \MM^G_k(\X_\b)$ and any neighborhood $U$ of $\supp A$ there exists a sequence $P_j\in \PP^G_k(\X_\b)$ supported in $U$ such that $P_j\to A$ with $\M(P_j)\to\M(A)$.

\subsubsection*{Tensor products of chains (see \cite[Sections~6]{Fleming66})}

Let $\b$, $\g\in I^n$ such that $\b\cap\g=\void$ and assume for simplicity that $\max\b<\min\g$. Given two polyhedral cells $p^1=(q^1,\x^1)$ and $p^2=(q^2,\x^2)$ respectively in $\X_\b$ and $\X_\g$ and $g\in G$, we define the tensor product $p^1\we(gp^2)$ of $p^1$ and $gp^2$ as the chain $gp\in\PP^G_*(\X_{\b\cup\g})$ where $p$ is the polyhedral chain with support  $q^1+q^2$ and orientation $\x^1\we\x^2$. Given $0\le k_1,k_2\le n$, the definition extends by $\Z$ bilinearity as a bilinear mapping 
\begin{align*}
\PP^\Z_{k_1}(\X_\b)\t\PP^G_{k_2}(\X_\b)&\longto\PP_{k_1+k_2}^G(\X_{\b\cup\g}),\\
(P^1,P^2)&\longmapsto P^1\t P^2.
\end{align*}
There holds, 
\[
\pt (P^1\t P^2)=(\pt P^1)\t P^2+(-1)^{k_1}P^1\t(\pt P^2),
\]
and we have the estimates, 
\[
\M(P^1\t P^2)\le \M(P^1)\M(P^2),\qquad\F(P^1\t P^2)\le\N(P^1)\F(P^2),\qquad\F(P^1\t P^2)\le\F(P^1)\N(P^2).
\]
Using these estimates, and a density argument, we can define $A^1\t A^2$ for $A^1\in\FF^\Z_{k_1}(\X_\b)$, $A^2\in\FF^G_{k_2}(\X_\g)$ as soon as both $A^1$ and $A^2$ have finite mass or at least one of them is normal.\medskip

This construction foreshadows the definition of tensor chains in Section~\ref{Stfc}. It has the disadvantage of not being well defined for any pair of flat chains $A^1$, $A^2$. The groups of tensor flat chains correct this defect but inevitably contain objects which are not flat chains.

\subsubsection*{Push-forwards by Lipschitz maps and homotopies  (see \cite[Sections~5\&6]{Fleming66})}
Let $\g\in I^n$. A Lipschitz mapping $f:\X_\b\to\X_\g$ induces a morphism of chain complexes of normed groups, $f\pf:\FF^G_*(\X_\b)\to\FF^G_*(\X_\g)$. In particular, for $A\in\FF^G_*(\X_\b)$, there holds $\pt f\pf A=f\pf\pt A$. Moreover the following estimates hold true.
\[
 \M(f\pf A)\le C^k \M(A),\qquad\qquad\F(f\pf A)\le\max(C^k,C^{k+1})\F(A),
\]
where $C$ is a Lipschitz constant of $f$ on $\supp A$.
This mapping has the following additional properties for $\g,\d\in I^n$ and $A\in \FF^G_k(\X_\b)$.
\begin{enumerate}[(1)]
\item If $f\in \Lip(\X_\b,\X_\g)$, and $A$ is supported in $\{x\in \X_\b: f(x)=0\}$ then $f\pf A=0$.
\item For $f\in \Lip(\X_\b,\X_\g)$ and $g\in \Lip(\X_\g,\X_\d)$ then $(g\circ f)\pf A=g\pf(f\pf A)$.
\item $\Id\pf A=A$ and it follows from the previous points that if $g\circ f=\Id$ on some closed set $S$ supporting $A$  then $g\pf(f\pf A)=A$.
\item If the image of $f\in \Lip(\X_\b,\X_\g)$ lies in some closed set $S$ then $f\pf A$ is supported in $S$. If moreover $S=L$ is an affine subspace of $\X_\g$  we have a stronger property: there exists a sequence of $G$-polyhedral chains $P_j$ all supported in $L$ such that $P_j\to f\pf A$ so that $f\pf A$ identifies with a chain in $L$.
\end{enumerate}

Given, $f,g\in \Lip(\X_\b,\X_\g)$ and $A\in\FF^G_k(\X_\b)$, we define a linear homotopy from $f\pf A$ to $g\pf A$ as follows. Let $e_0:=(1,0_{\R^n})\in \R\times \R^n$ and consider the mapping: 
\[
h: te_0+x\in \R e_0+\X_\b\longmapsto tf(x)+(1-t)g(x)\in \X_\g,
\]  
We have the identity
\be\label{homotopy}
f\pf A -g\pf A= \pt\lt[h\pf(\lb(0,e_0)\rb\t A)\rt]+h\pf(\lb(0,e_0)\rb\t \pt A),
\ee
and the estimate
\be\label{estim_homotopy}
\M(h\pf(\lb(0,e_0)\rb\t A))\le \|f-g\|_{L^\oo(\supp A)}C^k\M(A),
\ee
where $C$ is a common Lipschitz constant of $f$ and $g$ on $\supp A$.

\subsubsection*{Slicing of $G$-flat chains}
In this part, we recall the definition and properties of slices as introduced in~\cite[Section~3]{White1999-2} with the small addition of the new operators $\Sl_\g$ (see below). We only consider  slices by affine coordinate planes.

Let $\g\sub\b$ and set $m:=|\b|$, $r:=|\g|$. Given a $G$-polyhedral $k$-chain $P$ with representation~\eqref{Gpolyhedralchain} and $x\in\X_\g$, if $r\le k$, the slice of $P$ with respect to $\X_{\b\sm\g}(x)$ is the $(k-r)$-chain defined as, 
\[
P\cap\X_{\b\sm\g}(x) := \sum g_i q_i,
\]
where $q_i=0$ if the affine subspace spanned $p_i\cap \X_{\b\sm\g}(x)$ is not of dimension $k-r$ and in the other cases $q_i$ is the $(k-r)$-cell of $\X_{\b\sm\g}(x)$ described by (recall the notation \eqref{ebeta}): 
\[
\begin{cases}
\supp q_i:=\supp p_i\cap \X_{\b\sm\g}(x)\\
\x_{q_i}\text{ is the unit simple }(k-r)\text{-vector orienting }q_i\text{ such that }\lt<e_\g\we\x_{q_i},\x_{p_i}\rt>>0.
\end{cases}
\]
When $r>k$, we set $P\cap\X_{\b\sm\g}(x) :=0$. Notice that for $x\in\X_\g$, the chain $P\cap\X_{\b\sm\g}(x)$ is supported in $\X_{\b\sm\g}(x)$. Identifying each $\X_{\b\sm\g}(x)$ with $\X_{\b\sm\g}$, it will be convenient to see all these chains as chains in $\X_{\b\sm\g}$. For this reason we define, 
\be\label{Slcap}
\Sl_\g^xP:=\pi\pf (P\cap\X_{\b\sm\g}(x))\ \in\PP^G_k(\X_{\b\sm\g}),
\ee
where  $\pi$ is the orthogonal projection on $\X_{\b\sm\g}$.  Denoting $\pi_x:=x+\pi$ the orthogonal projection on $\X_\g(x)$, 
we see that $\pi_x:\X_{\b\sm\g}\to\X_{\b\sm\g}(x)$ is an affine isometry with inverse $\pi$. We deduce a reverse formula for~\eqref{Slcap},
\[
P\cap\X_{\b\sm\g}(x)=\pi_x\pf \Sl_\g^xP.
\]
The usual properties of $P\cap \X_\g(x)$ transfer to $\Sl^x_\g$. We have in particular,
\[
\pt\Sl^x_\g P=\Sl^x_\g\pt P\quad\text{for a.e. }x\qquad\text{ and }\qquad\int_{\X_\g}\M(\Sl_\g^xP)\,d\h^{r}(x) \le \M(P),
\]
so that $\Sl_\g$ extends on $\FF^G_k(\X_\b)$ as a Lipschitz continuous group morphism:
\begin{align*}
\Sl_\g : \FF^G_k(\X_\b) &\,\longto\ L^1\lt(\X_\g,\FF^G_{k-r}(\X_{\b\sm\g})\rt),\\
A\quad&\,\longmapsto\qquad  x\mapsto \Sl_\g^xA. 
\end{align*}  
For the usual slicing operator, we have:
\begin{align*}
 \FF^G_k(\X_\b) &\,\longto\ L^1(\X_\g,\FF^G_{k-r}(\X_\b),\\
A\quad&\,\longmapsto\qquad  x\mapsto A \cap(x+\X_{\b\sm \g}). 
\end{align*}  
The two definitions of slicing contain the same information (same mass, same $\F$-norm, ....). The second definition~\eqref{Slcap} is not standard but useful in the proof of the main result.

\begin{proposition}[{\cite[Proposition~3.1~(1,2)]{White1999-2}}]\label{prop_MWSl}
Let $\g\sub\b$. The operator $\Sl_\g$ is a group morphism  satisfying the following properties for $A\in \FF^G_k(\X_\b)$. 
\begin{enumerate}[(i)]
\item There holds $\Sl_\g^x\pt A=\pt\Sl_\g^xA$ for almost every $x\in\X_\g$.
\item For $\z\sub\g$ and almost every $x\in\X_\g$, writing $x=x'+\ov x$ with $x'\in \X_\z$, $\ov x\in\X_{\g\sm\z}$, we have, 
\[
\Sl_\g^x A=\Sl_\g^{\ov x}\Sl_\z^{x'}A.
\] 
With the usual operator, this translates as,
\[
A\cap\X_{\b\sm\g}(x)=\lt[A\cap \X_{\b\sm\z}(x')\rt]\cap\X_{\b\sm\g}(x). 
\]
\item Moreover, denoting $r:=|\g|$, we have the estimates:
\[
\int_{\X_\g}\M(\Sl_\g^xA)\,d\h^r(x)\le\M(A),\qquad\qquad
\int_{\X_\g}\F(\Sl_\g^xA)\,d\h^r(x)\le\F(A).
\]
In particular, if $A\in \MM^G_k(\X_\b)$ then $\Sl_\g^xA\in\MM^G_{k-r}(\X_{\b\sm\g})$ for almost every $x\in\X_\g$.
\end{enumerate}
\end{proposition}\medskip

The collection of the slices of $k$-chain by coordinate planes of codimension $k$ completely characterizes this chain.  

\begin{theorem}[{\cite[Theorem~3.2]{White1999-2}}]
\label{coro_defWhite2}
If $A\in \FF^G_k(\X_\b)$ is such that $\Sl_\g A=0$ for  every $\g\sub\b$ with $|\g|=k$, then $A=0$.
\end{theorem}
  
\begin{remark}[and vocabulary]
For $0\le m\le k$, the slices by planes of codimension $k-m$ are of dimension $m$ (if $A$ is a $k$-chain in $\X_\b$, for $\g\sub\b$ with $|\g|=m$, $\Sl^x_\g A$ and $A\cap\X_{\b\sm\g}(x)$ are $m$-chains). They are called its $m$-slices. Theorem~\ref{coro_defWhite2} states that a chain is determined by its $0$-slices (and, using iterate slicing, in fact  by its $m$-slices for any $m\in\{0,\dots,k\}$). Since finite mass $0$-chains have a nice structure (as we see below they are measures), this result is of particular interest.
\end{remark}

\subsubsection*{Finite mass $G$-flat chains}

Let us now turn our attention to finite mass $G$-flat chains.
 To every $A\in\MM^G_k(\X_\b)$ is associated a mapping 
\[
S\in\{\text{Borel subsets of }\X_\b\}\ \longmapsto\ A\restr S\in \MM^G_k(\X_\b),
\]
and a finite Borel measure $\mu_A$ defined by
\[
\mu_A(S):=\M(A\restr S).
\]

The restriction $A\restr S$ corresponds to ``the part of $A$ in $S$''. For a polyhedral chain $P$ and a $n$-interval $I$, $P\restr I$ is defined in the obvious way. The restriction enjoys the following properties for $A\in\MM^G_k(\X_\b)$ and $S$ Borel subset of $\X_\b$ (see\footnote{Beware that in Fleming's paper the restriction is denoted $\cap$ and not $\restr$.}~\cite[Lemma~4.2, Theorem~4.3]{Fleming66})
\begin{enumerate}[(i)]
\item If $A_j\in\MM^G_k(\X_\b)$ is such that $A_j\to A$ and $\M(A_j)\to\M(A)$ then $A_j\restr S\to A\restr S$ whenever $\mu_A(\frt S)=0$.
\item $A$ admits a support and if $S$ is closed, $\supp (A\restr S)=S\cap\supp\mu_A$. In particular, $\supp A=\supp\mu_A$. If $S$ is merely a Borel set, we still have $\supp(A\restr S)\sub \ov S\cap\supp\mu_A$.
\end{enumerate}
We recall that  the slicing and restriction operators commute.
\begin{proposition}[{\cite[Proposition~3.1~(3)]{White1999-2}}]\label{prop_Slrestr}
Let $A\in \FF^G_k(\X_\b)$,  let $S\sub\X_\b$ be a Borel set and let $\g\sub\b$. For almost every $x\in\X_\g$, there holds 
\[
(A\restr S)\cap \X_{\b\sm\g}(x)=(A\cap \X_{\b\sm\g}(x))\restr S = (A\cap \X_{\b\sm\g}(x))\restr (S\cap \X_{\b\sm\g}(x)).
\]
Using the $\Sl_\g$ operator, this rewrites as,
\[
\Sl_\g^x(A\restr S)=(\Sl_\g^x A)\restr (S-x) = (\Sl_\g^x A)\restr [(S-x)\cap \X_{\b\sm\g}].
\]
\end{proposition}
 
The next proposition is a key tool in the proof of Proposition~\ref{prop_main_0k}.
 
\begin{proposition}\label{prop_divformula}
 Let $A\in\NN^G_k(\R^n)$. Denoting  $H(s):=\{x\in\R^n: x_1>s\}$ then $A\restr H(s)\in \NN^G_k(\R^n)$ for almost every $s\in\R$ and the following identity holds true.
\[
\pt(A\restr H(s))=(\pt A)\restr H(s)+A\cap \X_{\ov1}(se_1).
\]
Notice that $ \X_{\ov1}(se_1)=\bndry H(s)$.
\end{proposition} 
\begin{proof}
Let $A\in\NN^G_k(\R^n)$. By~\cite[Theorem~5.6]{Fleming66}, there exists $Q_j\in \PP^G_k(\R^n)$ such that $Q_j\to A$, $\M(Q_j)\to\M(A)$ and $\M(\pt Q_j)\to\M(\pt A)$. By~\cite[Lemma~4.2]{Fleming66} we have for almost every $s\in\R$ (recall that convergence is intended in $\F$-norm),
\be\label{proof_prop_divformula1}
Q_j\restr H(s) \longto A\restr H(s),\qquad\qquad (\pt Q_j)\restr H(s) \longto(\pt A)\restr H(s).
\ee
Moreover, the proposition is obvious for polyhedral chains, so,
\be\label{proof_prop_divformula2}
\pt(Q_j\restr H(s))-(\pt Q_j)\restr H(s)  = Q_j\cap \X_{\ov1}(se_1).
\ee
By Proposition~\ref{prop_MWSl}, the mapping $s\mapsto Q_j\cap \X_{\ov1}(se_1)$ converges towards $s\mapsto A\cap \X_{\ov1}(se_1)$ in $L^1(\R,\FF^G_{k-1}(\R^n))$. Consequently, up to extraction, $Q_j\cap \X_{\ov1}(se_1)\to A\cap \X_{\ov1}(se_1)$. Passing to the limit in~\eqref{proof_prop_divformula2} and using~\eqref{proof_prop_divformula1} yield the result.
\end{proof}

Let $\MM(\X_\b,G)$ be the group of $G$-valued Borel measures in $\X_\b$. The total variation of $\nu\in\MM(\X_\b,G)$ is defined as,
\[
|\nu|:=\sup\lt\{\sum|\nu(S_j)|_G:S_j\text{ Borel partition of }\X_\b\rt\}.
\]
We recall that $(\MM(\X_\b,G),+,|\cd|)$ is a complete Abelian normed group and that finite mass $0$-chains identify with $G$-valued Borel measures. 
\begin{theorem}[{\cite[Theorem~2.1]{White1999-2}}]\label{thm_psi}
There exists an isometric isomorphism, 
\[
\psi:(\MM^G_0(\X_\b),+,\M)\longto(\MM(\X_\b,G),+,|\cd|).
\]
\end{theorem}

\begin{remark}
Using this identification, given finite sequences $x_i\in \X_\b$ and $g_i\in G$, we may write $\sum g_i\d_{x_i}$ for the polyhedral $0$-chain $\sum g_i\lb x_i\rb$.
\end{remark}

\subsubsection*{Morphisms of the group of coefficients}
Let $(G^a,+,|\cd|_{G^a})$ and $(G^b,+,|\cd|_{G^b})$ be two complete Abelian normed groups and  $\phi:G^a\to G^b$ be a Lipschitz continuous group morphism. For $P^a\in\PP_k(\X_\b,G^a)$ with representation,
\be\label{Pa}
P^a=\sum g^a_jp_j,
\ee
we define 
\begin{equation}\label{star}
\phi\et P^a:=\sum \phi(g^a_j)p_j.
\end{equation}
This mapping extends as a Lipschitz continuous morphism:
\[
\phi\et :\FF_k(\X_\b,G^a)\longto\FF_k(\X_\b,G^b).
\]
Moreover, for every $A\in\FF_k(\X_\b,G^a)$, there holds 
\begin{equation}\label{phidiese}
\pt\phi\et A=\phi\et\pt A \qquad \textrm{and} \qquad \M(\phi\et A)\le \Lip(\phi) \M(A),
\end{equation}
where $\Lip(\phi)$ is the Lipschitz constant of $\phi$.
More generally, if $(\Om,\mu)$ is a measured space and 
\[
\Phi:G^a\to L^1(\Om,G^b),
\]
is a Lipschitz continuous group morphism, the mapping defined, for $P^a$ of the form~\eqref{Pa} and $\om\in\Om$, by 
\[
(\Phi\et P^a)(\om):=\sum [\Phi(g^a_j)](\om)p_j,
\]
extends as a Lipschitz continuous morphism:
\[
\Phi\et :\FF_k(\X_\b,G^a)\longto L^1\lt(\Om,\FF_k(\X_\b,G^b)\rt).
\]
\begin{proposition}\label{prop_phi_Phi}
With the above notation, for every $A\in\FF_k(\X_\b,G^a)$, there hold
\[
\phi\et \Sl_\g A= \Sl_\g \phi\et A,\qquad\qquad(\Phi\et \Sl_\g A)(\om)= \Sl_\g[\Phi\et A(\om)]\quad\text{for almost every }\om\in\Om.
\]
Similarly,  for every $A\in\MM_k(\X_\b,G^a)$ and every Borel set $S\sub\X_\b$, there hold
\[
\phi\et(A\restr S)= (\phi\et A)\restr S,\qquad\qquad[\Phi\et(A\restr S)](\om)=[\Phi\et A(\om)]\restr S\quad\text{for almost every }\om\in\Om.
\]
\end{proposition}
\begin{proof}
These identities hold true for polyhedral chains and we get the general cases by a continuity/density argument.
\end{proof}

\subsubsection*{Rectifiable chains}  
It follows from the deformation theorem of White (see~\cite[Theorem 3.1]{White1999-1}) that given a $k$-chain $A\in\MM^G_k(\X_\b)$ and a Borel set $S\sub\X_\b$ if $\h^k(S)=0$ or even if  $S$ has vanishing $k$-integral geometric measure then $A\restr S=0$.\footnote{In~\cite[Theorem~3.1]{White1999-1} the result is stated under the stronger assumption $\h^k(\supp A)=0$, but the present statement transpires from some arguments of~\cite{White1999-2} and follows rather directly from~\cite[Corollary~6.1]{White1999-2}.} 
This means that $\mu_A$ cannot concentrate more than $\h^k$. The limit case corresponds to rectifiable chains defined below.\medskip

Let us recall that a Borel set $S\sub\X_\b$ is (countably) $k$-rectifiable if $S\sub\Sigma_0\cup(\cup_{j\ge1}\Sigma_j)$ where the $\Sigma_j$'s are Borel subsets of $\X_\b$ such that $\h^k(\Sigma_0)=0$ and for $j\ge1$, $\Sigma_j$ is a $C^1$ $k$-manifold.

\begin{definition}~
\begin{enumerate}[(1)]
\item A finite Borel measure $\mu$ on $\X_\b$ is $k$-rectifiable if $\mu=\mu\restr S$ for some $k$-rectifiable set $S\sub\X_\b$ and $\mu\ll\h^k\restr S$. Observe that the set $\Sigma_0$ does not play any role here.
\item A $G$-flat chain $A\in\MM^G_k(\X_\b)$ is rectifiable if $A=A\restr S$ for some $k$-rectifiable set $S\sub\X_\b$.  Equivalently $\mu_A$ is a $k$-rectifiable measure.
\end{enumerate}
\end{definition}
Remark that if by definition rectifiable chains have finite mass they are not assumed to be compactly supported.\medskip

By Theorem \ref{thm_psi}, finite mass $0$-chains identify with finite measures with values in $G$ and rectifiable $0$-chains in $\X_\b$ are the atomic measures (of the form $\sum g_j\lb x_j\rb$ with $x_j\in\X_\b$, $g_j\in G$ and $\sum|g_j|_G<\oo$).  Given a  rectifiable $k$-chain $A$ in $\X_\b$ and $\g\sub\b$ with $|\g|=k$, for almost every $x\in\X_\g$, $\Sl_\g^xA$ is a rectifiable $0$-chain in $ \X_{\ov\g}$.
White's rectifiable slices theorem states that the converse is also true. This is a crucial tool in our proof of the main result.
\begin{theorem}[\cite{White1999-2}]\label{thm_White}
Let $A\in\MM^G_k(\X_\b)$. The following are equivalent.
\begin{enumerate}[(i)]
\item A is rectifiable.
\item For every $\g\sub\b$ with $|\g|=k$ and almost every $x\in \X_\g$, $\Sl_\g^xA$ is a rectifiable $0$-chain. 
\end{enumerate}
\end{theorem}

The $(k_1,k_2)$-splitting property (recall Definition \ref{def:musplit}) can be characterized in terms of slices.
\begin{proposition}\label{prop_split_and_Sl}
Let $A\in\MM^G_k(\R^n)$ be rectifiable and let $n_1,n_2,k_1,k_2\ge0$ with $n_1+n_2=n$ and $k_1+k_2=k$. The following statements are equivalent. 
\begin{enumerate}[(i)]
\item $T\mu_A$  is $(k_1,k_2)$-split.
\item $\Sl_\g A=0$ for every $\g\in I_k^n$ such that $(|\g^1|,|\g^2|)\ne (k_1,k_2)$.
\end{enumerate}
\end{proposition}
\begin{proof}
Without loss of generality, we assume that $A\in\MM^G_k(\R^n)$ is such that $A=A\restr\Sigma$ where $\Sigma$ is a compact $k$-manifold of class $C^1$ (with boundary). \medskip

\noindent
\textit{Step 1, (i)$\implies$(ii).} 
Let $x\in\Sigma\mapsto\x(x)\in\WE_k(\R^n)$ be a Borel mapping such that for every $x\in\Sigma$, $\x(x)$ is a unit $k$-vector orienting  $T_x\Sigma$ and let $\g\in I^n_k$. 
Let us define, 
\[
D:=\{x\in \X_\g:A\cap \X_{\ov\g}(x)\ne0\}.
\]  
Denoting by $\pi$ the orthogonal projection on $\X_\g$ we have that, up to a $\h^k$-negligible set, $D\sub\pi\Sigma$. Indeed, let 
$x\in \X_\g\sm\pi\Sigma$. Since $\Sigma$ is compact, there exists $r>0$ such that denoting $B'_r(x):=\{y\in\X_\g:|y|<r\}$ we have $\Sigma\cap B'_r(x)+\X_{\ov\g}=\void$. It follows that 
\[
A\restr [B'_r(x)+\X_{\ov\g}]=(A\restr\Sigma) \restr [B'_r(x)+\X_{\ov\g}]=A\restr(\Sigma\cap[B'_r(x)+\X_{\ov\g}])=0.
\]
Hence $A\cap \X_{\ov\g}(y)=0$ $\h^k$-almost everywhere in $B_r'(x)$ and $D\cap B'_r(x)$ is negligible. We conclude that $D\sub\pi\Sigma$ as claimed.\\
Denoting $S:=\Sigma\cap(D+\X_{\ov\g})$, we have $D=\pi S$ by the previous inclusion and using the co-area formula~\cite[Theorem~2.93]{Am_Fu_Pal},  we compute,
\be\label{prop_split_and_Sl_1}
\h^k(D)\le\int_{\X_\g} \h^0( \X_{\ov\g}(y)\cap S)\,dy=\int_S\lt|\lt<e_\g,\x(x)\rt>\rt|\,d\h^k(x).
\ee
Next, let $S'\sub\R^n$ be a Borel set such that $\mu_A(S')=0$. We have $A\restr S'=0$ and using Proposition~\ref{prop_Slrestr} we see that for almost every $x\in \X_\g$, 
\[
\lt[A\cap \X_{\ov\g}(x)\rt]\restr S'= \lt[A\restr S'\rt]\cap \X_{\ov\g}(x)=0.
\]
It follows that $\h^k(S\cap S')=0$. We deduce that, 
\be\label{prop_split_and_Sl_2}
\h^k\restr S\ll\mu_A.
\ee
Let us now assume that~(i) holds true. For $\g$ such that $(|\g^1|,|\g^2|)\ne(k_1,k_2)$, we have by assumption $\lt<e_\g,\x\rt>=0$  $\mu_A$-almost everywhere. By~\eqref{prop_split_and_Sl_2}, this holds true $\h^k\restr S$-almost everywhere and the right-hand side of~\eqref{prop_split_and_Sl_1} vanishes. We conclude that $\Sl_\g A=0$ so~(ii) holds true.\medskip

\noindent
\textit{Step 2, (ii)$\implies$(i).} 
Conversely, let us assume that $A$ satisfies~(ii) and let us suppose by contradiction that $T\mu_A$ is not $(k_1,k_2)$-split. In this case there exist a $k$-manifold $\Sigma$ of class $C^1$, a point $x$ on $\Sigma$ and $r>0$ such that $T_x\Sigma$ is well defined and \emph{not $(k_1,k_2)$-split} and denoting $A':=A\restr\Sigma$, there holds,
\be\label{proof_prop_split_and_Sl_2-1}
\mu_{A'}(B_s(x))>0\qquad\text{for }s\in (0,r).
\ee
Let us assume without loss of generality that $x=0$. Since $T_0\Sigma$ is not $(k_1,k_2)$-split there exists $\g\in I^n_k$ with $(|\g^1|,|\g^2|)\ne (k_1,k_2)$ such that, in some neighborhood of $0$, $\Sigma$ is the graph of a Lipschitz mapping over $\X_\g$.
Denoting $B'_r:=B_r(0)\cap \X_\g$, we define the cylinder $C_r:=B'_r+\X_{\ov\g}$. Reducing $r$ if necessary, we may assume that the connected component of $\Sigma\cap C_r$ which contains $0$ is the graph of a Lipschitz continuous function over $B'_r$. We denote $\Sigma_*$ this component and write
 \be\label{proof_prop_split_and_Sl_2-15}
\Sigma_*=\{y+f(y):y\in B_r'\},
\ee
with  $f\in\Lip(\X_\g,\X_{\ov\g})$.\bigskip

\noindent
\textit{Step 2.a.}  Let $A_*:=A\restr\Sigma_*$. Notice that this chain also satisfies~\eqref{proof_prop_split_and_Sl_2-1}. Let $\pi$ be the orthogonal projection on $\X_\g$.  Let us establish the identity:
\be\label{proof_prop_split_and_Sl_2-2}
 \Sl_\g\pi\pf A'= \pi\pf \Sl_\g A'\qquad\text{for every }A'\in\FF^G_k(\R^n).
\ee
We first prove it for $A'=gp$ with $g\in G$ and $p=(q,\x)$ a polyhedral $k$-cell in $\R^n$. If $\h^k(\pi q)=0$ then both $\pi\pf( gp)$ and $\Sl_\g (gp)$ vanish. In the other case, denoting $q':=\pi q$, $q'$ is a convex $k$-polyhedron in $\X_\g$ and $q$ is the restriction on $q'$ of the graph of a linear mapping $\ell:\X_\g\to\X_{\ov\g}$. We observe that: \smallskip

\noindent
($*$) $\pi\pf A'$ is the polyhedral cell $\sigma g(q', e_\g)$ where $\sigma=\pm1$ is the sign of $\lt<e_\g;\x\rt>$. \smallskip\\  
\noindent
($*$) $\Sl^y_\g A'=\ds\begin{cases}\quad0&\text{for }y\in\X_\g\sm \ov{q'},\\ \sigma g\lb \ell(y)\rb&\text{for }y\text{ in the relative interior of }q'.\end{cases}\smallskip$

\noindent
It follows that $\Sl^y_\g\pi\pf A'=\pi\pf\Sl^y_\g A'=\ds\begin{cases}\quad0&\text{for }y\in\X_\g\sm \ov{q'},\\ \sigma g\lb 0\rb&\text{for }y\text{ in the relative interior of }q'.\end{cases}\smallskip$\\
The relation~\eqref{proof_prop_split_and_Sl_2-2} is then true for $A'=gp$. The case of polyhedral $k$-chains follows by $\Z$ linearity and the general case by a density/continuity argument.

Next, applying~\eqref{proof_prop_split_and_Sl_2-2} with $A'=A_*$, we get by assumption,
\be\label{proof_prop_split_and_Sl_2-1.5}
  \Sl_\g \pi\pf A_*=\pi\pf \Sl_\g A_*=\pi\pf (\Sl_\g(A\restr S_*))=\pi\pf ( (\Sl_\g A)\restr S_*) =\pi\pf ( 0\restr S_*)=0.
 \ee
  Since $\pi\pf A_*$ has maximal dimension in $\X_\g$, it writes as a $L^1$ function $g_*:\X_\g\to G$. With this notation,  $\Sl^y_\g\pi\pf A_*=g_*(y)\lb 0 \rb$ for almost every $y\in\X_\g$. We deduce from~\eqref{proof_prop_split_and_Sl_2-1.5} that $g_*\equiv 0$ and consequently, by Theorem \ref{coro_defWhite2}
 \be\label{proof_prop_split_and_Sl_2-3}
 \pi\pf A_*=0.
 \ee

\noindent
\textit{Step 2.b.}  
Recalling~\eqref{proof_prop_split_and_Sl_2-15}, we define $F\in\Lip(C_r,\R^n)$ by, 
\[
F(x)=y+f(y)\qquad\text{with the decomposition } x=y+z,\ y\in L\cap C_r,\ z\in\X_{\ov \g}.
\] 
Noticing that $F\circ\pi$ is the identity on $\Sigma_*$ and using $A_*=A_*\restr\Sigma_*$, we have 
\[
A_*=(F\circ\pi)\pf A_*=F\pf(\pi\pf A_*)\st{\eqref{proof_prop_split_and_Sl_2-3}}=F\pf0=0.
\] 
This contradicts~\eqref{proof_prop_split_and_Sl_2-1}.
\end{proof}

To end this section we recall the notion of set-decomposition of normal chains introduced in~\cite{GM_decomp} and state the existence of maximal set-decompositions for normal \emph{rectifiable} chains. This result and the formula of Proposition~\ref{prop_divformula} are the main ingredients of the proof of Proposition~\ref{prop_main_0k}.
\begin{definition}[{\cite[Definition 1.1]{GM_decomp}}]
\label{def_decomp}
Let $A\in\NN^G_k(\R^n)$.\medskip

\noindent 
(1) A set-decomposition of $A$ is a sequence of normal chains $A_j$ such that there exists a Borel partition $S_j$ of $\R^n$ such that $A_j=A_j\restr S_j$ for every $j$ and $\N(A)=\sum \N(A_j)$.\medskip

\noindent
(2) We say that $A$ is set-indecomposable if the only set-decompositions of $A$ are trivial, that is, for any such decomposition $A_j$ there holds $A_j=A$ for some index $j$ and $A_j=0$ for the others.
\end{definition}

\begin{theorem}[{\cite[Theorem 1.2]{GM_decomp}}]\label{thm_decomp}
Let $A\in\NN^G_k(\R^n)$, if $A$ is rectifiable then  it admits a set-decomposition in set-indecomposable components.
\end{theorem}

\section{Tensor flat chains}
\label{Stfc}

Let us assume $n=n_1+n_2$ with $n_1,n_2\ge0$. We present a generalization of the theory of $G$-flat chains adapted to the decomposition $\R^n=\R^{n_1}\times\R^{n_2}$ and to the partial boundary operators $\pt_1$, $\pt_2$ defined below. The case $\{n_1,n_2\}=\{0,n\}$ corresponds to the classical $G$-flat chains. 

We recall some notation:  we set $\a:=\{1,\dots,n_1\}$ and for $\b\in I^n$,  $\b^1:=\b\cap\a$ and $\b^2:=\b\sm\a=\b\cap\ova$. Eventually,  for $k\ge0$,
\[
D_k:=\{(k'_1,k'_2): 0\le k'_1\le n_1,\ 0\le k'_2\le n_2,\ k'_1+k'_2=k\}.
\]
Throughout the  section, $k_1,k_2,k\in\Z$ and $\b\in I^n$. If not indicated otherwise we assume $k=k_1+k_2$, $0\le k_1\le|\b^1|$ and $0\le k_2\le|\b^2|$.\medskip

\subsection{Tensor polyhedral chains and $\Ft$-norm}
We define the group $\PP_{k_1,k_2}^G(\X_\b)$ of polyhedral $(k_1,k_2)$-chains with coefficients in $G$ as the subgroup of $\PP^G_k(\X_\b)$ formed by elements with representations:
\be\label{k1k2polychain}
P=\sum g_i p^1_i\t p^2_i,
\ee
where for every $i$, $g_i\in G$, $p^1_i$ is a $k_1$-cell in $\X_{\b^1}$ and $p^2_i$ is a $k_2$-cell in $\X_{\b^2}$.  

We denote by $\TP^G_k(\X_\b)$ the subgroup of $\PP^G_k(\X_b)$ spanned by the union of the groups $\PP_{k'_1,k'_2}^G(\X_\b)$ for $(k'_1,k'_2)$ ranging in $D_k$.  The decomposition of $P\in \TP^G_k(\X_\b)$ as $\sum P_{k'_1,k'_2}$ with $P_{k'_1,k'_2}\in\PP_{k'_1,k'_2}^G(\X_\b)$ is unique. Moreover 
\begin{equation}\label{eq:MP}
 \M(P)=\sum \M(P_{k_1',k_2'}).
\end{equation}
We denote this decomposition,
\be\label{jP}
\begin{array}{rcl}
\j : \TP^G_k(\X_\b)&\longto& \lt(\PP_{k'_1,k'_2}^G(\X_\b)\rt)_{(k'_1,k'_2)\in D_k},\\
P&\longmapsto& \j_{k'_1,k'_2}P = P_{k'_1,k'_2}.
\end{array}
\ee
As a consequence of the deformation theorem~\cite[Corollary~1.3]{White1999-1}, the group $\TP^G_k(\X_\b)$ is dense in $\FF^G_k(\X_\b)$. For $A\in \FF^G_k(\X_\b)$, we define 
\[
\Mt(A):=\inf \{\liminf \M(P_j) : P_j\in\TP^G_k(\X_\b),\ P_j\to A\}.
\]
There holds, 
\be\label{MtA<=CMA}
\M(A)\le \Mt(A)\le C(k,n_1,n_2) \M(A),
\ee
for some constant $C(k,n_1,n_2)\ge1$. The left inequality follows directly from the definitions as $\TP_k^G(\X_\b)\sub\PP_k^G(\X_\b)$. The  right inequality holds true for polyhedral chains and then extends to all chains by lower semicontinuity of $\Mt$. In fact, using the Cauchy--Binet formula, the optimal constant is
\[
C(k,n_1,n_2)=\sqrt m,
\]
where $m$ is the cardinal of $D_k$, that is, 
\[
m=
\begin{cases}
1+\min(k_,n_1,n_2)&\text{if }k\le\max(n_1,n_2),\\
\quad1+n-k&\text{if }k>\max(n_1,n_2).
\end{cases}
\]
In particular $C(k,n_1,n_2)\le\sqrt{\min(k,n_1,n_2)+1}$.

We define the corresponding flat norm as, 
\[
\Ft(A):=\inf \{\Mt(B)+\Mt(C):B\in\FF^G_k(\X_\b),\,C\in\FF^G_{k+1}(\X_\b),\, A=B+ \partial C\}.
\]
With the convention $C(n+1,n_1,n_2)=1$, we have
\[
\F(A)\le\Ft(A)\le\max(C(k,n_1,n_2),C(k+1,n_1,n_2))\,\F(A)\qquad\text{for every }A\in\FF^G_k(\R^n),
\]
so that $\Ft$ is a norm on $\FF^G_k(\X_\b)$, equivalent to $\F$.\\
Notice that for $k=|\b|$, $\TP^G_{|\b|}(\X_\b)$ is $\M$-dense in $\FF^G_{|\b|}(\X_\b)$ and $\Ft=\Mt=\F=\M$. For $k=0$, $\TP^G_0(\X_\b)=\PP^G_0(\X_\b)$ and $\Mt=\M$.

\subsection{Partial boundary operators and $\Fw$-norm}
Let us now introduce the \emph{partial} boundary operators $\pt_1$, $\pt_2$. For $P\in \PP_{k_1,k_2}^G(\X_\b)$ of the form~\eqref{k1k2polychain} we set,
\[
\pt_1 P:=\sum g_i (\pt p^1_i)\t p^2_i,\qquad\qquad \pt_2 P:=(-1)^{k_1}  \sum g_i p^1_i\t (\pt p^2_i).
\]
This defines group morphisms,
\[
\pt_1:\PP_{k_1,k_2}^G(\X_\b)\to  \PP^G_{k_1-1,k_2}(\X_\b),\qquad\qquad\pt_2:\PP_{k_1,k_2}^G(\X_\b)\to  \PP^G_{k_1,k_2-1}(\X_\b),
\] 
which satisfy the relations, 
\[
\pt_1^2=0,\qquad\quad\pt_2^2=0,\qquad\quad \pt_2\pt_1=-\pt_1\pt_2\qquad\text{and}\qquad\pt=\pt_1+\pt_2.
\]
With these operators we build a \emph{tensor flat norm} $\Fw$ defined for $P\in\PP_{k_1,k_2}^G(\X_\b)$ by 
\begin{multline}\label{WwP}
\Fw(P):=\inf\lt\{\M(Q^{0,0})+\M(Q^{1,0})+\M(Q^{0,1})+\M(Q^{1,1}) : \rt.\\
\lt. Q^{i_1,i_2}\in \PP^G_{k_1+i_1,k_2+i_2}(\X_\b),\ P=Q^{0,0}+\pt_1Q^{1,0}+\pt_2Q^{0,1}+\pt_1\pt_2Q^{1,1}\rt\}.
\end{multline}
At this point, we could continue as in~\cite{Fleming66} and first show that $\Fw$ is a norm on $\PP^G_{k_1,k_2}(\X_\b)$ before defining the groups of $(k_1,k_2)$-chains by completion. However, we take a detour by first identifying the elements of $\PP^G_{k_1,k_2}(\X_\b)$ with flat chains on $\X_{\b^1}$.

We break the symmetry and interpret $\PP_{k_1,k_2}^G(\X_\b)$ as the tensor product $\PP^\Z_{k_1}(\X_{\b^1})\ot\PP^G_{k_2}(\X_{\b^2})$ and in fact as\footnote{For a better readability, we sometimes write $\PP_{k_1}(\X_\b,G)$ for $\PP_{k_1}^{G}(\X_\b)$.} $\PP_{k_1}(\X_{\b^1},G'_{k_2})$ where here and below $G'_{k_2}:=\PP^G_{k_2}(\X_{\b^2})$. For $P\in\PP_{k_1,k_2}^G(\X_\b)$ with representation of the form~\eqref{k1k2polychain}, we denote $\i P$ the corresponding element of $\PP_{k_1}(\X_{\b^1},G'_{k_2})$: 
\begin{equation}\label{def:i}
\i P:=\sum [g_i p^2_i] p^1_i,
\end{equation}
the terms $g_i p^2_i$ being understood as coefficients in $G'_{k_2}$.\\
Let us denote $G^*_{k_2}:=\FF^G_{k_2}(\X_{\b^2})$ and let us endow the group $(G^*_{k_2},+)$ with the $\F$-norm. Since $G'_{k_2}\sub G^*_{k_2}$, we have
 \[\i P\in\PP_{k_1}(\X_{\b^1},G^*_{k_2})\] 
and now $(G^*_{k_2},+,\F)$ is a legit complete Abelian normed group. There holds $\M(\i P)\le\M(P)$ but equality does not hold in general, except in the limit case $k_2=|\b^2|$. We recover a favorable behavior when passing to the flat norms.

\begin{proposition}\label{prop_Ww=W}
Let $P\in\PP^G_{k_1,k_2}(\X_\b)$, there hold,
\begin{enumerate}[(i)]
\item $\M(\i P)=\inf\ds \M(Q^{0,0})+\M(Q^{0,1})$,\\
where the infimum is over the decompositions $P=Q^{0,0}+\pt_2 Q^{0,1}$ with $Q^{0,0}\in\PP_{k_1,k_2}^G(\X_\b)$, $Q^{0,1}\in\PP_{k_1,k_2+1}^G(\X_\b)$.
\item $\F(\i P)=\Fw(P)$.
\end{enumerate}
\end{proposition}

\begin{proof} We first prove~(i). By definition,
\[
\M(\i P)=\inf\lt\{\sum \M(p^1_j)\F(P^2_j)\rt\},
\]
where the infimum runs over all the representations $P=\sum p^1_j\we P^2_j$, with $p^1_j$, $k_1$-cell in $\X_{\b^1}$ and $P^2_j\in G'_{k_2}$. Then, by definition of $\F$ in $G^*_{k_2}$, 
\[
\M(\i P)=\inf\lt\{\sum \M(p^1_j)(\M(Q^2_j)+\M(R^2_j))\rt\},
\]
where now the infimum runs over the representations $P=\sum p^1_j\t (Q^2_j+\pt R^2_j)$ with $p^1_j$, $k_1$-cell in $\X_{\b^1}$, $Q^2_j\in G'_{k_2}$ and $R^2_j\in G'_{k_2+1}$. Using the obvious identities $\M(p^1_j)\M(Q^2_j)=\M(p^1_j\t Q^2_j)$, $ \M(p^1_j)\M(R^2_j)=\M(p^1_j\t R^2_j)$, we get,
\[
\M(\i P)=\inf\lt\{\sum \M(S_j)+\M(T_j)\rt\},
\]
where $S_j\in \PP_{k_1,k_2}^G(\X_\b)$, $T_j\in\PP_{k_1,k_2+1}^G(\X_\b)$ are such that $P=\sum S_j+\pt_2 T_j$. Eventually setting $Q^{0,0}:=\sum S_j$ and $Q^{0,1}:=\sum T_j$, we have established the identity~(i).\medskip

\noindent 
Let us prove the second point. By definition,
\[
\F(\i P)=\inf\lt\{\M(P^0)+\M(P^1) : P^{i_1}\in\PP_{k_1+i_1}(\X_{\b^1},G'_{k_2})\text{ for }i_1=0,1,\, \i P=P^0+\pt P^1\rt\},
\]
and since  $\i$ is a bijection and $\i (\pt_1 P)=\pt(\i P)$, this rewrites as 
\[
\F(\i P)=\inf\lt\{\M(\i Q^0)+\M(\i Q^1) : Q^{i_1}\in\PP_{k_1+i_1,k_2}^G(\X_\b)\text{ for }i_1=0,1,,\, P=Q^0+\pt_1 Q^1\rt\}.
\]
The result then follows by applying point (i) to $Q^0$ and $Q^1$.
\end{proof}

\subsection{Groups of tensor flat chains}
As a consequence of Proposition~\ref{prop_Ww=W}, $\Fw$ is a norm on $\PP_{k_1,k_2}^G(\X_\b)$. The group of $(k_1,k_2)$-chains is defined as the completion of $(\PP_{k_1,k_2}^G(\X_\b),\Fw)$ and is denoted $\FF^G_{k_1,k_2}(\X_\b)$.  For pairs $(k_1,k_2)\in\Z$ with $\min(k_1,k_2)<0$ or $\max(k_1-|\b^1|,k_2-|\b^2|)>0$ we set by convention $\FF^G_{k_1,k_2}(\X_\b)=\{0\}$. Let $P\in \PP_{k_1,k_2}^G(\X_\b)$.  Given a decomposition, 
\be\label{Pdecomp}
P=Q^{0,0}+\pt_1Q^{1,0}+\pt_2Q^{0,1}+\pt_1\pt_2Q^{1,1},
\ee
we have $\pt_1P=\pt_1Q^{0,0}+\pt_1\pt_2Q^{0,1}$ and $\pt_2P=\pt_2Q^{0,0}-\pt_1\pt_2 Q^{1,0}$, so that $\Fw(\pt_1P)\le \M(Q^{0,0})+\M(Q^{0,1})$ and $\Fw(\pt_2P)\le \M(Q^{0,0})+\M(Q^{1,0})$. Optimizing over the decompositions of $P$, we obtain
\[
\Fw(\pt_lP)\le\Fw(P)\qquad\text{for }l=1,2.
\]
With these inequalities, we can extend the partial differential operators $\pt_1$, $\pt_2$ by continuity on $\FF^G_{k_1,k_2}(\X_\b)$. Besides, 
\[
\pt_1: \FF^G_{k_1,k_2}(\X_\b)\longto\FF^G_{k_1-1,k_2}(\X_\b)\quad\qquad\text{and}\qquad\quad\pt_2:\FF^G_{k_1,k_2}(\X_\b)\longto\FF^G_{k_1,k_2-1}(\X_\b)
\]
are 1-Lipschitz group morphisms. Moreover
\[
\pt_1^2=0,\qquad\qquad\pt_2^2=0,\qquad\qquad\pt_2\pt_1=-\pt_1\pt_2.
\]
 In summary, the family $(\FF^G_{k_1,k_2}(\X_\b))_{k_1,k_2}$ forms a two-dimensional chain complex of normed groups: 
\[
\begin{array}{ccccc} 
                 &\text{\scriptsize$\pt_2$}\Big\dw                              &                       & \text{\scriptsize$\pt_2$}\Big\dw& \smallskip\\
\cdots\ \st{\pt_1}\longto& \FF^G_{k_1,k_2}(\X_\b)  &\st{\pt_1}\longto& \FF^G_{k_1-1,k_2}(\X_\b) &\st{\pt_1}\longto\ \cdots \\
\\
                 &\text{\scriptsize$\pt_2$}\Big\dw                              &                       & \text{\scriptsize$\pt_2$}\Big\dw& \smallskip\\
\cdots\ \st{\pt_1}\longto& \FF^G_{k_1,k_2-1}(\X_\b)  &\st{\pt_1}\longto& \FF^G_{k_1-1,k_2-1}(\X_\b) &\st{\pt_1}\longto\ \cdots \\
\\
                 &\text{\scriptsize$\pt_2$}\Big\dw                              &                       & \text{\scriptsize$\pt_2$}\Big\dw& \smallskip\\
\end{array}
\]

As for flat chains, we  say that $A\in\FF^G_{k_1,k_2}(\X_\b)$ is supported in the closed set $S$ if for every neighborhood $S
\sub U\sub\X_\b$, there exists a sequence $P_j\in\PP_{k_1,k_2}^G(\X_\b)$  with each $P_j$ supported in $U$ such that $P_j\to A$ (from now on for $A_j,A\in\FF^G_{k_1,k_2}(\X_\b)$,  $A_j\to A$ means that $A_j$ converges to $A$ in $\Fw$-norm, that is $\Fw(A_j-A)\to0$).\medskip

When $\{|\b^1|,|\b^2|\}=\{0,|\b|\}$, $\FF^G_{k_1,k_2}(\X_\b)$ and $\FF^G_k(\X_\b)$ are exactly the same groups by construction. In the other cases, any element $P\in\PP^G_{k_1,k_2}(\X_\b)$ can be considered either as an element of $\FF^G_k(\X_\b)$ or as an element of  $\FF^G_{k_1,k_2}(\X_\b)$. For clarity, in the later case we use from now on the``wedge product'' notation, 
\[
P=\sum g_i p^1_i\we p^2_i,
\]
in place of the ``cross product''~\eqref{k1k2polychain} even though these expressions refer to the same object.

\begin{remark}\label{rem_on_tfc}~
\medskip

\noindent
(a)  If $\b^2=\void$, $\FF^G_{k,0}(\X_\b)$ is just $\FF^G_k(\X_\b)$ and on these groups $\Fw=\F$, $\pt_1=\pt$ and $\pt_2=0$. Symmetrically, if $\b^1=\void$, $\FF^G_{0,k}(\X_\b)=\FF^G_k(\X_\b)$ and $\pt_2=\pt$, $\pt_1=0$.\medskip

\noindent
(b) In the extreme case $(k_1,k_2)=(|\b^1|,|\b^2|)$, there holds $\FF^G_{k_1,k_2}(\X_\b)=\FF^G_k(\X_\b)=L^1(\X_\b,G)$ with $\Fw=\F=\M$. \medskip

\noindent
(c)  In the other cases, tensor flat chains are not flat chains in general. For instance let $n_1=n_2=1$ and let us consider the triangle 
\[
T=\{(x^1,x^2)\in\R^2:x^1,x^2\ge0, x^1+x^2\le1\}.
\]
As a current, $\pt_1\lb T\rb$ is a measure supported on the frontier $\frt T$ of $T$ but it is oriented along $e_2$ which is not tangent to $\frt T$ on the segment $\{(t,1-t):0\le t\le 1\}$, so $\pt_1\lb T\rb$ is not a flat chain. This example is detailed in the setting of general complete Abelian normed groups $G\ne\{0\}$ in the proof of~\cite[Proposition~A.5]{GM_comp}.\medskip

\noindent
(d)  Another phenomenon arises from the second order term $\pt_1\pt_2 Q^{1,1}$ in the decomposition~\eqref{Pdecomp}. Let us consider again $n_1=n_2=1$ and $G=\Z$. Let $S\sub\R^2$ be a (non degenerate) closed square with sides parallel to the axes and let us set $P:=\pt_1\pt_2\lb S\rb\in\FF^\Z_{0,0}(\R^2)$. We easily see that $P=-\lb a\rb+\lb b\rb-\lb c\rb+\lb d\rb$ where $a,b,c,d$ are the vertices of $S$  labeled counterclockwise with $a$ at the bottom left corner. Denoting $\ell$ the side-length of $S$, on the one hand we have of course $\Fw(P)\le \M(\lb S\rb)=\ell^2$ and in fact this inequality is an identity if $\ell\le 2$. On the other hand, $\F(P)=2\ell$ if $\ell\le 2$. \\
Now, let $S_j$ be a sequence of such squares with side lengths $\ell_j=1/j$ for $j\ge1$ and such that $d(S_i,S_j)\ge2\ell_i$ for $1\le i<j$. Let us denote $P_j:=\pt_1\pt_2\lb S_j\rb$. Since $\Fw(P_j)=1/j^2$, the sequence $A_j:=\sum_{i\le j} P_i$ converges in $\Fw$-norm towards some $(0,0)$-chain $A_\oo$ which identifies with a $0$-current. On the contrary, for $j\ge1$ and $m\ge2j$,
\[
\F(A_m-A_j)=2\sum_{r=j+1}^{m}\dfrac1r\, \ge\, 2\sum_{r=j+1}^{2j}\dfrac1r \sim 2\ln 2\quad \text{ as }j\up\oo.
\]
Hence $A_j$ does not admit any Cauchy subsequence in $(\FF^\Z_0(\R^2),\F)$. Moreover, as a current, $A_\oo$ is not a flat chain.\footnote{Indeed, if $A_\oo$ were a flat chain we would have $\sum_{j\ge1}|\vhi(a_j)-\vhi(b_j)+\vhi(c_j)-\vhi(d_j)|<\oo$ for every compactly supported Lipschitz function $\vhi$, where $(a_j,b_j,c_j,d_j)$ are the vertices of $S_j$ listed counterclockwise. Setting $\vhi(x):=\sum_{j\ge1} (\ell_j-|x-a_j|)_+$ we obtain the contradiction $\sum_{j\ge1}1/j<\oo$.}\medskip

(e) However, as mentioned in the introduction, we establish a positive result, namely: the groups of normal tensor chains (defined further on) identify with subgroups of normal chains, see~\cite[Theorem~A.3]{GM_comp}.\medskip

(f) The basic theory of tensor chains is similar to its counterpart for flat chains, however some results do not generalize. For instance, if $A$ is a flat chain, we can write $A=B+\pt C$ for some \emph{finite mass} flat chains $B$, $C$ and we have the obvious implication
\be\label{rem_on_tfc_1} 
\M(A)<\oo\ \implies\   C\text{ is normal}.
\ee 
In the case of a tensor chain $A'$, we can still write (see Proposition \ref{prop_Wwalternative} below) $A'=B'_{0,0}+\pt_1B'_{1,0}+\pt_2B'_{0,1}+\pt_1\pt_2B'_{1,1}$ for some finite mass tensor chains $B'_{i_1,i_2}$ but for the corresponding $\Mw$-mass (introduced in Subsection~\ref{Ss_Mtfc} below) the additional information $\Mw(A')<\oo$ does not imply that $\pt_1B'_{1,0}$, $\pt_2B'_{0,1}$ or $\pt_1\pt_2B'_{1,1}$ have finite $\Mw$-mass.

This seems innocent but~\eqref{rem_on_tfc_1} is used to establish that the spaces of flat chains $\FF^\R_k(\R^n)$, as defined above with $G=\R$, are the same as the spaces of real valued flat chains defined as subspaces of currents in~\cite{FedFlem60,Federer} (namely the closure of the space of polyhedral currents with respect to the dual of the norm defined by $\|\om\|_{\F}:=\max(\|\om\|_\oo,\|d\om\|_\oo)$ for $\om$ smooth and compactly supported $k$-form on $\R^n$).\footnote{See~\cite[Section~2.1.12]{Federer}. Only a partial result is established there but the proof can be easily completed.} Here, it is not clear whether the spaces of tensor chains identify with the closures of the spaces of tensor polyhedral currents with respect to the dual of the norm defined (with the obvious definition of the partial exterior derivatives $d_1,d_2$) by $\|\om\|_{\Fw}:=\max(\|\om\|_\oo,\|d_1\om\|_\oo,\|d_2\om\|_\oo,\|d_1d_2\om\|_\oo)$ .
\end{remark}

\subsection{Orthogonal projection of tensor chains}
We only consider orthogonal projections of tensor chains on affine subspaces $L\sub\X_\b$ which write as $L=L^1+L^2$ with $L^l$ affine $m_l$-plane of $\X^{\b^l}$ for $l=1,2$. Denoting $\pi$ the orthogonal projection on such $L$ and $\pi^l$ the orthogonal projection on $L^l$ for $l=1,2$, we have for $P\in\PP_{k_1,k_2}^G(\X_\b)$ with representation~\eqref{k1k2polychain},
\[
\pi\pf P=\sum g_i (\pi^1\pf p^1_i)\we (\pi^2\pf p^2_i).
\]
Consequently, 
\[
\pt_1\pi\pf P=\sum g_i (\pt \pi^1\pf p^1_i)\we (\pi^2\pf p^2_i)= \sum g_i  (\pi^1\pf \pt p^1_i)\we (\pi^2\pf p^2_i)= \pi\pf \pt_1P.
\]
Similarly, $\pt_2\pi\pf P=\pi\pf \pt_2P$. It follows that for a representation of $P$ of the form~\eqref{Pdecomp}, 
\[
\pi\pf P=\pi\pf Q^{0,0}+\pt_1\pi\pf Q^{1,0}+\pt_2\pi\pf Q^{0,1}+\pt_1\pt_2\pi\pf Q^{1,1}.
\]
Observe that for a tensor polyhedral chain $Q$ there holds $\M(\pi\pf Q)\le\M(Q)$. Using this with $Q=Q^{i_1,i_2}$ for $i_1,i_2\in\{0,1\}$ and optimizing over the decompositions of $P$, we get $\Fw(\pi\pf P)\le \Fw(P)$. We can then extend the operator $\pi\pf$ by continuity.

\begin{proposition}\label{prop_pipf}With the above notation the orthogonal projection on $L=L^1+L^2$ with $L^l$ affine $m_l$-plane of $\X^{\b^l}$ for $l=1,2$ extends as a continuous group morphism $\pi\pf :\FF^G_{k_1,k_2}(\X^\b)\to \FF^G_{k_1,k_2}(L)$. Moreover, for $A\in\FF^G_{k_1,k_2}(\X^\b)$, there hold
\[
\pt_l\pi\pf A=\pi\pf \pt_lA\quad\text{for }l=1,2,\qquad\quad\Fw(\pi\pf A)\le\Fw(A).
\]
\end{proposition}

\subsection{Mass of tensor chains}\label{Ss_Mtfc}
We are now almost ready to define the mass of a tensor chain, see~\eqref{def_Mw} after the proof of the next lemma. Later on we consider restrictions to Borel sets of tensor chains of finite mass. This analysis is almost identical to the one of~\cite{Fleming66} but we reproduce it in details because on the one hand the necessary adaptations might not be completely obvious and on the other hand the estimate~\eqref{lem2.1_1} of the next lemma is used in~\cite{GM_comp}. As in the case of classical chains, we start with restriction on intervals.
\begin{definition}
We call intervals of $\X_\b$ the sets of the form $I_1\t I_2\t \dots\t I_n$ where for $j\in\b$, $I_j\sub\R$ is any interval and for $j\in\ov\b$, $I_j=\{0\}$.
\end{definition}

\begin{lemma}[{Counterpart of~\cite[Lemma~2.1 \& Theorem~2.3]{Fleming66}}] \label{lem2.1}~
\begin{enumerate}[(i)]
\item Let $P_j\in \PP_{k_1,k_2}^G(\X_\b)$ such that $\sum \Fw(P_j)<\oo$, then for every interval $I\sub\X_\b$, 
there holds
\[
\sum \Fw(P_j\restr (x+I))<\oo\qquad\text{for almost every }x\in\X_\b.
\] 
The intervals $x+I$ for which $\sum \Fw(P_j\restr (x+I))<\oo$ are called non-exceptional with respect to the sequence $P_j$.\\
More precisely, for any interval $I$ of $\X_\b$ and any Cartesian product $\om=\om_1\t \om_2\t\cdots\t\om_n\sub\X^\b$ such that for $i\in\b$, $\om_i\sub\R$ is measurable with finite length (and $\om_i=\{0\}$ for $i\in\ov\b$), there holds
\be\label{lem2.1_1}
\int_\om\sum\Fw\lt(P_j\restr(x+I)\rt)\,d\h^{|\b|}(x)\le c_\om\sum\Fw(P_j)<\oo,
\ee
where $c_\om\ge0$ depends on $\om$.
\item The mapping $P\mapsto\M(P)$ is lower semi-continuous in $(\PP_{k_1,k_2}^G(\X_\b),\Fw)$.
\end{enumerate}
\end{lemma}
\begin{proof}
Without loss of generality, we assume that $\b=\{1,\dots,n\}$. We proceed exactly as in~\cite{Fleming66}.\medskip

\noindent
\textit{Proof of~(i).} We start with $I$ being a half-space. Lets $i\in\{1,\dots,n\}$. For $s\in\R$ we denote,
\[
H_i(s):=\lt\{x\in \R^n: x_1<s\rt\}.
\]
 Let $P_j\in \PP_{k_1,k_2}^G(\R^n)$ such that $\sum \Fw(P_j)<\oo$. By definition of $\Fw$, for every $j\ge 1$, we can write $P_j=Q_j^{0,0}+\pt_1Q_j^{1,0}+\pt_2Q_j^{0,1}+\pt_1\pt_2Q_j^{1,1}$ with $Q_j^{i_1,i_2}\in \PP^G_{k_1+i_1,k_2+i_2}(\R^n)$ such that $\sum_{i_1,i_2}\M(Q_j^{i_1,i_2})\le2\Fw(P_j)$. Consequently,
\[
\sum_j\sum_{i_1,i_2}\M(Q_j^{i_1,i_2})\le2\sum_j\Fw(P_j).
\] 
Next, there holds for almost every $s\in\R$ and $j\ge1$,
\begin{multline*}
P_j\restr H_i(s)=Q_j^{0,0}\restr H_i(s)+\pt_1(Q_j^{1,0}\restr H_i(s))+\pt_2(Q_j^{0,1}\restr H_i(s)) +\pt_1\pt_2(Q_j^{1,1}\restr H_i(s))+R_j(s).
\end{multline*}
where the remaining term decomposes as $R_j(s):=S_j^{1,0}(s)+S_j^{0,1}(s)+S_j^{1,1}(s)$ with,
\begin{align*}
S_j^{1,0}(s) &=  (\pt_1Q_j^{1,0})\restr H_i(s)  - \pt_1(Q_j^{1,0}\restr H_i(s)), \\
S_j^{0,1}(s)&=  (\pt_2Q_j^{0,1})\restr H_i(s)  - \pt_2(Q_j^{0,1}\restr H_i(s)),\\
S_j^{1,1}(s) &=  (\pt_1\pt_2Q_j^{1,1})\restr H_i(s)  - \pt_1\pt_2(Q_j^{1,1}\restr H_i(s)).
\end{align*}
We have,
\[
\sum\Fw(P_j\restr H_i(s)-R_j(s))\le \sum_{i_1,i_2\in\{0,1\}}\sum_j \M\lt(Q_j^{i_1,i_2}\restr H_i(s)\rt).
\]
Since $\M(Q_j^{i_1,i_2}\restr H_i(s) )\le \M(Q_j^{i_1,i_2}))$ for every $j\ge1$ and $i_1,i_2\in\{0,1\}$, we deduce,
\be\label{proof_lem2.1_1}
\sum\Fw(P_j\restr H_i(s)-R_j(s))\le2\sum_j\Fw(P_j)<\oo.
\ee

Let us treat the remaining terms $R_j(s)$. We assume that $i\in\a$ (the other case is similar). In this case, $\pt_2$ commutes with the restriction on $H_i(s)$, hence $S_j^{0,1}(s)=0$ and $S_j^{1,1}(s) $ rewrites as $S_j^{1,1}(s) =\pt_2 T_j(s)$, with
\[
T_j(s) =  -(\pt_1Q_j^{1,1})\restr H_i(s) +\pt_1(Q_j^{1,1}\restr H_i(s)).
\]
By Proposition~\ref{prop_divformula}, we have
\[
S_j^{1,0}(s) =\sigma Q^{1,0}_j\cap \X_{{\ov i}}(se_{i}),\qquad\qquad T_j(s) =\sigma' Q^{1,1}_j\cap \X_{\ov i}(se_i),
\] 
where the factors  $\sigma,\sigma'=\pm 1$ depend on the orientation conventions. From Proposition~\ref{prop_MWSl}, we have the estimates,
\[
\int_\R \M(S^{1,0}_j(s))\, ds \le \M(Q_j^{1,0}),\qquad\qquad\int_\R \M(T_j(s))\, ds \le \M(Q_j^{1,1}).
\]
Writing $R_j(s)=S^{1,0}_j(s)+\pt_2 T_j(s)$.  We deduce that
\be\label{proof_lem2.1_2}
\int_\R\sum\Fw(R_j(s))\, ds\le2\sum_j\Fw(P_j)<\oo.
\ee
Combining~\eqref{proof_lem2.1_1}\&\eqref{proof_lem2.1_2} we get that for every measurable set $\om_1\sub\R$ with finite length $\ell_1$, there holds
\be\label{proof_lem2.1_3}
\int_{\om_1}\sum\Fw\lt(P_j\restr H_i(s)\rt)\,ds\le2(1+\ell_1)\sum_j\Fw(P_j)<\oo.
\ee
The first point of the lemma in the case of an open half-space follows from~\eqref{proof_lem2.1_3} (notice that $x+H_i(0)=H_i(x_i)$). From $P\restr (\R^n\sm H_i(s))=P-P\restr H_i(s)$, we see that the result also holds for the closed half-space $\R^n\sm H_i(0)$. By symmetry the result holds for any closed or open half-space with boundary normal to some $e_i$. For a general interval $I$, writing $I$ as the intersection of at most $2n$ coordinate half-spaces and applying the result recursively, we obtain the estimate~\eqref{lem2.1_1}. This proves the first point.\medskip

\noindent
\textit{Proof of~(ii).}  Let $P\in\PP_{k_1,k_2}^G(\X_\b)$ and let $P_j\to P$ in $\Fw$-norm such that the sequence $\M(P_j)$ has a finite limit. Up to extraction, we assume moreover $\sum\Fw(P_i-P)<\oo$.\\
First we consider an interval  $I$ of $\X_\b$ which is non exceptional with respect to the sequence $P_j-P$  and is such that $P^0:=P\restr I$ is supported in some affine $(k_1+k_2)$-plane $L\sub\X_\b$ of the form $L=L^1+L^2$ (with $L^l$ affine $k_l$-plane of $\X^{\b^l}$ for $l=1,2$). Let us set $P_j^0:=\pi\pf (P_j\restr I)$ where $\pi$ is the orthogonal projection on $L$. The chain $P_j^0-P^0$ being of maximal dimension in $L$, there holds
\[
\M(P_j^0-P^0) = \Fw(P_j^0-P^0)\le \Fw((P_j-P)\restr I).
\] 
Since $I$ is non exceptional, we have $\sum\Fw((P_j-P)\restr I)<\oo$ and the right-hand side tends to $0$. We conclude that,
\be\label{proof_lem2.1_4}
\M(P^0)=\lim \M(P^0_j)\le \liminf \M(P_j\restr I).
\ee

Eventually, given $\eps>0$, there exists a finite set $I^1,\cdots,I^m$ of disjoint intervals of $\X_\b$ of the form above such that $\M(P\restr (\X_\b\sm \cup I^r))<\eps$. The result then follows from~\eqref{proof_lem2.1_4} applied to the $I^r$'s and the obvious fact that if $P$ and $P'$ have disjoint supports then $\M(P+P')=\M(P)+\M(P')$.
\end{proof}

We define the mass of $A\in\FF^G_{k_1,k_2}(\X_\b)$ in the same way as the mass of classical chains. Namely,
\be\label{def_Mw}
\Mw(A):=\inf\lt\{\liminf\M(P_j): P_j\in\PP_{k_1,k_2}^G(\X_\b),\, P_j\to A \text{ in }\Fw\text{-norm}\rt\}.
\ee
Thanks to Lemma~\ref{lem2.1}(ii), we have $\Mw(P)=\M(P)$ for $P\in \PP_{k_1,k_2}^G(\X_\b)$ and for such chains we use the two notations interchangeably.\\ 
The elements of $\FF^G_{k_1,k_2}(\X_\b)$ with finite mass form a subgroup denoted $\MM^G_{k_1,k_2}(\X_\b)$ and 
\[
(\MM^G_{k_1,k_2}(\X_\b),+,\Mw)\text{ is a complete normed group}.
\]
As in the case of flat chains, if $|\b|\ge1$ these groups do not form a chain complex
\begin{remark}
In the extreme case $(k_1,k_2)=(|\b^1|,|\b^2|)$, we have $\Mw=\Fw$ on $\FF^G_{k_1,k_2}(\X_\b)$ and  $\MM^G_{k_1,k_2}(\X_\b)=\FF^G_{k_1,k_2}(\X_\b)$. Recalling Remark~\ref{rem_on_tfc}(b) we have in fact $\MM^G_{k_1,k_2}(\X_\b)=\FF^G_{k_1,k_2}(\X_\b)=\MM^G_k(\X_\b)=\FF^G_k(\X_\b)=L^1(\X_\b,G)$ with $\Mw=\Fw=\M=\F$.

\end{remark}

Before continuing let us come back to the push-forwards by orthogonal projections and complete Proposition~\ref{prop_pipf}.
\begin{proposition}\label{prop_MpipfA}
Given a projection $\pi$ on $L=L^1+L^2\sub\X_\b$ as in Proposition~\ref{prop_pipf}, there holds $\Mw(\pi\pf A)\le\Mw(A)$ for every $A\in\FF^G_{k_1,k_2}(\X_\b)$.
\end{proposition}
\begin{proof}
The result holds true for $A\in\PP_{k_1,k_2}^G(\X_\b)$. The general case follows from the continuity of $\pi\pf$ in $(\FF^G_{k_1,k_2}(\X_\b),\Fw)$ and the definition of $\Mw$.
\end{proof}

Now that we have defined the mass, we can express $\Fw(A)$ with a formula similar to~\eqref{WwP}.
\begin{proposition}[{Counterpart of~\cite[Theorem~3.1]{Fleming66}}]
\label{prop_Wwalternative}
Let $A\in\FF^G_{k_1,k_2}(\X_\b)$, there holds 
\be\label{alternativeWw}
\Fw(A)=\inf \sum_{i_1,i_2\in\{0,1\}} \Mw(B^{i_1,i_2}),
\ee
 where the infimum runs over the decompositions
\be\label{Adecomp}
 A=B^{0,0}+\pt_1B^{1,0}+\pt_2B^{0,1}+\pt_1\pt_2B^{1,1},
 \ee
with $B^{i_1,i_2}\in\FF^G_{k_1+i_1,k_2+i_2}(\X_\b)$ for $i_1,i_2\in\{0,1\}$.
\end{proposition}

\begin{proof}Let $A\in\FF^G_{k_1,k_2}(\X_\b)$ and let us denote by $\wt\F(A)$ the right-hand side of~\eqref{alternativeWw}. 

Let us consider a decomposition of $A$ of the form~\eqref{Adecomp}. For $i_1,i_2\in\{0,1\}$, there exist sequences of tensor polyhedral chains $Q_j^{i_1,i_2}$ with $Q_j^{i_1,i_2}\to B^{i_1,i_2}$ in $\Fw$-norm and $\Mw(Q_j^{i_1,i_2})\to \Mw(B^{i_1,i_2})$. The sequence 
\[
P_j:=Q_j^{0,0}+\pt_1Q_j^{1,0}+\pt_2Q_j^{0,1}+\pt_1\pt_2Q_j^{1,1},
\]
converges to $A$ in $\Fw$-norm. Hence 
\[
\Fw(A)\le\liminf\Fw(P_j)\le\lim_j\sum_{i_1,i_2\in\{0,1\}}\Mw(Q_j^{i_1,i_2})=\sum_{i_1,i_2\in\{0,1\}}\Mw(B^{i_1,i_2}).
\]
We deduce that $\Fw(A)\le\wt\F(A)$.\medskip

Let us prove the converse inequality. Let $P_j\in \PP_{k_1,k_2}^G(\X_\b)$ converging rapidly to $A$ (that is $\sum\Fw(P_j-A)<\oo$) and let $\eps_j>0$ converging to $0$.  For $j\ge1$, there exist decompositions
\begin{align*}
P_j&=Q_j^{0,0}+\pt_1Q_j^{1,0}+\pt_2Q_j^{0,1}+\pt_1\pt_2Q_j^{1,1},\\
P_{j+1}-P_j&=R_j^{0,0}+\pt_1R_j^{1,0}+\pt_2R_j^{0,1}+\pt_1\pt_2R_j^{1,1},
\end{align*}
with  
\begin{align}
\label{proof_prop_Wwalternative_1}
\sum_{i_1,i_2}\Mw(Q_j^{i_1,i_2})<\Fw(P_j)+\eps_j\quad&\text{for }j\ge 1,\\
\label{proof_prop_Wwalternative_2}
\sum_j\Mw(R_j^{i_1,i_2})<\oo\quad&\text{for }i_1,i_2\in\{0,1\}.
\end{align}
For $i_1,i_2\in\{0,1\}$, we set $
B^{i_1,i_2}_j:=Q^{i_1,i_2}_j+\sum_{i\ge j} R^{i_1,i_2}_i$.
There holds for every $j\ge1$, 
\be\label{proof_prop_Wwalternative_3}
A=B_j^{0,0}+\pt_1B_j^{1,0}+\pt_2B_j^{0,1}+\pt_1\pt_2B_j^{1,1},
\ee
so that  we obtain,
\[
\wt\F(A)\stackrel{\eqref{proof_prop_Wwalternative_3}}{\le}\sum_{i_1,i_2}\Mw(B_j^{i_1,i_2})\le\sum_{i_1,i_2}\Mw(Q_j^{i_1,i_2})+\sum_{i\ge j}\Big(\sum_{i_1,i_2}\Mw(R_i^{i_1,i_2})\Big)\stackrel{\eqref{proof_prop_Wwalternative_1}\&\eqref{proof_prop_Wwalternative_2}}{\le}\Fw(P_j)+\eps'_j,
\]
with $\eps'_j\to0$. Passing to the limit, we get $\wt\F(A)\le\Fw(A)$.
\end{proof}

\subsection{The operators $\i$ and $\j$}\label{Ssij}
We introduce here operations that involve both tensor and classical chains.\medskip

Let us first extend the morphism $\i$. Recall that $\i:\PP_{k_1,k_2}^G(\X_\b)\to \PP_{k_1}\lt(\X_{\b^1},\PP^G_{k_2}(\X_{\b^2})\rt)$  is a group isomorphism which is moreover an isometry by the identity of Proposition~\ref{prop_Ww=W}(ii). By density of $\PP^G_{k_2}(\X_{\b^2})$ in $\FF^G_{k_2}(\X_{\b^2})$ and a diagonal extraction argument we deduce that $\i$ extends as the following group isomorphism which is again an isometry.
\[
\i : \lt( \FF^G_{k_1,k_2}(\X_\b),+,\Fw\rt)\longto \lt(\FF_{k_1}\lt(\X_{\b^1},\FF^G_{k_2}(\X_{\b^2})\rt),+,\F\rt).
\] 
Besides, by continuity, we still have the identity $\i \pt_1A= \pt \i A$ for $A\in\FF^G_{k_1,k_2}(\X_\b)$. Moreover 
\be\label{eq:iM}
\M(\i A)\le\Mw(A).
\ee

We now come back to the operator $\j$ of~\eqref{jP}. Let $P\in\TP^G_k(\X_\b)$ and let us consider a decomposition  $P=Q+\pt R$ with $Q\in\TP^G_k(\X_\b)$, $R\in\TP^G_{k+1}(\X_\b)$. By identification, we have for $(k'_1,k'_2)\in D_k$, 
\[
P_{k'_1,k'_2}=\j_{k'_1,k'_2}P=\j_{k'_1,k'_2}Q+\j_{k'_1,k'_2}(\pt R)=Q_{k'_1,k'_2}+\pt_1R_{k'_1+1,k'_2}+\pt_2R_{k'_1,k'_2+1}. 
\]
We deduce $\Fw(P_{k'_1,k'_2})\le\M(Q_{k'_1,k'_2})+\M(R_{k'_1+1,k'_2} )+\M(R_{k'_1,k'_2+1})$ (remark that $Q_{k'_1,k'_2}$, $R_{k'_1+1,k'_2}$ and $R_{k'_1,k'_2+1}$ are tensor polyhedral chains). Summing over $(k'_1,k'_2)\in D_k$ and optimizing, with respect to the decompositions of $P$, we get
\[
\sum_{(k'_1,k'_2)\in D_k}\Fw(P_{k'_1,k'_2})\le 2\Ft(P).
\]
This inequality and  the density of $\TP^G_k(\X_\b)$ in $\FF^G_k(\X_\b)$ allows to extend $\j$ on the latter. We obtain a continuous group morphism
\[
\begin{array}{rcl}
\j : \FF^G_k(\X_\b)&\longto& \lt(\FF^G_{k'_1,k'_2}(\X_\b)\rt)_{(k'_1,k'_2)\in D_k},\\
A&\longmapsto& \j_{k'_1,k'_2}A = A_{k'_1,k'_2}.
\end{array}
\]
Moreover, for $A\in\FF^G_k(\X_\b)$, we have the estimate
\be\label{estimjA}
\sum_{(k'_1,k'_2)\in D_k}\Fw(\j_{k'_1,k'_2}A)\le 2\Ft(A).
\ee
In particular given $A\in\FF^G_k(\X_\b)$ a sequence of polyhedral chains $P_j\in\PP^G_k(\X_\b)$ such that $P_j\to A$, we have  $\j_{k'_1,k'_2}P_j\to \j_{k'_1,k'_2}A$ for every $(k'_1,k'_2)\in D_k$. We deduce from~\eqref{eq:MP},~\eqref{MtA<=CMA} and the lower semicontinuity of $\Mw$ in $(\FF^G_{k_1',k_2'}(\X_\beta),\Fw)$ that, 
\be\label{estimjA_M}
\sum_{(k_1',k_2')\in D_k}\Mw(\j_{k'_1,k'_2}A)\le\Mt(A)\le C(k,n_1,n_2)\Mw(A).
\ee

For later use, we state the following relations.
\begin{proposition}\label{prop_pt_and_j}
For every $A\in\FF^G_k(\R^n)$ and every $(k'_1,k'_2)\in D_{k-1}$,
\[
\j_{k'_1,k'_2} (\pt A)= \pt_1 \j_{k'_1+1,k'_2}A+\pt_2 \j_{k'_1,k'_2+1}A.
\]
\end{proposition}
\begin{proof}
The identity holds true by identification for polyhedral $(k_1,k_2)$-chains. Recalling that by the deformation theorem of White~\cite{White1999-1}, $\PP^G_{k_1,k_2}(\R^n)$ is dense in $\FF^G_k(\R^n)$, the general case follows by continuity of $\j$ and of the (partial) boundary operators.
\end{proof}

\subsection{Tensor products of chains as tensor chains}
\label{Ss_tpofc}
Before introducing restrictions and slices of tensor chains, we pause to discuss the structure of the groups of tensor chains as (completion of) tensor products.

We have seen in Section~\ref{SGfc} that given $P_1\in\PP^\Z_{k_1}(\X_{\b^1})$ and $P_2\in\PP^G_{k_1}(\X_{\b^2})$, we can form a chain $P:=P^1\t P^2\in\PP^G_k(\X_\b)$. More precisely, $P$ is of the form~\eqref{k1k2polychain} and we have $P\in\PP^G_{k_1,k_2}(\X_\b)\sub\FF^G_{k_1,k_2}(\X_\b)$. Recall that when we consider $P$ as an element of $\FF^G_{k_1,k_2}(\X_\b)$ we write $P=P^1\we P^2$. We have obviously
\[
\i P=[P^2]P^1\ \in\PP_{k_1}\lt(\X_{\b^1},\FF^G_{k_2}(\X_{\b^2})\rt),
\]
where here $P^2$ is a coefficient. It follows that
\[
\Fw(P^1\we P^2)=\F(\i P)\le\F(P^1)\F(P^2).
\] 
We also have
\[
\Mw(P^1\we P^2)\le\M(P^1)\M(P^2).
\]
In general, these inequalities are not identities. Indeed, if $G=\Z/2\Z=\ds\lt\{0,\ov1\rt\}$ and $p^1$, $p^2$ are non zero polyhedral cells in $\X_{\b^1}$ and $\X_{\b^2}$ respectively, taking $P^1=2p^1$ and $P^2=\ov1p^2$, the chains $P^1$ and $P^2$ do not vanish but 
\be\label{2.1=0}
P^1\we P^2=(2p^1)\we(\ov1p^2)=p^1\we(2\cd\ov1p^2)=p^1\we(0p^2)=0.
\ee
Anyway, we can extend the tensor product by density as a continuous  $\Z$-bilinear mapping,
\be\label{wedge}
\begin{array}{rcl}
\FF^\Z_{k_1}(\X_{\b^1})\t\FF^G_{k_2}(\X_{\b^2})&\longto&\FF^G_{k_1,k_2}(\X_\b),\smallskip\\
(A^1,A^2)\qquad&\longmapsto&A^1\we A^2.
\end{array}
\ee
Moreover, there hold
\[
\label{ptA1A2}
\pt_1(A^1\we A^2)=(\pt A^1)\we A^2,\qquad\qquad\pt_2(A^1\we A^2)=(-1)^{k_1}A^1\we(\pt A^2),
\]
and 
\be
\label{estimA1A2}
\Mw(A^1\we A^2)\le\M(A^1)\M(A^2),\qquad\qquad\Fw(A^1\we A^2)\le\F(A^1)\F(A^2).
\ee
Taking finite sums, 
\[
\sum A^1_i\we A^2_i\ \in\FF^G_{k_1,k_2}(\X_\b),
\] 
we obtain the tensor product of the groups $\FF^\Z_{k_1}(\X_{\b^1})$ and $\FF^G_{k_2}(\X_{\b^2})$ that we denote $\FF^\Z_{k_1}(\X_{\b^1})\ot\FF^G_{k_2}(\X_{\b^2})$ and~\eqref{wedge} extends as  
\be\label{wedge2}
\begin{array}{rcl}
\phi:\FF^\Z_{k_1}(\X_{\b^1})\ot\FF^G_{k_2}(\X_{\b^2})&\longto&\FF^G_{k_1,k_2}(\X_\b),\smallskip\\
\sum A^1_i\otimes A^2_i\qquad&\longmapsto& \sum A^1_i\we A^2_i.
\end{array}
\ee
By convention we use the index $i$ for finite sums and the index $j$ for (at most) countable sums.
We observe that by density of tensor polyhedral chains, the image of $\phi$ is dense in $(\FF^G_{k_1,k_2}(\X_\b),\Fw)$. Let us first state more precisely this density result and give another expression for $\Fw(A)$. Apart from~\eqref{Adecomp_2} the rest of this subsection is not used in the proof of the main results and can be safely skipped.

\begin{proposition}\label{prop_proj_oo}
Let $A\in\FF^G_{k_1,k_2}(\X_\b)$, there exists (infinite) sequences $A_j^1\in\FF^\Z_{k_1}(\X_{\b^1})$, $A_j^2\in\FF^G_{k_2}(\X_{\b^2})$ such that, 
\be\label{Adecomp_2}
 A=\sum A_j^1\we A_j^2,
 \ee
with a normal convergence of the series in $\Fw$ norm. Moreover,
\be\label{alternativeWw_2}
\Fw(A)=\inf \sum\F(A_j^1)\F(A_j^2),
\ee
where the infimum runs over the decompositions~\eqref{Adecomp_2}.
\end{proposition}

\begin{proof}
The proof is reminiscent of the proof of Proposition \ref{prop_Wwalternative} and we use similar notation. Let $A\in\FF^G_{k_1,k_2}(\X_\b)$. We denote by $\wt\F(A)$ the right-hand side of \eqref{alternativeWw_2}. For the moment we do not know whether $A$ admits a decomposition of the form~\eqref{Adecomp_2}, if not we set $\wt\F(A):=+\oo$.\medskip

\noindent
Let us consider a decomposition of $A$ of the form~\eqref{Adecomp_2}. By the triangle inequality and~\eqref{estimA1A2}, we have
\[
\Fw(A)\le\sum\Fw(A_j^1\we A_j^2)\le\sum\F(A_j^1)\F(A_j^2), 
\]
and taking the infimum with respect to all the decompositions, we obtain,
\be\label{FwA<tFA}
\Fw(A)\le\wt\F(A).
\ee
Let us establish the converse inequality. Let  $P_j\in \PP_{k_1,k_2}^G(\X_\b)$ converging rapidly to $A$ and let $\eps_j>0$ converging to $0$.  From the proof of Proposition \ref{prop_Wwalternative}, there exist sequences of polyhedral tensor chains $Q_j^{i_1,i_2}$, $R_j^{i_1,i_2}$ for $i_1,i_2\in\{0,1\}$ such that,
\begin{align*}
P_j&=Q_j^{0,0}+\pt_1Q_j^{1,0}+\pt_2Q_j^{0,1}+\pt_1\pt_2Q_j^{1,1},\\
P_{j+1}-P_j&=R_j^{0,0}+\pt_1R_j^{1,0}+\pt_2R_j^{0,1}+\pt_1\pt_2R_j^{1,1},
\end{align*}
with  
\begin{align*}
\sum_{i_1,i_2}\M(Q_j^{i_1,i_2})\le\Fw(P_j)+\eps_j\quad&\text{for }j\ge1,\\
\label{proof_prop_Wwalternative_2}
\sum_j\M(R_j^{i_1,i_2})<\oo\quad&\text{for }i_1,i_2\in\{0,1\}.
\end{align*}
By construction, each tensor chain $Q=Q_j^{0,0}$ or $Q=R_j^{0,0}$ writes as a finite sum $\sum_i p^1_i\we (g_i p_i^2)$ with $g_i\in G$ and  $p^\ell_i$ oriented $k_\ell$-cell of $\X_{\beta^\ell}$ such that
\be\label{proof_proj_oo_0}
\M(Q)=\sum_i\M(p^1_i)\M(g_i p^2_i).
\ee
Denoting, $A_j:=Q_j^{0,0}+\sum_{r\ge j}R_r^{0,0}$, we see that
\[
A_j=\sum_s A_{j,s}^{1}\we  A_{j,s}^{2},
\] 
for some sequence of polyhedral chains $A_{j,s}^1\in\PP^\Z_{k_1}(\X_{\b^1})$, $A_{j,s}^2\in\PP^G_{k_2}(\X_{\b^2})$, with
\be\label{proof_proj_oo_1}
\sum_s\F(A_{j,s}^1)\F(A_{j,s}^2)\le\sum_s\M(A_{j,s}^1)\M(A_{j,s}^2)\le\M(Q_j^{0,0})+\sum_{r\ge j}\M(R_r^{0,0}).
\ee
Next, each chain $Q=Q_j^{1,0}$ or $Q=R_j^{1,0}$ also writes as a finite sum $\sum_i p^1_i\we (g_i p_i^2)$ with $g_i\in G$ and  $p^\ell_i$ oriented $(k_\ell+1)-$cell in $\X_{\b^\ell}$ such that~\eqref{proof_proj_oo_0} holds. Writing
\[
\pt_1Q=\sum_i(\pt p^1_i)\we (g_i p_i^2),
\] 
and defining $B_j:=\pt_1Q_j^{1,0}+\sum_{r\ge j}\pt_1R_r^{1,0}$, we see that $B_j$ writes as
\[
B_j=\sum_s(\pt B_{j,s}^{1})\we B_{j,s}^{2},
\]
for some sequence of polyhedral chains $B_{j,s}^1\in\PP^\Z_{k_1+1}(\X_{\b^1})$, $B_{j,s}^2\in\PP^G_{k_2}(\X_{\b^2})$ and we have the estimate,
\be\label{proof_proj_oo_2}
\sum_s\F(\pt B_{j,s}^1)\F(B_{j,s}^2)\le\sum_s\M(B_{j,s}^1)\M(B_{j,s}^2)\le\M(Q_j^{1,0})+\sum_{r\ge j}\M(R_r^{1,0}).
\ee
Similarly, we define $C_j:=\pt_2Q_j^{0,1}+\sum_{r\ge j}\pt_2R_r^{0,1}$ and $D_j:={\partial_1}\pt_2Q_j^{1,1}+\sum_{r\ge j}\pt_1\pt_2R_r^{1,1}$ which write as
\[
C_j=\sum_sC_{j,s}^{1}\we (\pt C_{j,s}^{2}),\qquad\qquad D_j=\sum_s(\pt D_{j,s}^{1})\we (\pt D_{j,s}^{2}),
\]
with the estimates,
\begin{align}
\label{proof_proj_oo_3}
\sum_s\F(C_{j,s}^1)\F(\pt C_{j,s}^2)&\le\M(Q_j^{0,1})+\sum_{r\ge j}\M(R_r^{0,1}),\\
\label{proof_proj_oo_4}
\sum_s\F(\pt D_{j,s}^1)\F(\pt D_{j,s}^2)&\le\M(Q_j^{1,1})+\sum_{r\ge j}\M(R_r^{1,1}).
\end{align}
By construction, there holds, for every $j\ge1$,
\begin{align*}
A&=A_j+B_j+C_j+D_j\\
&=\sum_sA_{j,s}^{1}\we A_{j,s}^{2}+\sum_s(\pt B_{j,s}^{1})\we B_{j,s}^{2}+(-1)^{k_1}\sum_sC_{j,s}^{1}\we (\pt C_{j,s}^{2})+(-1)^{k_1}\sum_s(\pt D_{j,s}^{1})\we (\pt D_{j,s}^{2}),
\end{align*}
which is a series of the form~\eqref{Adecomp_2}. Collecting~\eqref{proof_proj_oo_1}--\eqref{proof_proj_oo_4}, we obtain for $j\ge1$,
\[
\wt\F(A)\le\sum_{i_1,i_2\in\{0,1\}}\Big(\M(Q_j^{i_1,i_2})+\sum_{r\ge j}\M(R_r^{i_1,i_2})\Big)\le\Fw(A)+\eps'_j,
\]
with $\eps'_j\to0$ as $j\up\oo$. We conclude that $\wt\F(A)\le\Fw(A)$ and with~\eqref{FwA<tFA} we get~\eqref{alternativeWw_2}.
\end{proof}

In the context of Banach spaces, given two (say real) Banach spaces $(X,\|\cd\|_X)$ and $(Y,\|\cd\|_Y)$, the natural norms on their algebraic tensor product $X\ot Y$ are the so called cross norms which satisfy $\|x\ot y\|=\|x\|_X\|y\|_Y$ for every $x\in X$, $y\in Y$. These norms have been classified by Grothendieck in~\cite{Gro} (see~\cite[Section~3]{Pisier} or~\cite{Ryan} for an account in english). Among these norms the projective norm is defined for 
\be\label{rep_z}
z=\sum_i x_i\ot y_i\ \in X\ot Y,
\ee
by
\[
\pi(z):=\inf\sum_i\|x_i\|_X\|y_i\|_Y,
\]
where we take the infimum over all representations of $z$ as finite sums of the form~\eqref{rep_z}. With this definition, $\pi$ is obviously a seminorm on $X\ot Y$ and it turns out that it is also a norm. Moreover, by the triangle inequality, we have $\|\cd\|\le\pi$ for any cross norm $\|\cd\|$ so the projective norm is the largest natural norm and the space $X\wh\ot_\pi Y,$ obtained by taking the completion of $(X\ot Y,\pi)$ is the smallest natural completion of $X\ot Y$.

Let us extend these notions to the tensor product $\GG:=\GG_1\ot\GG_2$ of the two complete normed Abelian groups 
\[
\GG_1:=(\FF_{k_1}^\Z(\X_{\beta^1}),+,\F),\qquad\qquad\GG_2:=(\FF_{k_1}^G(\X_{\beta^2}),+,\F).
\]
More precisely, $\GG$ is the group of formal finite sums $\sum (A_i^1,A_i^2)$ of pairs $(A_i^1,A_i^2)\in\GG_1\t\GG_2$ quotiented by the relations 
\[
(A^1_a+A^1_b,A^2)=(A^1_a,A^2)+(A^1_b,A^2),\qquad\qquad(A^1,A^2_a+A^2_b)=(A^1,A^2_a)+(A^1,A^2_b),
\]
for every $A^l,A^l_a,A^l_b\in \GG_l$, $l\in\{1,2\}$. The coset of $\sum (A_i^1,A_i^2)$ is denoted $\sum A_i^1\ot A_i^2$.\\
Next, we say that a seminorm $\|\cd\|$ on $\GG$ is a cross seminorm if 
 \be\label{cross_norm}
 \|A^1\ot A^2\|\le\F(A^1)\F(A^2)\qquad\text{for every }A^1\in \GG_1,\, A^2\in\GG_2.
 \ee
By the second estimate of~\eqref{estimA1A2} the mapping defined by $\|\sum A_i^1\ot A_i^2\|_{\Fw}:=\Fw(\sum A_i^1\we A_i^2)$ defines a cross seminorm  on $\GG$ (notice that in view of~\eqref{2.1=0} it would be too optimistic to ask for equality in~\eqref{cross_norm}).\\ 
Next, we define the projective seminorm of $A\in\GG$ as 
  \be\label{proj_norm}
  \pi(A):=\inf \sum_i \F(A_i^1)\F(A_i^2),
 \ee
 where the infimum runs over the representation of $A$ as \emph{finite} sums $\sum_i A_i^1\ot A_i^2$. Again, by the triangle inequality, we have $\|\cd\|\le \pi$ for any cross seminorm $\|\cd\|$ on $\GG$. \\
 By Proposition~\ref{prop_proj_oo} the morphism of~\eqref{wedge2} satisfies
 \[
 \Fw(\phi(A))\le\pi(A)\qquad\text{for every }A\in\GG.
 \]
In particular  $\phi\equiv0$ on the subgroup
\[
\GG_0:=\lt\{A\in\GG:\pi(A)=0\rt\}.
\] 
Denoting $\HH:=\GG/\GG_0$ and setting $\pi(A+\GG_0):=\pi(A)$ for $A\in\GG$, we have that $(\HH,+,\pi)$ is an Abelian normed group and taking the quotient in~\eqref{wedge2}, we obtain a 1-Lipschitz group morphism 
\[
\begin{array}{rcl}
\wt\phi:(\HH,+,\pi)&\longto&(\FF^G_{k_1,k_2}(\X_\b),+,\Fw)\\
A+\GG_0&\longmapsto& \phi(A).
\end{array}
\]
Eventually, denoting $\GG_1\wh\ot_\pi \GG_2$ the completion of $(\GG/\GG_0,\pi)$ and extending $\wt\phi$ by continuity we obtain a 1-Lipschitz group morphism 
\[
\wh \phi:(\GG_1\wh\ot_\pi \GG_2,+,\pi)\longto(\FF^G_{k_1,k_2}(\X_\b),+,\Fw).
\]
In the introduction, we claimed that $\FF^G_{k_1,k_2}(\X_\b)$ could be viewed as the complement of $\GG_1\ot \GG_2$ with respect to some appropriate norm. This is justified by the following result up to the operation of taking the quotient of $\GG=\GG_1\ot \GG_2$ with respect to $\GG_0$. We do not know whether in general $\GG_0=\{0\}$ or equivalently whether $\pi$ is a norm on $\GG$. 
\begin{proposition}\label{prop_tensor}With the above notations, the mapping $\wh\phi$ is an isometric group isomorphism from $(\GG_1\wh\ot_\pi \GG_2,+,\pi)$ onto $(\FF^G_{k_1,k_2}(\X_\b),+,\Fw)$.

\end{proposition}
\begin{remark}
With this proposition, we see that the norm we choosed on $\FF_{k_1}^\Z(\X_{\b^1})\ot\FF_{k_2}^G(\X_{\b^2})$ is the largest among ``natural'' norms and that our definition of $\FF_{k_1,k_2}^G(\X_\b)$ corresponds to the smallest complete normed group that extends the definition of the product $(A^1,A^2)\mapsto A^1\t A^2$ to any pair of chains $(A^1,A^2)\in\FF_{k_1}^\Z(\X_{\b^1})\t\FF_{k_2}^G(\X_{\b^2})$ as a  Lipschitz continuous group morphism.
\end{remark}

\begin{proof}[Proof of Proposition~\ref{prop_tensor}] Since the image of $\phi$ is dense in $(\FF_{k_1,k_2}^G(\R^n),\Fw)$ it is enough to establish that $\Fw(A)=\Fw(\pi(A))$ for every $\wh A\in\GG_1\wh\ot_\pi\GG_2$. Equivalently, in view of Proposition~\ref{prop_proj_oo} it is enough to establish that for every $\wh A\in\GG_1\wh\ot_\pi\GG_2$ there holds 
\begin{equation}\label{prf_prop_tensor_1}
 \pi(\wh A)=\inf \lt\{ \sum_{j\ge1}\F(A_j^1)\F(A_j^2) \, : \, \wh A= \sum_{j\ge1}[A_j^1\we A_j^2+\GG_0]\rt\},
\end{equation}
where the convergence of the series is in $\pi$ norm. In fact this identity is the analogue in the context of tensor products of Abelian normed groups of~\cite[Proposition~2.8]{Ryan} which deals with tensor products of Banach spaces.

\noindent
Let us prove~\eqref{prf_prop_tensor_1}.
Let $\wh A\in\GG_1\wh\ot_\pi\GG_2$. We denote by $\wt\pi$ the right-hand side of~\eqref{prf_prop_tensor_1}. First given a decomposition $\wh A= \sum_{j\ge1}[A_j^1\we A_j^2+\GG_0]$ with $\sum \F(A_j^1)\F(A_j^2)<\oo$ we introduce the partial sums
\[
\wh A_j:=\sum_{i=1}^j[A_j^1\we A_j^2+\GG_0].
\]
Since $\pi(A_j^1\we A_j^2)\le \F(A_j^1)\F(A_j^2)$, $\wh A_j$ is a Cauchy sequence in $(\GG_1\wh\ot_\pi\GG_2,\pi)$. Its limit is $\wh A$ and we have $\pi(\wh A)\le \sum \F(A_j^1)\F(A_j^2)$. Taking the infimum with respect to the decompositions, we get
\be\label{prf_prop_tensor_2}
\pi(\wh A)\le\wt\pi(\wh A).
\ee
Let us establish the converse inequality. Let $\wh A_j=A_j+\GG_0\in\GG/\GG_0$ with $\pi(\wh A_j-\wh A)\to 0$. Let $\eps>0$, up to extraction   we may assume that $\pi(\wh A_1-\wh A)<\eps/3$ and denoting $B_j=A_{j+1}-A_j$,
\[
 \sum \pi(B_j)<\eps/3.
\]
By definition of $\pi$, there exist finite sequences $A_r^l\in \GG_l$ and (for $j\ge 1$) $B_{j,s}^l\in\GG_l$ such that
\[
 A_1=\sum_r A^1_r\ot A^2_r,\qquad\qquad B_j=\sum_s B^1_{j,s}\ot B^2_{j,s}
\]
and
\be\label{prf_prop_tensor_3}
 \sum_r \F(A^1_r)\F(A^2_r) +\sum_j \sum_s \F(B^1_{j,s})\F(B^2_{j,s})<\pi(A_1) +\sum \pi(B_j) +\eps/3<\pi(\wh A)+\eps.
\ee
Writing
\[
 \wh A=\sum_r [A^1_r\ot A^2_r+\GG_0]+ \sum_j\sum_s [B^1_{j,s}\ot B^2_{j,s}+\GG_0]=:\sum_t [C_t^1\ot C_t^2+\GG_0],
\]
we deduce  
\[
\wt \pi(\wh A)\le\sum_t\Fw(C_t^1)\Fw(C_t^2)\st{\eqref{prf_prop_tensor_3}}< \pi(\wh A) +\eps.
\]
Since $\eps$ is arbitrary we get $\wt\pi(\wh A)\le\pi(\wh A)$ and together with~\eqref{prf_prop_tensor_2}, this establishes~\eqref{prf_prop_tensor_1} and the proposition.\end{proof}

Let us consider eventually the particular case when $G=V$ is a Banach space. In this case we obviously have
\[
\FF^\Z_{k_1}(\X_{\b^1})\ot\FF^V_{k_2}(\X_{\b^2})  \ =\ \FF^\R_{k_1}(\X_{\b^1})\ot\FF^V_{k_2}(\X_{\b^2}),
\]
that is: the ``group tensor product" is the same as the ``vector tensor product''. As a direct consequence of Proposition~\ref{prop_tensor} and~\cite[Proposition 2.1]{Ryan} we have the following result.

\begin{proposition}~\\
$\Fw$  is the projective norm on the tensor product of Banach spaces $\FF^\R_{k_1}(\X_{\b^1})\ot\FF^V_{k_2}(\X_{\b^2})$. 
\end{proposition}
In particular, when $G$ is a Banach space, the second inequality of~\eqref{estimA1A2} is an identity.  We also have in this case $\GG_0=\{0\}$.

\subsection{Slicing of tensor chains}

The slicing operators $\Sl_\g$  have been introduced in Section~\ref{SGfc}.  Let us check that these operators have a good behavior on tensor polyhedral chains. Let $\g\sub\b$ and let $P\in \PP_{k_1,k_2}^G(\X_\b)$ with representation~\eqref{k1k2polychain}. From the definition of $\Sl_\g$ we have that for almost every $x=x^1+x^2\in\X_{\g}$ with $x^l\in\X_{\g^l}$,
\[
\Sl_\g^xP=\sum g_i (\Sl_{\g^1}^{x^1}p^1_i)\we(\Sl_{\g^2}^{x^2}p^2_i).
\]
We deduce from this formula that
\[
\pt_1\Sl_\g^xP=\Sl_\g^x\pt_1P,\qquad \pt_2\Sl_\g^xP=(-1)^{|\g^1|}\Sl_\g^x\pt_2P\qquad\text{for almost every }x\in\X_\g.
\]
Together with the bound
\[
\int_{\X_\g}\Mw(\Sl_\g^xP)\,dx\le\Mw(P),
\]
this leads to 
\[
\int_{\X_\g}\Fw(\Sl_\g^xP)\,dx\le \Fw(P).
\] 
Moreover, denoting $r_1=|\g^1|$, $r_2=|\g^2|$, there holds for $(k'_1,k'_2)\in D_k$ (recall the definition \eqref{jP} of $\j$), and $P\in \PP_k^G(\X_\b)$, 
\[
\j_{k'_1-r_1,k'_2-r_2} \Sl_\g^x P= \Sl_\g^x \j_{k'_1,k'_2}P.
\] 
Notice that if $P\in \PP_{k_1,k_2}^G(\X_\b)$, then $\j_{k'_1,k'_2}P=0$ if $(k_1',k_2')\neq (k_1,k_2)$. 
If $\g\sub\b^1$, we have the obvious identity
\[
\Sl_\g^x \i P = \i \Sl_\g^x P.
\]
Eventually, by construction $\Sl_\g^x P$ is a tensor chain in $\X_{\b\sm\g}$.

We deduce the following result by continuity and density arguments.
\begin{proposition}\label{prop_Sltfc}
Let $\g\sub\b$, the slicing operators $\Sl_\g$ extend as continuous group morphisms,
\[
\FF^G_{k_1,k_2}(\X_\b)\longto L^1(\X_\g,\FF^G_{k_1-|\g^1|,k_2-|\g^2|}(\X_{\b\sm\g})),
\]
with the following properties, for $A\in \FF^G_{k_1,k_2}(\X_\b)$ and almost every $x\in\X_\g$.
\begin{enumerate}[(i)]
\item  $\pt_1\Sl_\g^xA=\Sl_\g^x\pt_1A$ and $\pt_2\Sl_\g^xA=(-1)^{|\g^1|}\Sl_\g^x\pt_2A$.
\item We have the estimates
\[
\int_{\X_\g}\Mw(\Sl_\g^xA)\,dx\le \Mw(A)\qquad\quad\text{and}\qquad\quad\int_{\X_\g}\Fw(\Sl_\g^xA)\,dx\le \Fw(A).
\] 
\item For $A\in \FF_k^G(\X_\b)$ and  $(k'_1,k'_2)\in D_k$, there holds $\j_{k'_1-|\g^1|,k'_2-|\g^2|} \Sl_\g^x A= \Sl_\g^x \j_{k'_1,k'_2}A$.
\item If $\g\sub\b^1$, there holds $\Sl_\g^x \i A = \i \Sl_\g^x A$.
\end{enumerate}
\end{proposition}

Let us establish the counterpart of Theorem~\ref{coro_defWhite2} for tensor chains, namely that if the $(0,0)$-slices of a tensor chain vanish then it also vanishes.
\begin{proposition}
\label{coro_defWhite3}
If  $A\in \FF^G_{k_1,k_2}(\X_\b)$ is such that $\Sl_\g A=0$ for every $\g\sub\b$ with $(|\g^1|,|\g^2|)=(k_1,k_2)$, then $A=0$.
\end{proposition}
\begin{proof}
Let $A\in \FF^G_{k_1,k_2}(\X_\b)$ and let us assume that $\Sl_\g A=0$ for every $\g\sub\b$ such that $(|\g^1|,|\g^2|)=(k_1,k_2)$ . We fix $\g^2\in I^n_{k_2}$ with $\g^2\sub\b^2$ and denote $A_{\g^2}:=\i \Sl_{\g^2} A$.  Recalling Proposition~\ref{prop_Sltfc}(iv) (commutativity of $\i$ with slicing) and Proposition~\ref{prop_MWSl}(ii) (iterate slicing), we compute, for every $\g^1\in I^n_{k_1}$ with $\g^1\sub\b^1$,
\[
\Sl_{\g^1}A_{\g^2}=\Sl_{\g^1} \i \Sl_{\g^2} A=\i\Sl_{\g^1}\Sl_{\g^2} A=\i\Sl_\g A=0.
\]
 Since this holds true for every $\g^1\in I^n_{k_1}$, we deduce from Theorem~\ref{coro_defWhite2} that $A_{\g^2}=0$. By injectivity of $\i$ (since it is an isometry) we get
\be\label{proof_coro_defWhite3}
\Sl_{\g^2} A=0\qquad\text{for every }\g^2\in I^n_{k_2}\text{ such that }\g^2\sub\b^2.
\ee
Exchanging the roles of $\b^1$ and $\b^2$ and denoting $\ov\i:\FF^G_{k_1,k_2}(\X_\b)\to \FF_{k_2}(\X_{\b^2},\FF^G_{k_1}(\X_{\b^1}))$ the corresponding operator, we compute for  $\g^2\in I^n_{k_2}$ with $\g^2\sub\b^2$,
\[
\Sl_{\g^2}\ov \i A= \ov\i\Sl_{\g^2} A \st{\eqref{proof_coro_defWhite3}}=0.
\] 
Again, Theorem~\ref{coro_defWhite2} leads to $\ov \i A=0$, that is $A=0$.
\end{proof}

\subsection{Restriction of finite mass tensor chains}
\label{Ss_fmtfc}
To define  the restrictions of finite mass tensor chains on Borel subsets, we can proceed exactly as in~\cite{Fleming66} using Lemma~\ref{lem2.1}(i) in place of~\cite[Lemma~2.1]{Fleming66} and the $\Fw$-convergence in place of the $\F$-convergence. Given $A\in \MM^G_{k_1,k_2}(\X_\b)$, we obtain a finite positive Borel measure $\mu_A$ and a $\FF^G_{k_1,k_2}(\X_\b)$-valued measure $A\restr\cd\,$ satisfying the relation $\mu_A(S)=\Mw(A\restr S)$ for $S$ Borel subset of $\X_\b$. Besides $A\restr S$ is supported in $\ov S$ and for a sequence $A_j\in\MM^G_{k_1,k_2}(\X_\b)$ with $\sum\Mw(A_j)<\oo$, there holds 
\[
\lt(\sum A_j\rt)\restr S= \sum \lt(A_j\restr S\rt).
\]
Eventually, given a sequence $A_j\to A$ with $\Mw(A_j)\to\Mw(A)$ we have $A_j\restr S\to A\restr S$ for any Borel subset such that $\mu_A(\bndry S)=0$.\medskip

With the aim of  establishing the natural identities involving the restrictions of $A$, $\i A$ and $\j A$, we recall the main lines of the construction of $\mu_A$ and $A\restr\cd\,$.\\
Let $A\in\MM^G_{k_1,k_2}(\X_\b)$. We first define the measure $\mu_A$. Let $P_j$ converging to $A$ rapidly and such that $\Mw(P_j)\to\Mw(A)$.  The sequence $\mu_{P_j}$ is bounded, hence there exists a finite positive Borel measure $\mu_A$ such that up to extraction $\mu_{P_j}\st*\rightharpoonup\mu_A$.\\
Let $I$ be an interval of $\X_\b$. By Lemma~\ref{lem2.1}(i), for almost every $x\in \X_\b$, the series $\sum \Fw([P_{j+1}-P_j]\restr (x+I))$ is converging. We define for such $x$,
\[
A\restr (x+I):=P_1\restr (x+I)+\sum (P_{j+1}-P_j)\restr (x+I).
\]
The construction then proceeds as follows. We first show a counterpart of~\cite[Lemma~4.1]{Fleming66}, that is, for almost every $x\in\X_\b$, there hold
\[
\Mw(A\restr (x+I))=\mu_A (x+I)\quad\qquad\text{ and }\qquad\quad\Mw(A-A\restr  (x+I))=\mu_A(\X_\b\sm (x+I)).
\]
This allows us to define $A\restr J$ for $J$ finite union of disjoint non-exceptional intervals (with respect to the sequence $P_j$). This restriction is an additive set functions and we still have $\Mw(A\restr J)=\mu_A(J)$. Since the sets of this form span the Borel algebra, we can extend $A\restr\cd\,$ as a Borel measure with values in $\FF^G_{k_1,k_2}(\X_\b)$. Notice that arguing as in \cite{Fleming66}, both $\mu_A$ and $A\restr\cd\,$ do not depend on the choice of sequence $P_j$.\\
In the proof of Proposition~\ref{prop_restr} below we use the following facts: 
\begin{enumerate}[(1)]
\item 
The set function $A\restr\cd\,$ is determined by its values on any collection of sets $x+I$ where $I$ ranges over the intervals of $\X_\b$ and $x$ ranges over $\X_\b\sm E_I$ for some negligible set $E_I\sub\X_\b$. 
\item If $I\sub\X_\b$ is an interval and $Q_j\to A$ rapidly, then $Q_j\restr (x+I)\to A\restr(x+I)$ for almost every $x\in\X_\b$.
\end{enumerate}

\begin{proposition}\label{prop_restr} Let $A'\in\MM^G_{k_1,k_2}(\X_\b)$, $A\in\MM^G_k(\X_\b)$ and let $S,\tilde S\sub\X_\b$ be Borel sets. We have the following properties.
\begin{enumerate}[(i)]
\item For $\g\sub\b$ and almost every $x\in\X_\g$, $[{A'}\cap \X_{\b\sm\g}(x)]\restr S=[{A'}\restr S]\cap \X_{\b\sm\g}(x)$.\smallskip\\
With the $\Sl_\g$ operator, this translates into $\Sl^x_\g (A'\restr S)= (\Sl^x_\g A')\restr ([S-x]\cap\X_{\b\sm\g})$.
\item $(\j_{k'_1,k'_2} A)\restr S=\j_{k'_1,k'_2}(A\restr S)$ for any $(k'_1,k'_2)\in D_k$. 
\item If $S\sub\X_{\b^1}$ then $(\i A')\restr S=\i(A'\restr [S+\X_{\b^2}])$.
\item $(A'\restr S)\restr \tilde S=A'\restr(S\cap \tilde S)$.
\end{enumerate}
\end{proposition}
\begin{proof}
These identities hold true when $A'$ is a tensor polyhedral chain (resp. $A$ is a polyhedral chain) and when $S$ and $\tilde S$ are intervals of $\X_\b$ or of $\X_{\b^1}$ for point~(iii). Taking a sequence of tensor polyhedral chains $P'_j$ converging rapidly to $A'$ (resp. a sequence of  polyhedral chains $P_j$ converging rapidly to $A$ for~(ii)) and using the continuity properties of the operators $\Sl_\g$, $\j$ and $\i$, the result follows from the two facts stated above.
\end{proof}

\section{Identification of normal tensor chains with normal chains}
\label{SId}
In this section, we first identify finite mass $(0,0)$-chains with finite mass $0$-chains and deduce that $\j$ is one-to-one. Next, we introduce the subgroups of $k$-chains $\NN_{k,(k_1,k_2)}(\R^n)$ formed by the elements of $\NN^G_k(\R^n)$ which are in the kernel of $\j_{k'_1,k'_2}$ for every $(k'_1,k'_2)\in D_k\sm\{(k_1,k_2)\}$. For our purpose, the interest of these groups lies in the fact that a normal \emph{rectifiable} $k$-chain belongs to $\NN^G_{k,(k_1,k_2)}(\R^n)$ if and only if it is $(k_1,k_2)$-split (see Corollary~\ref{coro_splitting_iff_jk'1k'2=0}). We also introduce the subgroup of normal $(k_1,k_2)$-chains $\NN^G_{k_1,k_2}(\R^n)$ and show that $\j_{k_1,k_2}$ maps (injectively) $\NN^G_{k,(k_1,k_2)}(\R^n)$ into $\NN^G_{k_1,k_2}(\R^n)$. We also prove that for $k_1=0$ this mapping is onto (Theorem~\ref{thm_Ups}).\footnote{In~\cite{GM_comp} we establish that in fact $\j_{k_1,k_2}:\NN^G_{k,(k_1,k_2)}(\R^n)\to\NN^G_{k_1,k_2}(\R^n)$ is onto for every $(k_1,k_2)\in D_k$.} These results allow us to identify normal rectifiable $k$-chains which are $(k_1,k_2)$-split with some normal rectifiable $(k_1,k_2)$-chains and deduce the main result, Theorem~\ref{thm_main} (reformulated as Theorem~\ref{thm_main_bis}) from its counterpart for normal rectifiable $(k_1,k_2)$-chains (Theorem~\ref{thm_main2}).

From now on we use capital letters for chains $A,B,C,Q$,\, \dots\ and capital letters with a prime for tensor chains: $A',B',C',Q'$,\,\dots. We also use the notation $A''$ for the chain $\i A'$.


\subsection{Finite mass $(0,0)$-chains as measures and injectivity of $\j$}
\label{Ss_00}

Recall that by Theorem~\ref{thm_psi}, there exists an isometric isomorphism $\psi:\MM^G_0(\X_\b)\to\MM(\X_\b,G)$. We generalize the construction of~\cite[Section~2]{White1999-2} to obtain a similar result for $(0,0)$-chains  in Lemma~\ref{lem_00}. For this we need to introduce further material.

As a preliminary we recall the following from~\cite{White1999-2}.
\begin{theorem}[{\cite[Theorem~2.1]{White1999-2}}]\label{thm_chi}
There exists a group morphism $\chi:\FF^G_0(\X_\b)\to G$ such that,
\begin{enumerate}[(i)]
\item $\chi(\sum g_i\lb x_i \rb)=\sum g_i$, for any finite sequences $g_i\in G$ and $x_i\in\X_\b$.
\item $|\chi (A)|_G\le \F(A)$ for every $A\in\FF^G_0(\X_\b)$.
\end{enumerate}
\end{theorem}
The operator $\chi$ is defined on polyhedral $0$-chains by the relation~(i) and then extended  by continuity thanks to~(ii). Notice also that given a polyhedral $1$-chain $B$, there holds $\chi(\pt B)=0$.\\
We need a complement of this result which is not stated in~\cite{White1999-2} but is obvious from the construction there. Let $(G^a,+,|\cd|_{G^a})$ and $(G^b,+,|\cd|_{G^b})$ be two complete Abelian normed groups and let $\phi:G^a\to G^b$ be a Lipschitz continuous group morphism and $\Phi:G^a\to L^1(\Om,G^b)$ as in Proposition~\ref{prop_phi_Phi}.
\begin{proposition}\label{prop_cplmt_chi}
With the notation of Proposition~\ref{prop_phi_Phi}, we have for every $A\in\FF_0(\X_\b,G^a)$, 
\[
\chi(\phi\et A)=\phi(\chi (A)),\qquad\qquad\chi\lt[(\Phi\et A)(\om)\rt]=\lt(\Phi\lt[\chi (A)\rt]\rt)(\om)\quad\text{for almost every }\om\in\Om.
\]
\end{proposition}
\begin{proof}
These identities hold true for polyhedral $0$-chains and the general case follow by density of these latter in $\FF_0(\X_\b,G^a)$ and Lipschitz continuity of $\phi$, $\Phi$ and of the operator $\chi$ (point~(ii) of Theorem~\ref{thm_chi}).
\end{proof}

We introduce the corresponding operator on $\FF^G_{0,0}(\X_\b)$. 
\begin{definition}\label{def_chiw}
We set for $A'\in\FF^G_{0,0}(\X_\b)$,
\[
\chiw(A'):=\chi(\chi(\i A')).
\] 
Beware that on the right-hand side the notation $\chi$ refers both to the morphism $\FF_0(\X_{\beta^2},G)\to G$ and to the morphism $\FF_0(\X_{\b^1},\FF_0(\X_{\b^2},G))\to \FF_0(\X_{\b^2},G)$.
\end{definition}

\begin{proposition}\label{prop_chiw}
The group morphism $\chiw:\FF^G_{0,0}(\X_\b)\to G$ satisfies the following properties. 
\begin{enumerate}[(i)]
\item  For every $A'\in\FF^G_{0,0}(\X_\b)$, there holds $|\chiw(A')|_G\le\Fw(A')\le\Mw(A')$. 
\item  For every  $A\in\FF^G_0(\X_\b)$, there holds $\chi(A)=\chiw(\j_{0,0}A)$.
\end{enumerate}
\end{proposition}
\begin{proof}
Let $A'\in\FF^G_{0,0}(\X_\b)$. Using~point~(ii) of Theorem~\ref{thm_chi} twice, we have
\[
|\chiw(A')|_G=\big|\chi[\underbrace{\chi(\i A')}_{\in\FF^G_{k_2}(\X_{\b^2})}]\big|_G\le\lt|\chi(\i A')\rt|_{\FF^G_{k_2}(\X_{\b^2})}\le\F(\i A')=\Fw(A')\le\Mw(A'),
\]
where we used the notation $|B|_{\FF^G_{k_2}(\X_{\b^2})}:=\F(B)$ for $B\in\FF^G_{k_2}(\X_{\b^2})$.
This proves~(i). Next,~(ii) holds true  for $0$-polyhedral chains and the general case follows by density and continuity.  
\end{proof}

We can now generalize Theorem~\ref{thm_psi} to tensor chains, with the extra relation~\eqref{lem_00_1} below.
\begin{lemma}
\label{lem_00}
There is an isometric isomorphism,  
\[
\psiw:(\MM^G_{0,0}(\X_\b),+,\Mw)\longto(\MM(\X_\b,G),+,|\cd|).
\] 
Moreover, $\psiw$ is related to the mapping $\psi$ of~Theorem~\ref{thm_psi} by
\be\label{lem_00_1}
\psi=\psiw\circ\j_{0,0}\qquad\text{on }\MM^G_0(\X_\b).
\ee
\end{lemma}
\begin{proof}
Let $A'\in\MM^G_{0,0}(\X_\b)$. For every Borel subset  $S\sub\X_\b$, we define,
\[
\psiw_{A'}(S):=\chiw(A'\restr S).
\]
Using Proposition~\ref{prop_chiw}(i) and arguing  as in the proof of~\cite[Theorem~2.2]{White1999-2}, we get that $\psiw_{A'}$ is a finite Borel measure in $\X_\b$ with values in $G$. Moreover,
\be\label{proof_lem_00_1}
|\psiw_{A'}|\le\Mw(A').
\ee

Let us establish~\eqref{lem_00_1}. Let $A\in\MM^G_0(\X_\b)$. By Proposition~\ref{prop_chiw}(ii) we have for any Borel set $S\sub\X_\b$,
\[
\psiw_{\j_{0,0}A}(S)=\chiw\lt((\j_{0,0}A)\restr S\rt)=\chi(A\restr S).
\]
Now, by definition (see~\cite[Theorem~2.2]{White1999-2}), $\chi(A\restr S)=\psi_A(S)$. Hence~\eqref{lem_00_1} holds true. 

It remains to show that $A'\mapsto\psiw_{A'}$ is onto and that
\be\label{proof_lem_00_2}
\Mw(A')\le|\psiw_{A'}|.
\ee
Let $\nu\in\MM(\X_\b,G)$. Since $\psi$ is an isometric isomorphism, there exists $A\in\MM^G_{0}(\X_\b)$ such that $\psi_A=\nu$ and $\M(A)=|\nu|$.   
Setting $A':=\j_{0,0}A\in\MM^G_{0,0}(\X_\b)$, we have $\psiw_{A'}=\nu$ by~\eqref{lem_00_1}. This proves that $\psiw$ is onto. Eventually, there holds
\[
\Mw(A')=\M(\j_{0,0}A)\le\M(A)=|\nu|,
\]
and we get~\eqref{proof_lem_00_2}. With~\eqref{proof_lem_00_1}, we conclude that $\psiw$ is an isometry from $(\MM^G_{0,0}(\X_\b),\Mw)$ onto $(\MM(\X_\b,G),|\cd|)$.
\end{proof}
The previous lemma allows to identify isomorphically $\MM^G_{0,0}(\X_\b)$ with $\MM^G_0(\X_\b)$ so that, essentially, $\MM^G_{0,0}(\R^n)$ does not depend on $n_1,n_2$ (but as shown by the example of Remark~\ref{rem_on_tfc}(d), $\FF^G_{0,0}(\X_\b)$ does!). Identity~\eqref{lem_00_1} also provides a natural isometric isomorphism.
\begin{theorem}\label{thm_MM00MM0}~
\begin{enumerate}[(i)]
\item The mapping  $\j_{0,0}:(\MM^G_0(\X_\b),+,\M)\longto(\MM^G_{0,0}(\X_\b),+,\Mw)$ 
is an isometric isomorphism. In the sequel, its inverse is denoted $\Ups_0$.
\item Besides, $\Mw(\j_{0,0}A)=\M(A)$ for every $A\in\FF^G_0(\X_\b)$. In particular, $\j_{0,0}A=0\implies A=0$.
\end{enumerate}
\end{theorem}

\begin{proof}
From the previous lemma and Theorem~\ref{thm_psi}, we have the following commutative diagram of isomorphisms,
\begin{diagram}
\MM^G_0(\X_\b)&         &  \rTo^{\j_{0,0}}           &           &\MM^G_{0,0}(\X_\b)\\
                         &\rdTo<\psi&                      & \ldTo>{\psiw} &\\
                        &         &\MM(\X_\b,G)&            &
\end{diagram}
Moreover $\psi$ and $\psiw$ are isometries and we conclude that, 
\[
\j_{0,0}:(\MM^G_0(\X_\b),+,\M)\longto(\MM^G_{0,0}(\X_\b),+,\Mw),
\] 
is an isometric isomorphism, which is~(i). As a consequence, the identity of the second point holds true as soon as $\M(A)$ is finite but we still have to check that $\Mw(\j_{0,0}A)=\oo\iff \M(A)=\oo$.  This follows again from the diagram above. Indeed we have for $A\in\FF^G_0(\X_\b)$,
\[
\M(A)<\oo\iff A\in\Ima\psi^{-1}\iff\j_{0,0}A\in\Ima{\psiw}^{-1}\iff\Mw(\j_{0,0}A)<\oo. 
\]
This proves the theorem.
\end{proof}
Using slicing, we deduce that $\j$ is one-to-one.
\begin{theorem}\label{thm_j_121}
The group morphism $\j:\FF^G_k(\X_\b)\to(\FF^G_{k'_1,k'_2}(\X_\b))_{(k'_1,k'_2)\in D_k}$ is one-to-one.
\end{theorem}
\begin{proof}
Let  $A\in\FF^G_k(\X_\b)$ such that $\j A=0$. By Proposition~\ref{prop_Sltfc}(iii), for every $\g\sub\b$ with $|\g|=k$, there holds
\[
\j_{0,0}\Sl_\g A=\Sl_\g\j_{|\g^1|,|\g^2|}A=\Sl_\g0=0.
\]
It then follows from  Theorem~\ref{thm_MM00MM0}(ii) that $\Sl_\g A=0$ for every $\g\sub\b$ with $|\g|=k$ and by~Theorem~\ref{coro_defWhite2} we get $A=0$.
\end{proof}
With the same line of arguments and using Proposition \ref{coro_defWhite3} instead of Theorem \ref{coro_defWhite2}, we deduce the following statement.
\begin{proposition}\label{coro_M_Sl=0_iif_j=0} 
Let  $A\in \FF^G_k(\X_\b)$. There holds for every  $(k'_1,k'_2)\in D_k$,  
\[
\j_{k'_1,k'_2}A=0\quad\iff\quad\Sl_\g A=0\text{ for every }\g\in I^n_k\text{ such that }(|\g^1|,|\g^2|)=(k'_1,k'_2). 
\]
\end{proposition}

\subsection{Normal and rectifiable tensor chains}
\label{Snrtfc}

Let us introduce some subgroups of $k$-chains which identify through $j_{k_1,k_2}$ with $(k_1,k_2)$-chains.
\begin{definition}\label{defFFkk1k2}
We set 
\[
\FF^G_{k,(k_1,k_2)}(\X_\b):=
\lt\{ A\in \FF^G_k(\X_\b): \j_{k'_1,k'_2} A=0 \text{ for every } (k'_1,k'_2)\in D_k\sm\{(k_1,k_2)\}\rt\}.
\]
Then we define, 
\[
\MM^G_{k,(k_1,k_2)}(\X_\b):=\FF^G_{k,(k_1,k_2)}(\X_\b)\cap\MM^G_k(\X_\b),\quad\qquad\NN^G_{k,(k_1,k_2)}(\X_\b):=\FF^G_{k,(k_1,k_2)}(\X_\b)\cap\NN^G_k(\X_\b).
\]
\end{definition}
\begin{remark}\label{rem_A1A2}
Let $A^1\in\FF^\Z_{k_1}(\X_{\b^1})$ and $A^2\in\FF^G_{k_2}(\X_{\b^2})$. We have seen in Section~\ref{SGfc} that when $A^1$ and $A^2$ have finite mass or at least one of them is normal, we can define the Cartesian product $A:=A^1\t A^2$ in $\FF^G_k(\X_\b)$. Considering first polyhedral chains, it is easy to see that (recall \eqref{wedge})
\[
\j_{k_1,k_2}A=A^1\we A^2=:A'\ \in\FF^G_{k_1,k_2}(\X_\b),
\] 
and 
\[
\j_{k'_1,k'_2}A=0\qquad\text{for }(k'_1,k'_2)\ne(k_1,k_2).
\]
It follows that $A$ lies in $\FF^G_{k,(k_1,k_2)}(\X_\b)$ and since $\j$ is one-to-one we can identify $A$ with $A'$. More generally, considering a sequence of pairs $(A^1_j,A^2_j)$ of the above form, if the series $\sum A_j^1\t A_j^2$ converges in $\FF^G_k(\X_\b)$, we can identify  $\sum A_j^1\t A_j^2\in\FF^G_{k,(k_1,k_2)}(\X_\b)$ with $\sum A_j^1\we A_j^2\in\FF^G_{k_1,k_2}(\X_\b)$.
\end{remark}
\begin{remark}\label{j=0<=>Sl=0}
By~Proposition~\ref{coro_M_Sl=0_iif_j=0}, we have the alternative definition:
\[
\FF^G_{k,(k_1,k_2)}(\X_\b)=
\lt\{ A\in \FF^G_k(\X_\b): \Sl_\g A=0 \text{ for every } \g\in I^n_k\text{ such that }(|\g^1|,|\g^2|)\ne(k_1,k_2)\rt\},
\]
and similarly for $\MM^G_{k,(k_1,k_2)}(\X_\b)$ and $\NN^G_{k,(k_1,k_2)}(\X_\b)$.
\end{remark}

As a consequence of this remark, we observe that Proposition~\ref{prop_split_and_Sl} rephrases as follows.

\begin{corollary}\label{coro_splitting_iff_jk'1k'2=0}
The following statements are equivalent for every $A\in\MM^G_k(\X_\b)$ rectifiable. 
\begin{enumerate}[(i)]
\item $T\mu_A$  is $(k_1,k_2)$-split.
\item $\Sl_\g A=0$ for every $\g\in I^n_k$ with $(|\g^1|,|\g^2|)\ne(k_1,k_2)$.
\item $A\in\MM^G_{k,(k_1,k_2)}(\X_\b)$.
\end{enumerate}
\end{corollary}

Let us  now introduce the notion of normal tensor chains. 
\begin{definition}\label{def_ntfc}
Let $A'\in\FF^G_{k_1,k_2}(\X_\b)$, we say that  $A'$ is normal if 
\[
\Nw(A'):=\Mw(A')+\Mw(\pt_1 A')+\Mw(\pt_2 A') <\oo.
\] 
We denote by $\NN^G_{k_1,k_2}(\X_\b)$ the group of normal $(k_1,k_2)$-chains in $\X_\b$.
\end{definition}
The normed groups $(\NN^G_{k_1,k_2}(\X_\b),+,\N)$ are complete but do not form a two-dimensional chain complex\footnote{To get a true analogue of the complex  $\NN^G_*(\X_\b)$ we should consider instead $
\wt\Nw(A'):=\sum\limits_{0\le i_1,i_2\le1}\Mw(\pt_1^{i_1}\pt_2^{i_2}A)$.} (unless $|\b^1|,|\b^2|\le1$). 
 
 Notice also that $\j_{k_1,k_2}$ maps $\FF^G_{k,(k_1,k_2)}(\X_\b)$ into  $\FF^G_{k_1,k_2}(\X_\b)$ and that (recalling~\eqref{estimjA}\&\eqref{estimjA_M}) we have for $A\in\FF^G_{k,(k_1,k_2)}(\X_\b)$,
\[
\Fw(\j_{k_1,k_2}A)\le2\Ft(A),\qquad\qquad\Mw(\j_{k_1,k_2}A)\le \Mt(A).
\]
Besides, by Proposition~\ref{prop_pt_and_j}, for $(k'_1,k'_2)\in D_k$ (and still $A\in\FF^G_{k,(k_1,k_2)}(\X_\b)$), we have the formula,
\[
\j_{k'_1,k'_2}\pt A=\begin{cases}
\pt_1\j_{k_1,k_2}A&\text{if }(k'_1,k'_2)=(k_1-1,k_2),\\
\pt_2\j_{k_1,k_2}A&\text{if }(k'_1,k'_2)=(k_1,k_2-1),\\
\qquad 0&\text{in the other cases}.
\end{cases}
\] 
We deduce that $\Nw(\j_{k_1,k_2}A)\le \Nt(A)= \Mt(A)+\Mt(\pt A)$.
Let us summarize the above statements. 
\begin{proposition}\label{prop_embeddings} We have the continuous embeddings:
\begin{align*}
\j_{k_1,k_2}&:\FF^G_{k,(k_1,k_2)}(\X_\beta)\hookrightarrow\FF^G_{k_1,k_2}(\X_\beta),\\
\j_{k_1,k_2}&:\MM^G_{k,(k_1,k_2)}(\X_\beta)\hookrightarrow\MM^G_{k_1,k_2}(\X_\beta),\\
\j_{k_1,k_2}&:\NN^G_{k,(k_1,k_2)}(\X_\beta)\hookrightarrow\NN^G_{k_1,k_2}(\X_\beta).
\end{align*} 
\end{proposition}

A natural question is whether some of these mappings are onto. In a complementary paper, we establish that $\j_{k_1,k_2}$ defines an isometric isomorphisms from $\NN^G_{k,(k_1,k_2)}(\X_\b)$ onto $\NN^G_{k_1,k_2}(\X_\b)$, (see~\cite[Theorems~A.1]{GM_comp}) where the isometry holds with respect to a $\N$-norm built on the so called \emph{coordinate slicing mass} rather than on the mass itself. We also show there, that except in some limit cases, $\j_{k_1,k_2}: \MM^G_{k,(k_1,k_2)}(\X_\b)\to\MM^G_{k_1,k_2}(\X_\b)$ and $\j_{k_1,k_2}: \FF^G_{k,(k_1,k_2)}(\X_\b)\to\FF^G_{k_1,k_2}(\X_\b)$ are not onto. Here, we only establish the positive result about normal chains in the case $k_1=0$ and without the isometry property.
 
\begin{theorem}\label{thm_Ups}
The mapping $\j_{0,k}$ defines a group isomorphism from $\NN^G_{k,(0,k)}(\X_\b)$ onto $\NN^G_{0,k}(\X_\b)$. Its inverse is denoted $\Ups$.
\end{theorem}

\begin{proof}
Without loss of generality, we assume that $\b=\{1,\dots,n\}$ so that $\beta^1=\a$ and $\b^2=\ov \a$. In the sequel, we denote $\NN:=\NN^G_{k,(0,k)}(\R^n)$ and $\NN':=\NN^G_{0,k}(\R^n)$.

The proof is split into two steps. In the first step we define a group morphism $\Ups:\NN'\to\NN$ and show that $\j_{0,k}\circ\Ups=\Id_{|\NN'}$. In a second step, we establish that $\Ups\circ{\j_{0,k}}_{|\NN}=\Id_{|\NN}$.\medskip

\noindent
\textit{Step 1. Definition of $\Ups$ and proof of $\j_{0,k}\circ\Ups=\Id$ on $\NN'$.}\\
We start with the following observation. If $A=\sum A_j \lb x_j \rb\in \PP_0(\X_\a, \FF_k^G(\X_{\ov \a}))$ for some sequence $A_j\in \FF_k^G(\X_{\ov \a})$ such that $\sum\M(A_j)<\oo$ and some sequence of pairwise distinct points $x_j\in\X_\a$, we have $\chi(A)=\sum_j A_j$ and thus $\M(\chi(A))\le \sum_j \M(A_j)$, that is,
\begin{equation}\label{Mchi}
\M(\chi(A))\le \M(A).
\end{equation}
By density and continuity of $\chi(A)$, \eqref{Mchi} extends to all $A\in \MM_0(\X_\a,\FF^G_k(\X_\ova))$.\medskip

Next, for $j\ge 0$ and $m\in\Z^{n_1}$, we set $y^j_m:=2^{-j}m$ and $Q_m^j:=y^j_m+[0,2^{-j})^{n_1}$. For $j\ge0$ and  $A'\in\FF^G_{0,k}(\R^n)$ such that $\i A'\in\MM_0(\X_\a,\FF^G_k(\X_\ova))$ we define 
\[
\Ups_jA':=\sum_m\lb y^j_m \rb{\times}\chi\lt([\i A']\restr Q_j^m\rt).
\]
Let $A'\in\NN'$ and let us denote $A_j:=\Ups_j A'$ and $A'_j:=\j_{0,k}A_j$. By construction,
\begin{align}
\label{Step2b_0}
A_j=\sum_m\lb y^j_m \rb\t\chi([\i A']\restr Q_j^m)\ \in\FF^G_k(\R^n),\\
\label{Step2b_1}
A'_j =\sum_m\lb y^j_m \rb\we\chi([\i A']\restr Q_j^m)\ \in\FF^G_{0,k}(\R^n).
\end{align}
In the sequel, we show that $A_j$ converges in $\FF^G_k(\R^n)$ to some chain $A\in \NN$ and we define $\Ups A':=A$. We then show that $A'_j$ converges to $A'$ in $\Fw$-norm and finally get $A'=\j_{0,k}\Ups A'$ by continuity of $\j_{0,k}$.\medskip 

Let us fix ${j\ge i\ge 1}$. Let $m\in\Z^{n_1}$ and  denote by $\mathcal{Q}$ the set formed by the $2^{n_1(j-i)}$ cubes $Q^j_r$ contained in $Q^i_m$. For $Q=Q^j_r\in\mathcal{Q}$, denoting $p_Q:=\lb(y^i_m,y^j_r)\rb\in\PP^\Z_1(\X_\a)$, we set,
\[
P_Q:=p_Q\we\chi([\i A']\restr Q)\ \in\PP_1(\X_\a,\FF^G_k(\X_\ova)).
\]
Using $\pt=\pt_1+\pt_2$, we have $(\lb y^j_r \rb-\lb y^i_m \rb)\we\chi([\i A']\restr Q)=\pt P_Q+p_Q\we\pt \chi([\i A']\restr Q)$. Summing on $\mathcal{Q}$, we get
\[
(A_j-A_i)\restr[Q^i_m+\X_\ova]=\pt\sum_{Q\in\mathcal{Q}}P_Q+\sum_{Q\in\mathcal{Q}}p_Q\we\pt \chi([\i A']\restr Q).
\]
This yields, 
\be\label{Proof_Ups_step1_3}
\F\lt((A_j-A_i)\restr[Q^i_m+\X_\ova]\rt)\le\sum_{Q\in\mathcal{Q}}\M(P_Q)+\M(p_Q\we\pt \chi([\i A']\restr Q)).
\ee
On the one hand, for every $Q\in\mathcal{Q}$, there holds, $\M(p_Q)\le\sqrt{n_1}2^{-i}$ and we deduce,
\be\label{Proof_Ups_step1_4}
\sum_{Q\in\mathcal{Q}}\M(P_Q)\le\sqrt{n_1}2^{-i}\sum_{Q\in\mathcal{Q}}\M(\chi([\i A']\restr Q))\stackrel{\eqref{Mchi}}{\le}\sqrt{n_1}2^{-i}\M([\i A']\restr Q^i_m).
\ee
On the other hand, considering the Lipschitz continuous morphism
\[
\ov\pt:=\pt:\FF^G_k(\X_\ova)\to\FF^G_{k-1}(\X_\ova),
\]
we have with the notation of Proposition~\ref{prop_phi_Phi},
\[
\ov\pt\et(\i A')=\i(\pt_2 A'),
\]
and by Proposition~\ref{prop_cplmt_chi},
\begin{equation}\label{eq:etoile}
\pt\lt(\chi([\i A']\restr Q)\rt)=\chi\lt(\ov\pt\et([ \i A']\restr Q)\rt)=\chi([\i\pt_2A']\restr Q).
\end{equation}
We deduce similarly to~\eqref{Proof_Ups_step1_4},
\be\label{Proof_Ups_step1_5}
\sum_{Q\in\mathcal{Q}}\Mw(p_Q\we\pt \chi([\i A']\restr Q))\le\sqrt{n_1}2^{-i}\Mw((\i \pt_2A')\restr Q^i_m).
\ee
Plugging~\eqref{Proof_Ups_step1_4}\&\eqref{Proof_Ups_step1_5} in~\eqref{Proof_Ups_step1_3}, we obtain
\[
\F\lt((A_j-A_i)\restr[Q^i_m+\X_\ova]\rt)\le\sqrt{n_1}2^{-i}\lt(\M([\i A']\restr Q^i_m)+\M((\i\pt_2A')\restr Q^i_m\rt).
\]
Then, summing over $m$, we get
\begin{equation}\label{Cauchyisom}
\F(A_j-A_i)\le\sqrt{n_1}2^{-i}\Nw(A').
\end{equation}
Notice that  we cannot substitute the $\Mw$ for $\Nw$ on the right-hand side. Indeed, the corresponding result does not hold for finite mass chains see~\cite[Proposition~A.5]{GM_comp}.
From \eqref{Cauchyisom} we conclude that  $A_j$ is a Cauchy sequence in $\FF^G_k(\R^n)$ and $A_j\to A$. Moreover, since
\[
 \M(A_j)\le \sum_m\M(\chi([\i A']\restr Q^j_m))\le \M(\chi[\i A'])\le \M(\i A'),
\]
we have
\[
 \M(A)\le\liminf\M(A_j)\le\M(\i A')<\oo.
\]
Similarly,  since
\[
 \partial A_j=\sum_m\lb y^j_m \rb \times \pt\chi([\i A']\restr Q_j^m)\stackrel{\eqref{eq:etoile}}{=}\sum_m\lb y^j_m \rb \times \chi([\i \pt_2 A']\restr Q_j^m)
\]
we have
\[
 \M(\pt A_j)\le \M( \i \pt_2 A')
\]
and thus
\[
 \M(\pt A)\le\liminf\M(\pt A_j)\le\M(\i\pt_2A')<\oo.
\]
Besides, every $A_j$ belongs to $\NN$ so that by continuity of $\i$ we also have $A\in\NN$. Defining $\Ups A':=A$, the mapping $\Ups$ is obviously a group morphism and we have just  established that it maps $\NN'$ into $\NN$.\medskip

Let us now show that  $A'_j\to A'$. Let us fix $j\ge0$, we have, 
\be\label{Step2b_2}
\Fw(A'_j-A')=\F(\i (A'_j-A'))\le\sum_m\F\lt((\i(A'_j-A'))\restr Q^j_m\rt).
\ee
For a fixed $m\in\Z^{n_1}$ we denote $B'':=[\i A']\restr Q^j_m$. We then have from~\eqref{Step2b_1},
\[
\i(A'_j-A')\restr Q^j_m = {\chi(B'')\lb y^j_m \rb-B''}.
\]
Using the homotopy formula~\eqref{homotopy} in $\X_\a$ with $f(x)=x$ and $g(x)=y^j_m$ and taking into account that $\pt B''=0$ (because $B''$ is a $0$-chain), we get with the notation there, 
\be\label{Step2b_3}
\pt [h\pf(\lb(0,e_0)\rb\t B'')]=B''- g\pf B''.
\ee
Remark that for any polyhedral $0$-chain $C$ in $\X_\a$, we have $g\pf C+\chi(C)\lb y^j_m\rb$ and this relation extends by continuity to any $0$-chain. Hence~\eqref{Step2b_3} rewrites as
\[
\pt [h\pf(\lb(0,e_0)\rb\t B'')]=B''-B''-\chi(B'')\lb y^j_m \rb.
\]
We deduce:
\begin{align*}
\F(\i(A'_j-A')\restr Q^j_m )=\F\lt(B''- \chi(B'')\lb y^j_m \rb\rt)
&\stackrel{\phantom{\eqref{estim_homotopy}}}\le\M(h\pf(\lb(0,e_0)\rb\t B''))\\
&\stackrel{\eqref{estim_homotopy}}\le C\|f-g\|_{L^\oo(\supp B'')} \M(B'')\\
&\stackrel{\phantom{\eqref{estim_homotopy}}}\le C\sqrt{n_1}2^{-j}\M(B'')=C\sqrt{n_1}2^{-j}\M([\i A']\restr Q^j_m).
\end{align*}
Summing over $m$ and using~\eqref{Step2b_2}, we obtain
\[
\Fw(A'_j-A')\le C \sqrt{n_1}2^{-j}\M(\i A')\stackrel{\eqref{eq:iM}}{\le}C \sqrt{n_1}2^{-j}\Mw(A').
\]  
We conclude that $A'_j\to A'$ in $\Fw$-norm.
\medskip
 
Eventually, by continuity of $\j_{0,k}$, there holds
\[
\j_{0,k}\Ups A'=\j_{0,k}A=\j_{0,k}\lim A_j=\lim\j_{0,k}A_j=\lim A'_j=A'.
\]
We have established that $\j_{0,k}\circ\Ups=\Id$ on $\NN'$.\bigskip

\noindent
\textit{Step 2, $\Ups\circ\j_{0,k}=\Id$ on $\NN$.} Let us fix $\g\in I^n_k$ with $\g^1=\void$. We first show the following identity for $A'\in\NN'$,
\be\label{Ups_Step_1}
\Sl_\g A'=\j_{0,0}\Sl_\g\Ups A'.
\ee
Returning to the definition of $\Ups$ in the previous step and using the notation there we have $A'_j=\j_{0,k} A_j$ so that by (iii) of Proposition \ref{prop_Sltfc},
\[
 \Sl_\g A'_j=j_{0,0} \Sl_\g A_j.
\]
By the continuity of $\Sl_\g$ and $\j_{0,0}$ we can pass to the limit (recall that $A'_j \to A'$  and $A_j=\Ups_j A'\to \Ups A'$ by the previous step) and  obtain~\eqref{Ups_Step_1}.\medskip


Let now  $A\in\NN$. Using again (iii) of Proposition \ref{prop_Sltfc} and applying~\eqref{Ups_Step_1} with $A'=\j_{0,k}A$, we compute
\[
\j_{0,0} \Sl_\g A=\Sl_\g\j_{0,k}A=\j_{0,0} \Sl_\g \Ups \j_{0,k}A.
\]
Since $\j_{0,0}$ is one-to-one by Theorem~\ref{thm_MM00MM0}, we deduce that $\Sl_\g A=\Sl_\g \Ups \j_{0,k}A$.
Now if $|\gamma|=k$ but $\gamma^1\neq \void$ then $\Sl_\gamma A=0$ since by assumption  $A\in \NN=\NN^G_{k,(0,k)}(\R^n)$. Similarly, since $\Ups\j_{0,k} A\in \NN$, we also have $\Sl_\gamma \Ups\j_{0,k} A=0$ in this case. Hence, $\Ups\j_{0,k}A$ and $A$ have the same $0$-slices and it follows from Theorem~\ref{coro_defWhite2} that these chains are equal. We conclude that $\Ups\circ\j_{0,k}=\Id$ on $\NN$ and with Step~1, $\Ups$ is an isomorphism from $\NN'$ onto $\NN$ with inverse $\j_{0,k}$.
\end{proof}

\section{Proof of the main results and a counterexample for norm preserving decompositions}
\label{Smain}

We obtain Theorem~\ref{thm_main} as a corollary of the analogous result for rectifiable tensor chains (Theorem~\ref{thm_main2} below). Let us first define this notion. 
\begin{definition} Let $A'\in\MM^G_{k_1,k_2}(\R^n)$, we say that:
\begin{enumerate}[(1)]
\item $A'$ is rectifiable if $A'=A'\restr\Sigma$ for some $k$-rectifiable set $\Sigma\sub\R^n$.
\item $A'$ is $(k_1,k_2)$-rectifiable if $A'=A'\restr (\Sigma^1\t\Sigma^2)$ for some $k_l$-rectifiable subsets $\Sigma^l\sub\R^{n_l}$, $l=1,2$.
\end{enumerate}
\end{definition}

We first consider the case $k_1=0$.
\begin{proposition}\label{prop_main_0k}
Let $A'\in\NN^G_{0,k}(\R^n)$ be a rectifiable tensor chain, then there exists $S\sub\X_\a$, at most countable, such that $\i A'=(\i A')\restr S$.\\
More precisely, there exist sequences $y^1_j\in\X_\a$ and $A^2_j\in\NN^G_k(\X_\ova)$ with $\sum\N(A_j^2)<\oo$ such that (recall \eqref{wedge})
\be
\label{decomp_A'}
A'=\sum\lb y^1_j \rb\we A_j^2.
\ee
\end{proposition}

\begin{proof}
Let $A'\in\NN^G_{0,k}(\R^n)$ be rectifiable and let $\Sigma\sub\R^n$ be a rectifiable set such that $A'=A'\restr\Sigma$. Using the isomorphism $\Ups$ of Theorem~\ref{thm_Ups}, we set $A:=\Ups A'\in\NN^G_{k,(0,k)}(\R^n)$. Since the inverse of $\Ups$ (i.e. the restriction of $\j_{0,k}$ to $\NN_{k,(0,k)}^G(\R^n)$) commutes with restrictions (see~(ii) of Proposition \ref{prop_restr}), we also have $A\restr\Sigma=\Ups (A'\restr\Sigma)$ and $A$ is rectifiable. Besides, by Corollary~\ref{coro_splitting_iff_jk'1k'2=0},
\be\label{prop_main_0k_0}
\Sl_\g A=0\qquad\text{for }\g\in I^n_k\text{ with }(|\g^1|,|\g^2|)\ne(0,k).\bigskip 
\ee

\noindent
\textit{Step 1.}
Let $i\in\a$ and let $\g\in I^n_k$ such that $i\in\g$. By Proposition~\ref{prop_MWSl}, for almost every $s\in \R$,  and almost every $x\in \X_\g$ with $x_i=s$, there holds
\[
A\cap \X_{\ov\g}( x) = (A\cap\X_{\ov i}(s e_i)) \cap\X_{\ov\g}(x).
\]
Since $i\in\a\cap\g$, we have $\g^1\ne \void$ and we deduce from~\eqref{prop_main_0k_0} that  $A\cap \X_{\ov\g}( x)=0$ for almost every $x\in\X_\g$. Hence, for almost every $s\in\R$ and for $i\in \g\in I^n_k$, there holds,
\[
(A\cap\X_{\ov i}(s e_i)) \cap \X_{\ov\g}(x) = 0.
\]
In other words, almost all the coordinate $0$-slices of $A\cap\X_{\ov i}(s e_i)$ vanish. We deduce from Theorem~\ref{coro_defWhite2} that $A\cap\X_{\ov i}(s e_i)=0$ for almost every $s\in\R$.\bigskip 

\noindent
\textit{Step 2.}
Next, by Theorem~\ref{thm_decomp}, there exists a sequence of set-indecomposable (rectifiable and normal) chains $A_j$ such that $A=\sum A_j$ and such that  for each $j$,  $A_j=A\restr S_j$ for some Borel set $S_j\sub\R^n$.

Let us fix $j\ge 1$. By Proposition~\ref{prop_Slrestr} (commutation of slicing and restriction), $A_j$ also satisfies~\eqref{prop_main_0k_0}. As a consequence, the analysis of Step~1 applies to $A_j$ and 
\be\label{prop_main_0k_1}
A_j\cap\X_{\ov i}(s e_i)=0\quad\text{for }i\in\a\text{ and almost every }s\in\R.
\ee
Let us apply Proposition~\ref{prop_divformula} to $A_j$. For $s\in \R$, we denote $H(s):=\{x\in\R^n: x_1>s\}$, then for almost every $s\in\R$, $A_j\restr H(s)\in \NN^G_k(\R^n)$ and
\[
\pt(A_j\restr H(s))=(\pt A_j)\restr H(s)+A_j\cap \X_{\ov1}(se_1)\st{\eqref{prop_main_0k_1}}= (\pt A_j)\restr H(s).
\]
Denoting $A_j(s):=A_j\restr H(s)$,  the chain $A_j-A_j(s)= A_j\restr (\R^n\backslash H(s))$ satisfies
\[
\pt\lt(A_j\restr \R^n\sm H(s)\rt) = \pt A_j -\pt(A_j(s))=(\pt A_j)\restr\lt(\R^n\sm H(s)\rt).
\]
It follows that 
\begin{align*}
\M(\pt A_j(s))+\M(\pt[A_j-A_j(s)]) 
&= \M( (\pt A_j)\restr H(s))+ \M( (\pt A_j)\restr [\R^n\sm H(s)] )\\
&= \mu_{\pt A_j}(H(s))+ \mu_{\pt A_j}(\R^n\sm H(s))\\
&=\mu_{\pt A_j}(\R^n) =\M(\pt A_j).
\end{align*}
We also obviously have $\M(A_j(s))+\M(A_j-A_j(s)) =\M(A_j)$ and therefore 
\[
\lt(A_j(s),\,A_j-A_j(s),0,0,\dots\rt)
\]
is a set-decomposition of $A_j$. But by construction, $A_j$ is set-indecomposable and we conclude that for almost every $s\in\R$, either $A_j=A_j(s)=A_j\restr H(s)$ or $A_j=A_j\restr [\R^n \sm H(s)]$. Proceeding by dichotomy, we obtain that there exists some $s_j\in \R$ such that
\[
A_j=A_j\restr \lt\{x\in\R^n:x_1=s_j\rt\}.
\] 
We can perform the same reasoning with the direction $e_1$ replaced by $e_i$ for $i\in\a$. Taking the intersections, we conclude that there exists $y_j^1\in\X_\a$ such that 
\[
A_j=A_j\restr \lt\{x\in\R^n:x^1=y_j^1\rt\}.
\]     
Eventually, summing over $j$, we deduce $A=A\restr (S+\X_\a)$ with $S=\{y_j^1:j\ge 1\}$. Now, applying the operators $\j_{0,k}$ and $\i$, we get $\i A'=(\i A')\restr S$. This proves the first part of the proposition.
\bigskip 

\noindent
\textit{Step 3.} Let us establish~\eqref{decomp_A'}.  For $j\ge1$ and with the notation of {\it Step 2}, we denote $A'_j:=\j_{0,k}A_j$ and  set $A_j^2:=\chi(\i A'_j)$ so that $A'_j=\lb y^1_j \rb\we A_j^2$. Indeed, this formula holds true for $A'_j$ polyhedral $(0,k)$-chain supported on the affine subspace $y^1_j+\X_\ova$ and we deduce the general case by density using the estimate of Theorem~\ref{thm_chi}(ii). Summing over $j$ we obtain~\eqref{decomp_A'} since $A'=\j_{0,k} A =\sum \j_{0,k} A_j$. Moreover,  since $A_j$ is normal and rectifiable, $A'_j$ is also normal and rectifiable. From the fact that $A'_j=\lb y^1_j \rb\we A_j^2$ we conclude that each $A_j^2$ is normal and rectifiable. Moreover $\N(A^2_j)=\N(A'_j)$ so that
\[
 \sum \N(A_j^2)=\sum \Nw(A'_j)\le \sum \Nt(A_j)\stackrel{\eqref{MtA<=CMA}}{\le} C \sum \N(A_j)=C \N(A)<\oo.
\]
\end{proof}
The counterpart of Theorem~\ref{thm_main} for tensor chains states as follows.
\begin{theorem}\label{thm_main2}
Let $A'\in\NN^G_{k_1,k_2}(\R^n)$, The following are equivalent.
\begin{enumerate}[(i)]
\item $A'$ is rectifiable.
\item $A'$ is $(k_1,k_2)$-rectifiable.
\end{enumerate}
\end{theorem}

\begin{proof}
Since we clearly have $(ii) \Rightarrow (i)$, we only need to prove $(i)\Rightarrow (ii)$. Let $A'\in\NN^G_{k_1,k_2}(\R^n)$ be rectifiable and let us set $A'':=\i A'\in \FF_{k_1}(\X_\a,\FF^G_{k_2}(\X_\ova))$. We use the rectifiable slices theorem, Theorem~\ref{thm_White} to show that $A''$ is rectifiable. Let $\g\sub\a$ with $|\g|=k_1$, we have, by Proposition~\ref{prop_Sltfc}(iv), for a.e. $x\in \X_\g$,
\[
\Sl_\g^x A''=\Sl_\g^x \i A' = \i \Sl_\g^x A'.
\]
Since $A'$ is normal and rectifiable, for almost every $x\in\X_\g$, $\Sl_\g^x A'\in\NN^G_{0,k_2}(\X_{\ov\g})$ is normal and rectifiable. By Proposition~\ref{prop_main_0k}  there exists a countable set $S_x\sub\X_{\a\sm\g}$ such that $\Sl_\g^xA''=(\Sl_\g^xA'')\restr S_x$, this means exactly that $\Sl_\g^xA''$ is $0$-rectifiable for every $\g\sub\a$ with $|\g|=k_1$ and almost every $x\in\X_\g$. We conclude from Theorem~\ref{thm_White} that  $A''$ is rectifiable. Consequently, there exists a $k_1$-rectifiable set $\Sigma^1\sub\R^{n_1}$ such that,
\[
A''=A''\restr\Sigma^1.
\] 
By Proposition~\ref{prop_restr}(iii) and injectivity of $\i$ this is equivalent to 
\[
A'=A'\restr\Sigma^1\t\R^{n_2}.
\]
Eventually exchanging the roles of $\a$ and $\ova$, there exists a $k_2$-rectifiable set $\Sigma^2\subset \R^{n_2}$ such that $A'=A'\restr \R^{n_1}\t\Sigma^2$. By successive restrictions (Proposition~\ref{prop_restr}(iv)), we have,
\begin{align*}
A'&=A'\restr\Sigma^1\t\R^{n_2} =(A'\restr\Sigma^1\t\R^{n_2})\restr \R^{n_1}\t\Sigma^2\\
&=A'\restr (\Sigma^1\t\R^{n_2}  \cap  \R^{n_1}\t\Sigma^2)= A'\restr\Sigma^1\t\Sigma^2.
\end{align*}
This proves that $A'$ is $(k_1,k_2)$-rectifiable.
\end{proof}

We can now prove the main result. Taking into account the characterization of the splitting property of Proposition~\ref{prop_split_and_Sl}, it restates as follows.
\begin{theorem}[Reformulation of Theorem~\ref{thm_main}]
\label{thm_main_bis}
Let $k_1,k_2\ge0$ and $k:=k_1+k_2$. The following statements are equivalent for any rectifiable  flat chain $A\in \NN^G_k(\R^n)$.
\begin{enumerate}[(i)]
\item The measure $T\mu_A$ is $(k_1,k_2)$-split.
\item $A\in\NN^G_{k,(k_1,k_2)}(\R^n)$, \textit{i.e.} $\j_{k'_1,k'_2}A=0$ for $(k'_1,k'_2)\in D_k\sm\{(k_1,k_2)\}$.
\item $A$ is $(k_1,k_2)$-rectifiable.
\end{enumerate}
\end{theorem}

\begin{proof}
We have already established that (i)$\iff$(ii) in Corollary~\ref{coro_splitting_iff_jk'1k'2=0}. This equivalence is true even if $A$ not a normal chain. The implication (iii)$\implies$(i) is obvious (see~\eqref{obvious}) and again, we do not need $A$ to be normal. 

To complete the proof of the theorem, we have to show that (ii)$\implies$(iii).\\
Let $A\in\NN^G_{k,(k_1,k_2)}(\R^n)$ rectifiable and let $\Sigma\sub\R^n$ be $k$-rectifiable and such that $A\restr\Sigma=A$. We define $A'=\j_{k_1,k_2} A$. By Proposition~\ref{prop_restr}(iii), $A'\restr\Sigma=A'$ and by Proposition~\ref{prop_pt_and_j} and our assumption~(ii), $\pt_1 A'=\j_{k_1-1,k_2}\pt A$, $\pt_2 A'=\j_{k_1,k_2-1}\pt A$ so that $A'$ is a normal rectifiable $(k_1,k_2)$-chain. By Theorem~\ref{thm_main2} there exist $k_l$-rectifiable subsets $\Sigma^l\sub\R^{n_l}$, $l=1,2$ such that $A'\restr (\Sigma^1\t\Sigma^2)=A'$, that is
\[
\j_{k_1,k_2}A=(\j_{k_1,k_2}A)\restr(\Sigma^1\t\Sigma^2)=\j_{k_1,k_2}[A\restr(\Sigma^1\t\Sigma^2)].
\] 
Eventually, by assumption, $\j_{k'_1,k'_2}A=0$ for $(k'_1,k'_2)\ne(k_1,k_2)$ and we deduce that $\j A=\j(A\restr(\Sigma^1\t\Sigma^2))$. Since by Theorem~\ref{thm_j_121}, $\j$ is one-to-one we conclude that $A=A\restr  (\Sigma^1\t\Sigma^2)$ so that $A$ is $(k_1,k_2)$-rectifiable.
\end{proof}

\subsection{Decomposition of $(k_1,k_2)$-rectifiable chains in tensor products. A counterexample}
\label{Ss_ctrex}
In this last subsection we consider the group $G=\Z$ endowed with the usual norm. Starting from a $(0,k)-$rectifiable chain $A\in \NN^\Z_{k}(\R^n)$ and arguing as in the proof of Proposition~\ref{prop_main_0k} (using again that by Theorem~\ref{thm_Ups}, $\j_{0,k}$ is an isomorphism and recalling Remark~\ref{rem_A1A2}) we see that there exist sequences $y_j^1\in\X_\a$ and $A_j^2\in\NN^\Z_k(\X_\ova)$ such that
\be\label{NA=SumNAA}
A=\sum \lb y^1_j\rb\t A_j^2\ \in\FF^\Z_k(\R^n)\qquad\text{ and }\qquad\N(A)=\sum\N(A_j^2).
\ee
In the general case, one may ask whether any $(k_1,k_2)$-rectifiable chain $A\in\NN^\Z_k(\R^n)$ decomposes as $A=\sum A_j^1\t A_j^2$ where the $A_j^l$'s are normal rectifiable $k_l$-chains in $\X_\a$ and $\X_\ova$ respectively such that, 
 \begin{equation}\label{contrexample}
\N(A)=\sum \N(A_j^1\t A_j^2)=\sum \M(A_j^1)\M(A_j^2)+\sum \M(\pt A_j^1)\M(A_j^2)+\sum \M(A_j^1)\M(\pt A_j^2).
\end{equation}
Remark that we do not ask for the decomposition to be a set-decomposition. Remark also that for $A^l\in \MM^{\Z}_{k_l}(\R^{n_l})$, the identity $\M(A^1\times A^2)=\M(A^1)\times \M(A^2)$ (which is used in \eqref{contrexample}) can be readily obtained by considering $A^l$ as integral currents and using the definition of the mass through duality. \\
Let us assume for simplicity that $A$ is a cycle. In this case \eqref{contrexample} implies that $A_j^1$ is a cycle whenever $A_j^2\ne0$ and symmetrically $A_j^2$ is a cycle when $A_j^1\ne0$. We may thus assume that the $A_j^l$'s are cycles so that the identity rewrites as
\be\label{proof_ctrex_1}
\M(A)=\sum\M(A_j^1)\M(A_j^2).
\ee
The following result shows that in general, one cannot expect the existence of such decompositions. 

\begin{proposition}[Counterexample]\label{prop_ctrex} Let $n_1=n_2=2$. For every $\eps>0$, there exists a $(1,1)$-rectifiable cycle $A\in\MM^\Z_2(\R^4)$ with $\M(A)>0$ such that for every decomposition $A=\sum A_j^1\t A_j^2$ with $A_j^1\in\MM^\Z_2(\X_\a)$ and $A_j^2\in\MM^\Z_2(\X_\ova)$ cycles, there holds
\be\label{estim_prop_ctrex}
\sum \M(A_j^1)\M(A_j^2)\,\ge\lt(\dfrac43-\eps\rt)\M(A).
\ee
\end{proposition}
\begin{proof}
By assumption, $n_1=n_2=2$ so that $\dim \X_\a=\dim\X_\ova=2$. Let us denote $S=0_{\X_\a}$ and $N=e_1$ and let $\ell>1$. We consider three polyhedral $1$-chains in $\X_\a$ denoted $C_i$ ($i=0,1,2$). We assume that the $C_i$'s have multiplicity $1$ and are supported by three broken lines, each of which of length $\ell$ and with end points $S$ and $N$. We assume that the supports of the $C_i$'s only  intersect at $S$ and $N$ and we orient the edges so that $\pt C_i=\lb N\rb-\lb S \rb$. Eventually, let $L:\X_\a\to\X_\ova$ be the linear isometry defined by $L(e_j)=e_{j+2}$ for $j=1,2$. We set $C'_i:=L\pf C_i$ for $i=0,1,2$. We define:
\[
A:=C_0\t(C'_0-C'_1)+C_1\t(C'_1-C'_2)+C_2\t(C'_2-C'_0)\ \in\ \FF^\Z_{1,1}(\R^4).
\]
This definition has some symmetry. Indeed, $A$ rewrites as
\[
A=(C_2-C_1)\t C'_2+(C_0-C_2)\t C'_0+(C_1-C_0)\t C'_1.
\]
Obviously $A$ is a $(1,1)$-rectifiable chain with integer coefficients in $\R^4$ and $\M(A)=6\ell^2$. We compute, using the first formula for $\pt_1$ and the second formula for $\pt_2$,
\begin{align*}
\pt A=&\, \lt(\lb N\rb-\lb S \rb\rt)\t(C'_0-C'_1+C'_1-C'_2+C'_2-C'_0)\\ 
&\,-(C_2-C_1+C_0-C_2+C_1-C_0)\t\lt(\lb L(N)\rb-\lb L(S) \rb\rt)=0.
\end{align*}
In conclusion, $A$ is a $(1,1)$-rectifiable and normal cycle.\smallskip

Now, let us consider two sequences of cycles $A_j^1\in\MM^\Z_2(\X_\a)$, $A^2_j\in\MM^\Z_2(\X_\ova)$ such that $A=\sum A^1_j\t A^2_j$ and the mass identity~\eqref{proof_ctrex_1} holds. Let $\gamma=(1,3)$, we have almost every $x=se_1+te_3 \in\X_\g$,
\be\label{prf_ctrx0}
\Sl_\g^x A= \sum\Sl_1^{se_1}(A^1_j)\t\Sl_3^{te_3}(A^2_j) 
\ee
In order to obtain a lower bound on $\sum_j \M(A_j^1)\M(A_j^2)$ we estimate the masses of the slices $\Sl_1^{se_1}(A^1_j)$, $\Sl_3^{te_3}(A^2_j)$ from below and use the following bounds form Proposition~\ref{prop_Sltfc}(ii).
\be\label{prf_ctrx1}
\M(A_j^1)\ge\int_0^1\M(\Sl_1^{se_1}(A^1_j))\, ds,\qquad\M(A_j^2)\ge\int_0^1\M(\Sl_3^{te_3}(A^2_j))\, dt\qquad\text{for }j\ge1.
\ee
Let us fix $0<s,t<1$ and let us denote $se_2+c_ie_2:=C_i\cap(se_1+\R e_2)$, $te_3+c'_ie_4:=C'_i\cap(te_2+\R e_4)$. We have for every $s,t\in(0,1)$,
\be\label{prf_ctrx15}
\Sl_\g^x A=
\lb c_0e_2\rb\t(\lb c'_0e_4\rb-\lb c'_1e_4\rb)+\lb c_1e_2\rb\t(\lb c'_1e_4\rb-\lb c'_2e_4\rb)+\lb c_2e_2\rb\t(\lb c'_2e_4\rb-\lb c'_0e_4\rb).
\ee
For $j\ge1$ and almost every $0<s<1$,  $\Sl_1^{se_1}(A^1_j)$ is a sum of $0$-cells $\sum q_{j,a}\lb ae_2\rb$ with $a\in\R$, $q_{j,a}\in\Z$. We single out the points $a=c_0,c_1,c_2$ and write 
\[
A^1_s:=\Sl_1^{se_1}(A^1_j)=m_{j,0}\lb c_0e_2\rb+m_{j,1}\lb c_1e_2\rb+m_{j,2}\lb c_2e_2\rb+\sum_{a\not\in\{c_0,c_1,c_2\}}q_{j,a}\lb ae_2\rb.
\]
Since $A^1_j$ is a cycle, there holds for $j\ge 1$,
\be\label{prf_ctrx2}
m_{j,0}+m_{j,1}+m_{j,2}+\sum_{a\not\in\{c_0,c_1,c_2\}} q_{j,a}=0.
\ee
Similarly, for almost every $0<t<1$ we write 
\[
A^2_t:=\Sl_3^{te_3}(A^2_j)=m'_{j,0}\lb c'_0e_4\rb+m'_{j,1}\lb c'_1e_4\rb+m'_{j,2}\lb c'_2e_4\rb+\sum_{a'\not\in\{c'_0,c'_1,c'_2\}}q'_{j,a'}\lb a'e_4\rb,
\]
and there holds for $j\ge 1$,
\be\label{prf_ctrx3}
m'_{j,0}+m'_{j,1}+m'_{j,2}+\sum_{a'\not\in\{c'_0,c'_1,c'_2\}} q_{j,a'}=0.
\ee
By identification,  the identity~\eqref{prf_ctrx0} writes as 
\be\label{prf_ctrx4}
\begin{array}{rl}
\ds\sum_j m_{j,0}m'_{j,0}=\sum_j m_{j,1}m'_{j,1}=\sum_j m_{j,2}m'_{j,2}&=1,\smallskip\\
\ds\sum_j m_{j,0}m'_{j,1}=\sum_j m_{j,1}m'_{j,2}=\sum_j m_{j,2}m'_{j,0}&=-1,
\end{array}
\ee
and 
\begin{multline}\label{prf_ctrx5}
\sum_jm_{j,i}q'_{j,a'}=\sum_jq_{j,a}m'_{j,i}=\sum_jq_{j,a}q'_{j,a'}=0\\
\text{for }i=0,1,2,\text{ for }a\in\R\sm\{c_0,c_1,c_2\}\text{ and for }a'\in\R\sm\{c'_0,c'_1,c'_2\}.
\end{multline}
Next, we define $\wt m_{j,0}:=m_{j,0}+\sum_{a\not\in\{c_0,c_1,c_2\}}q_{j,a}$ and set 
\[
\wt A^1_j(s):=\wt m_{j,0}\lb c_0e_2\rb+m_{j,1}\lb c_1e_2\rb+m_{j,2}\lb c_2e_2\rb.
\]
Similarly, setting $\wt m'_{j,0}:=m'_{j,0}+\sum_{a\not\in\{c'_0,c'_1,c'_2\}}q'_{j,a}$,
\[
\wt A^2_j(t):=\wt m'_{j,0}\lb c'_0e_4\rb+m'_{j,1}\lb c'_1e_4\rb+m'_{j,2}\lb c'_2e_4\rb.
\]
We have by triangular inequality,
\be\label{prf_ctrx6}
\M(\wt A^1_j(s))\le\M(\Sl_1^{se_1}),\qquad\M(\wt A^2_j(t))\le\M(\Sl_1^{te_3})\qquad\text{for almost every }s,t\in(0,1)
\ee
and taking into account~\eqref{prf_ctrx2}--\eqref{prf_ctrx5} there hold
\be\label{prf_ctrx7}
\wt m_{j,0}+m_{j,1}+m_{j,2}=0,\qquad\qquad\wt m'_{j,0}+m'_{j,1}+m'_{j,2} = 0,
\ee
and 
\be\label{prf_ctrx8}
\sum \wt A^1_j(s)\t \wt A^2_j(t)=\Sl_\g^x A.
\ee
We see from~\eqref{prf_ctrx7} that $\M\lt(\wt A^1_j(s)\rt)=|\wt m_{j,0}|+|m_{j,1}|+|m_{j,2}|$ and $\M\lt(\wt A^2_j(t)\rt)=|\wt m'_{j,0}|+|m'_{j,1}|+|m'_{j,2}|$ are even integers, hence
\be\label{prf_ctrx9}
\sum_j\M\lt(\wt A^1_j(s))\rt)\M\lt(\wt A^2_j(t)\rt)\,\in \, \{0,4,8,12,\dots\}.
\ee
But from~\eqref{prf_ctrx8},~\eqref{prf_ctrx15} and the triangle inequality we also have
\[
6=\M\lt(\Sl_\g^x A\rt)\le\sum_j\M\lt(\wt A^1_j(s)\t\wt A^2_j(t)\rt)\le\sum_j\M\lt(\wt A^1_j(s)\rt)\M\lt(\wt A^2_j(t)\rt).
\]
Together with~\eqref{prf_ctrx9} we get for almost every $s,t\in(0,1)$,
\[
8\le \sum_j\M\lt(\wt A^1_j(s)\rt)\M\lt(\wt A^2_j(t)\rt).
\]
Eventually, integrating on $\{(s,t):0<s,t<1\}$ and using~\eqref{prf_ctrx1} and~\eqref{prf_ctrx6} we conclude that 
\[
\sum \M(A_j^1)\M(A_j^2)\ge 8.
\] 
Since $\M(A)=6\ell^2$, given $\eps>0$, we can choose $\ell>1$ such that $\sum \M(A_j^1)\M(A_j^2)\ge8\ge 8\ell^2-6\eps=(4/3-\eps)\M(A)$. This proves the proposition.
\end{proof}

\begin{remark}\label{rem_ctrex_1}~

\noindent
(1) If $k_1=1$, we have seen that  decompositions satisfying~\eqref{NA=SumNAA} do exist. In the other extreme case $k_1=n_1$, the only $(n_1,k_2)$-rectifiable cycle is $0$. In the other cases $1\le k_1\le n_1-1$, $1\le k_2\le n_2-1$, we obtain  counterexamples as above by substituting for the $C_i$'s three multiplicity 1 smooth $k_1$-chains close to a unit $k_1$-disk whose boundary is a unit $(k_1-1)$-sphere  and whose supports only meet on their common frontier. We substitute similarly for the $C'_i$'s three smooth $k_2$-chains with the corresponding properties. \medskip

\noindent
(2) One may wonder whether there exists a universal constant $\kappa\ge0$ possibly depending on $k_1,k_2,n_1,n_2$ such that for every $(k_1,k_2)$-rectifiable cycle $A\in\MM^\Z_k(\R^n)$,
\be\label{infimum_decomp}
\inf \lt\{\sum \M(A^1_j)\M(A^2_j) : A_j^l\ \text{rectifiable $1$-cycles s.t. } A=\sum A^1_j\t A^2_j\rt\}
\le\kappa\M(A).
\ee
From the example considered above we see that for $1\le k_1\le n_1-1$, $1\le k_2\le n_2-1$, we have $\kappa\ge4/3$. This lower bound is not optimal. Indeed, modifying the construction of the proof of the proposition by considering $N>3$ non-intersecting $1$-chains $C_i$ instead of three and setting $C'_i:=L\pf C_i$ for $i=0,1,\dots,N-1$ (where $L:\X_\a\to\X_\ova$ is still the linear isometry defined by $L(e_j)=e_{j+2}$ for $j=1,2$) and then 
\[
A:=\sum_{i=0}^{N-1} C_i\t\lt(C'_i-C'_{(i+1)(\text{mod }N)}\rt),
\]
we obtain a $(1,1)$-rectifiable cycle such that $\M(A)=2N\ell^2$. Reasoning as in the proof of the proposition, we get that the left hand side of~\eqref{infimum_decomp} is bounded from below by $4(N-1)$ (details are left to the reader). This leads to $\kappa\ge2$. However, we do not believe that~\eqref{infimum_decomp} holds true in general for $\kappa=2$ and we do not even know whether there exist $\kappa<\oo$ such that~\eqref{infimum_decomp} holds true. The question remains open.
\end{remark}

\subsection*{Acknowledments}
The authors thank Antoine Julia and Filippo Santambroggio for suggesting the counterexample~{5} of Section~\ref{Sexples}. M. Goldman is partially supported by the ANR SHAPO.
B. Merlet is partially supported by the INRIA team RAPSODI and the Labex CEMPI (ANR-11-LABX-0007-01).

\bibliographystyle{alpha}
\bibliography{BibStripes}

\end{document}